\documentclass[a4paper,11pt,english]{smfart}
\usepackage{amsfonts}
\usepackage{amsmath} 
\usepackage{latexsym}
\usepackage{array}
\usepackage{amssymb}
\usepackage{smfthm}
\newcommand{\R}{ \mathbb{R} }
\newcommand{\W}{ \mathcal{W} }

\newcommand{\eps}{\varepsilon}
\newcommand{\e}{ \text{e} }

\newcommand{\C}{ \mathbb{C} }
\newcommand{\Z}{ \mathbb{Z} }
\newcommand{\N}{ \mathbb{N} }

\newcommand{\E}{ \mathbb{E} }
\renewcommand{\H}{ \mathcal{H} }

\newcommand{\dis}{\displaystyle}

\newcommand{\om}{ \omega }
\newcommand{\ov}{ \overline }
\renewcommand{\a}{ \alpha }

\newcommand{\s}{ \sigma }

\renewcommand{\phi}{ \varphi }
\renewcommand{\L}{ \mathcal{L} }

\newcommand{\<}{ \langle }
\renewcommand{\>}{ \rangle }
\numberwithin{equation}{section}
\author{ Nicolas Burq}
\address{Laboratoire de Math\'ematiques, B\^at. 425,
Universit\'e Paris Sud, 91405 Orsay Cedex, France}
\email{nicolas.burq@math.u-psud.fr}
\author{ Laurent Thomann }
\address{Laboratoire de Math\'ematiques J. Leray, Universit\'e de Nantes, UMR CNRS 6629\\
2, rue de la Houssini\`ere,
44322 Nantes Cedex 03, France}
\email{laurent.thomann@univ-nantes.fr}
\author{ Nikolay Tzvetkov}
\address{University of Cergy-Pontoise, UMR CNRS 8088, Cergy-Pontoise, F-95000}
\email{nikolay.tzvetkov@u-cergy.fr}
\title[Long time dynamics for the 1d NLS]
{Long time dynamics for the one dimensional non linear Schr\"odinger equation}
\begin{document}
\begin{abstract}
In this article, we first present the construction of Gibbs measures associated to nonlinear Schr\"odinger equations with 
harmonic potential. Then we show that the corresponding Cauchy problem is globally
well-posed for rough initial conditions 
in a statistical set (the support of the measures). Finally, we prove that the
Gibbs measures are indeed invariant by the flow of the equation.
As a byproduct of our analysis, we give a global well-posedness and scattering
result for the $L^2$ critical and super-critical NLS (without harmonic potential).
\end{abstract}
\thanks{The authors were supported in part by the grant ANR-07-BLAN-0250.}
\date{}
\maketitle
%\tableofcontents
\section{Introduction}
The purpose of this work is twofold. First we construct Gibbs measures and prove their 
invariance by the flow of the nonlinear (focusing and defocusing) Schr\"odinger equations 
(defined in a strong sense) in the presence of a harmonic potential. In the construction of these measures, most of the 
difficulties appear for the focusing case (for which case our results are only true for 
the cubic non linearity while in the defocusing case we have no restriction on the size of the non linearity). 
The non linear harmonic oscillator appears as a model in the context of
Bose-Einstein condensates and our result gives some insights concerning the long time dynamics of 
these models. The second purpose of this work is to prove global well-posedness  for the $L^2$ 
critical and super-critical nonlinear Schr\"odinger equation (NLS) on $\R$, 
with or without harmonic potential, for data of low regularity. Furthermore, we also obtain 
scattering when there is no harmonic potential. 
Such kind of result seems to be out of reach of the present critical regularity 
deterministic methods. 
%%%%%%%%%%%%%%%%%%%%%%%%%%%%%%%%%%%%%%%%%%%%%%%%%%%%%%%%%%%%%%%%%%%%%%%%%%%%%%%%%%%%%%%%%%%%
\subsection{The NLS with harmonic potential }
Our analysis here on the NLS with harmonic potential 
enters into the line of research initiated by Lebowitz-Rose-Speer in \cite{LeRoSp} and
aiming to construct Gibbs measures for Hamiltonian PDE's. This program offers
analytic challenges both in the measure construction and in the construction of
a well-defined flow on the support of the measure. Usually the support of the
measure contains low regularity functions and this fact may be seen as one of
the motivations of studying low regularity well-posedness of Hamiltonian
PDE's. The approach of \cite{LeRoSp} has been implemented successfully in
several contexts, see e.g. Bourgain \cite{Bourgain1,Bourgain2}, Zhidkov
\cite{Zhidkov}, Tzvetkov \cite{Tzvetkov1,Tzvetkov2,Tzvetkov3},
Burq-Tzvetkov \cite{BT1}, Oh \cite{Oh1,Oh2}, and the references therein.

A very natural context where one may try to construct Gibbs measures is the
Nonlinear Schr\"odinger equation (NLS) with harmonic potential. Indeed, in
this case the spectrum of the linear problem is discrete and the construction
of \cite{LeRoSp} applies at least at formal level. As we already mentioned 
this context
is natural since the NLS with harmonic potential appears as a model in the Bose-Einstein condensates. 
As we shall see, it turns out that the construction of \cite{LeRoSp} provides a
Gibbs measure supported by functions for which the corresponding Cauchy
problem was not known to be well-posed. In addition the density of the measure
can not be evaluated by applying only deterministic arguments such as the
Sobolev inequality. All these facts present serious obstructions to make
rigourous the Gibbs measure construction.

On the other hand, recent works as \cite{BT1,BT2,Tzvetkov3} showed that by
applying more involved probabilistic techniques in combination with the
existing deterministic technology for studying these problems one may approach
the above difficulties successfully. In particular, in \cite{BT1} an approach
to handle regularities for which the corresponding Cauchy problem is ill-posed
is developed. Our goal here is to show that the NLS with harmonic oscillator
fits well in this approach. In fact, the eigenfunctions of the linear operator
enjoy good estimates which is compensated by the bad separation properties of
the spectrum. Such a situation is particularly well adapted for the approach
of \cite{BT1,BT2}.

We are able to construct Gibbs measures and the corresponding flow for
the cubic focusing and arbitrary defocusing NLS in the presence of
harmonic potential. The analysis turns out to contain several significant new
points with respect to previous works on the subject. Indeed, it seems that it is the first case 
where the construction and analysis of Gibbs measure with a strong flow 
can be performed on a non compact phase space.
Furthermore, taking into account the low regularity of the initial data, to develop a nice local Cauchy 
theory at this level of regularity, one has to obtain somehow a gain in terms of derivatives. In the context 
of wave equations, this gain can be obtained (see~\cite{BT3}) rather easily by first proving a gain at 
the probabilistic level in terms of $L^p$ regularity, and then balancing this gain on the non linearity and 
using that the non homogeneous wave propagator itself gains one derivative with respect to the source term 
regularity. For Schr\"odinger equations, the situation is much less well behaved. Indeed, no such gain of 
regularity occurs for the non homogeneous Schr\"odinger propagator, and the starting point of our analysis 
was precisely that a gain of derivatives occurs at the probabilistic level in terms of $L^p$ regularity. 
However, this gain which would allow to perform the analysis for low power nonlinearities ($k \leq 7$) 
falls well short of what is needed to obtain the full range result ($k < + \infty$) in Theorem~\ref{thm4}. 
As a consequence, our analysis requires a full bi-linear analysis at the probabilistic 
level (see~\eqref{eq.carre}).  Finally, let us mention some previous works on the non linear harmonic Schr\"odinger equation~\eqref{nlsk}.
The (deterministic) Cauchy problem for \eqref{nlsk} was studied by Carles in \cite{Carles}, whereas in \cite{Reika}, Fukuizumi studies the stability of standing waves associated to \eqref{nlsk}.

Let us now describe in more details our results and consider thus the one dimensional cubic Schr\"odinger equation with harmonic potential
\begin{equation}\label{nlsk}
\left\{
\begin{aligned}
&i\partial_t u+\partial_x^2 u -x^{2}u = \kappa_0 |u|^{k-1}u,\quad
(t,x)\in\R\times {\R},
\\
&u(0,x)= f(x),
\end{aligned}
\right.
\end{equation} 
where $k\geq 3$ is an odd integer and where either $\kappa_{0}=1$ 
(defocusing case) or $\kappa_{0}=-1$ (focusing case). The case of cubic nonlinearity, i.e. $k=3$ is the 
one which is most relevant in the context of Bose-Einstein condensates.

We now state our result concerning (\ref{nlsk}). For more detailed results, see Theorem~\ref{prop}
and Theorem~\ref{thm4} below.
\begin{theo}
Consider the $L^2$ Wiener measure on $\mathcal{D}' (\mathbb{R})$, $\mu$, constructed on the 
harmonic oscillator eigenbasis, i.e. $\mu$ is the distribution of the random variable 
$$
\sum_{n=0}^{\infty}\sqrt {\frac{2}{2n+1}}g_{n}(\omega)h_n(x),
$$
where $(h_n)_{n=0}^{\infty}$ are the Hermite functions (see \eqref{formula}) and 
$(g_n)_{n=0}^{\infty}$ is a system of standard independent complex Gaussian random variables. 
Then in the defocusing case, for any order of nonlinearity $k<+\infty$, and in the focusing 
case for the cubic non linearity, the Cauchy problem~\eqref{nlsk} is 
globally well posed for $\mu$-almost every initial data. Furthermore, in both cases, there exists a 
Gibbs measure, absolutely continuous with respect to 
$\mu$, which is invariant by this flow. 
\end{theo}
The equation \eqref{nlsk} is a Hamiltonian PDE with a Hamiltonian $J(u)$ (see \eqref{J} below).
As usual the Gibbs measure is a suitable renormalisation of the formal object $\exp(-J(u))du$. Let us recall that the distribution function of a standard ($0$ mean and $1$ variance)  Gaussian complex random variable  is 
$$ \frac 1 \pi e^{- |z|^2} dL,$$
where $dL$ is the Lebesgue measure on $\mathbb{C}$.

Notice that the results above are not in the "small data" class of results. Indeed, it follows from 
our analysis that the measure $\mu$ is such that for every $p>2$ and every $R>0$ $\mu(u : \|u\|_{L^p}>R)>0$,
i.e. our statistical set contains ``many'' initial data which are arbitrary large in $L^p(\R)$, $p>2$. 
Moreover, we use no smallness argument in any place of the proof.

We conjecture that our results 
hold when $x^{2}$ is replaced with a potential $V\in \mathcal{C}^{\infty}(\R,\R_{+})$, so that
$V(x)\sim x^2$ for $|x|\gg 1$ and $|\partial_{x}^{j}V(x)|\leq C_{j}\<x\>^{2-|j|}$
(in particular in such a situation there exists $C>0$ so that 
$\lambda^{2}_{n}\sim Cn$).

Let us define the Sobolev spaces ${\mathcal H}^s$, 
associated to the harmonic oscillator $-\partial_x^2+x^2$ via the norm
$\|u\|_{{\mathcal H}^s}=\|(-\partial_x^2+x^2)^{s/2}u\|_{L^2}$.
As can easily be seen, for any $s\geq 0$, the Sobolev space of regularity $s$, 
$\mathcal{H}^s( \mathbb{R})$ has zero $\mu$ measure but for every $s<0$ the space 
${\mathcal H}^s$ is of full $\mu$ 
measure. As a consequence, the initial data in our result 
is not covered by the present well-posedness theory for (\ref{nlsk}). 
What is even worse: according to Christ, Colliander, Tao~\cite{CCT} and~\cite[Appendix]{BGT} 
(notice that these results do not apply stricto sensu to the harmonic oscillator, but the proof can 
easily be modified), we know that as soon as $k\geq 7$ the system~\eqref{nlsk} is supercritical and 
there exists no continuous flow on the Sobolev spaces $H^{s}$, for $s\in (0, \frac{1}{2}-\frac{2}{k-1})$. 
As a consequence, even in the local in time analysis we need to appeal to a bi-linear probabilistic argument.
The bi-linear nature of our probabilistic analysis can be seen though the following statement
\begin{equation}\label{eq.carre}
\forall\,\theta<1/2,\,\,
\forall\,t\in\R,\quad
\|(e^{-it H} u)^2\|_{\mathcal{H}^{\theta}} <+ \infty,\quad\mu\,\, {\rm almost\,\, surely}.
\end{equation}
In our actual proof we do not make use of (\ref{eq.carre}) but it was the starting point of our analysis
for large $k$'s. We give the proof of \eqref{eq.carre} in the appendix of this article.
%%%%%%%%%%%%%%%%%%%%%%%%%%%%%%%%%%%%%%%%%%%%%%%%%%%%%%%%%%%%%%%%%%%%%%%%%%%%%%%%%%%%%%%%%%%%%%%%%%%%%
%%%%%%%%%%%%%%%%%%%%%%%%%%%%%%%%%%%%%%%%%%%%%%%%%%%%%%%%%%%%%%%%%%%%%%%%%%%%%%%%%%%%%%%%%%%%%%%%%%%%%%%
\subsection{Global well-posedness and scattering for the ``usual'' 
$L^2$ critical ad super-critical NLS on $\R$}
It turns out that the result described in the previous section has an interesting byproduct. 
Thanks to an explicit transform, we are able to prove a scattering result for the $L^2$ critical 
and super-critical equation 
\begin{equation}\label{nls5}
\left\{
\begin{aligned}
&i\partial_t u+\partial_x^2 u = |u|^{k-1}u,\quad k\geq 5,\quad
(t,x)\in\R\times {\R},\\
&u(0,x)= f(x)
\end{aligned}
\right.
\end{equation}
for $f(x)$ of ``super-critical'' regularity.
\begin{theo}\label{thm5}
For any $0<s<1/2$,
the equation \eqref{nls5} 
has for $\mu$-almost every initial data a unique global solution satisfying
$$u(t,\cdot)-\e^{-it\Delta}f\in {C}\big( \R; {\mathcal H}^{s}(\R) \big)$$
(the uniqueness holds in a space continuously embedded in ${C}\big( \R; {\mathcal H}^{s}(\R) \big)$).
Moreover, the solution scatters in the following sense. There exists $\mu$ a.s. states 
$g_{\pm}\in {\mathcal H}^{s}(\R)$ so that 
\begin{equation*}
\|u(t,\cdot)-\e^{it\Delta}(f+g_{\pm})\|_{{\mathcal H}^{s}(\R)}\longrightarrow 0,\quad 
\text{when}\quad t\longrightarrow \pm\infty.
\end{equation*}
\end{theo}
The result of Theorem~\ref{thm5} is a large data result and as far as we know there is no large data 
scattering results for the problem (\ref{nls5}) for data which are localized (tending to zero at infinity) 
but missing $H^1$, i.e. it seems that the result of Theorem~\ref{thm5} is out of reach of the present deterministic results to get scattering.
We refer to \cite{Kenji} for deterministic scattering results for \eqref{nls5} in
Sobolev spaces, $H^s$, $s\geq 1$. We also refer to \cite{KTV} for an approach for obtaining scattering results for $L^2$ critical problems.

The result of Theorem~\ref{thm5} is based on a transformation which reduces (\ref{nls5}) to a problem
which fits in the scope of applicability of our previous analysis. However (except in the scale invariant case $p=5$), the reduced problem is not
autonomous which makes the arguments more delicate. In particular there is no conserved energy for the
reduced problem. However, we will be able to substitute this lack of conservation law by a monotonicity
property which in turn will lead to the fact that, roughly speaking, the measure of a set can not decrease 
along the flow which is the key of the globalization argument. As a consequence, we are able to carry out the global existence strategy whilst no invariant measure is available (see also Colliander-Oh~\cite{CO1, CO2} for results in this direction).

%%%%%%%%%%%%%%%%%%%%%%%%%%%%%%%%%%%%%%%%%%%%%%%%%%%%%%%%%%%%%%%%%%%%%%%%%%%%%%%%%
%%%%%%%%%%%%%%%%%%%%%%%%%%%%%%%%%%%%%%%%%%%%%%%%%%%%%%%%%%%%%%%%%%%%%%%%%%%%%%%%%
\subsection{Plan of the paper}
In the following section, we present in details the construction of the Gibbs measure and we give a detailed measure invariance statement. In Section~3, we give the proof 
of the approximation property of the Gibbs measure by ``finite dimensional'' measures.
In the next section, we establish a  functional calculus of 
$-\partial_x^2+x^2$, fundamental for the future analysis. The following two sections are devoted to establishing two families of linear 
dispersive estimates, namely the Strichartz and the local smoothing estimates. Here we develop the very
classical deterministic estimates but also the more recent stochastic variants of them. The interplay
between these two families of estimates is at the heart of our approach. In Section~7, we use the estimates 
of the two previous sections together with the functional calculus to develop a local Cauchy  theory.
In Section~8, we present the global arguments, leading to almost sure global well-posedness on the support 
of the measure. Notice that the path followed here has been much clarified with respect to our previous papers, and consequently is much more versatile. In Section~9, 
we prove the measure invariance. Section~10 is devoted to the proof of 
Theorem~\ref{thm5}. Finally, in an appendix, we review some typical properties on the support of $\mu$.
%%%%%%%%%%%%%%%%
%%%%%%%%%%%% 
\subsection{ Acknowledgements} We would like to thank P. G\'erard for pointing out to us the bilinear estimates enjoyed by the Hermite functions, which was the starting point of this work (see Lemma~\ref{patrick}).  We are also indebted to Thomas Duyckaerts 
for suggesting us to use the pseudo-conformal transformation in our analysis. 
%%%%%%%%%%%%%%%%%%%%%%%%%%%%%%%%%%%%%%%%%%%%%%%%%%%%%%%%%%%%%%%%%%%%%%%%%%%%%%%%
\section{Hamiltonian formulation and construction of the Gibbs measure}
Set $H=-\partial_x^2 +x^{2}$.
The operator $H$ has a self-adjoint extension on $L^{2}(\R)$ (still denoted
by $H$) and has eigenfunctions 
$\big(h_{n}\big)_{n\geq 0}$ which form a Hilbertian basis of $L^{2}(\R)$ and satisfy
$
Hh_{n}=\lambda_{n}^{2}h_{n}
$
with $\lambda_{n}=\sqrt{2n+1}$. 
Indeed, $h_{n}$ are given by the formula 
\begin{equation}\label{formula}
h_{n}(x)=(-1)^{n}c_{n}\,\e^{x^{2}/2}\frac{\text{d}^{n}}{\text{d}x^{n}}
\big(\,\e^{-x^{2}}\,\big),\;\;\text{with}\;\;\;\frac1{c_{n}}
=\big(n\,!\big)^{\frac12}\,2^{\frac{n}2}\,\pi^{\frac14}.
\end{equation}
The equation \eqref{nlsk} has the following Hamiltonian
\begin{equation}\label{J}
J(u)=\frac{1}{2}\int_{-\infty}^{\infty}|H^{1/2}u(x)|^2\,
\text{d}x+\frac{\kappa_{0}}{k+1}\int_{-\infty}^{\infty}|u(x)|^{k+1}\, \text{d}x.
\end{equation}
Write
$
u=\sum_{n=0}^{\infty}c_n h_{n}.
$
Then in the coordinates $c=(c_n)$ the Hamiltonian reads
$$
J(c)=
\frac{1}{2}\sum_{n=0}^{\infty}\lambda_{n}^{2}|c_n|^2
+\frac{\kappa_{0}}{k+1}\int_{-\infty}^{\infty}\big|\sum_{n=0}^{\infty}c_n h_{n}(x)\big|^{k+1}\, \text{d}x.
$$
Let us define the complex vector space $E_N$ by $E_N={\rm span}(h_0,h_1,\cdots,h_N)$. 
Then we introduce the spectral projector $\Pi_{N}$ on $E_{N}$ by 
\begin{equation*}
\Pi_{N}\big(\sum_{n=0}^{\infty}c_n h_{n}\big)=\sum_{n=0}^{N}c_n h_{n}\,.
\end{equation*}
Let $\chi\in \mathcal{C}_{0}^{\infty}(-1,1)$,  so that $\chi=1$ on $[-\frac12,\frac12]$. Let $S_{N}$ be the operators
\begin{equation}\label{S_NN}
S_{N}\big(\sum_{n=0}^{\infty}c_n h_{n}\big)=\sum_{n=0}^{\infty}\chi\big(\frac{2n+1}{2N+1}\big)c_n h_{n}
=\chi\big(\frac{H}{2N+1}\big)\big(\sum_{n=0}^{\infty}c_n h_{n}\big)\,.
\end{equation}

It is clear that $\|S_{N}\|_{L^{2}\to L^{2}}= \|\Pi_{N}\|_{L^{2}\to L^{2}}=1$ and we have 
\begin{equation}\label{commut}
S_{N}\,\Pi_{N}=\Pi_{N}\,S_{N}=S_{N},\quad \text{and}\quad S_{N}^{*}=S_{N}.
\end{equation}
The interest of introducing the smooth cut-off $S_{N}$ is its better mapping properties on $L^p$, $p\neq 2$, 
compared to $\Pi_N$ (see Proposition~\ref{prop.cont}). 

Let us now turn to the definition of the Gibbs measure.
Write $c_{n}=a_{n}+ib_{n}$. For $N\geq 1$, consider the probability measures on $\R^{2(N+1)}$ defined by
$$
\text{d}\tilde{\mu}_N=\prod_{n=0}^{N}\frac{ 2 \pi} { \lambda_n^2}e^{-\frac{\lambda^{2}_{n}}{2}(a_n^2+b_n^2)}\text{d}a_n\text{d}b_n,
$$
The measure $\tilde{\mu} _N$ defines a measure on $E_N$ via the map
\begin{equation}\label{transport}
(a_n,b_n)_{n=0}^{N}\longmapsto \sum_{n=0}^{N}(a_n+ib_n)h_n,
\end{equation}
which will still be denoted by $\tilde{\mu}_N$. 
Notice that  $\tilde{\mu}_N$ may be seen as the distribution of the $E_N$ valued random
variable
\begin{equation}\label{def.phin}
\omega\longmapsto \sum_{n=0}^{N}{\frac{\sqrt 2}{\lambda_{n}}}g_{n}(\omega)h_n(x)\equiv
\varphi_{N}(\omega,x),
\end{equation}
where $(g_n)_{n=0}^{N}$ is a system of independent, centered, $L^2$ normalized
complex Gaussians on a probability space $({\Omega, \mathcal{F},{\bf p}})$. \par
In order to study convergence properties of $\varphi_N$ as $N\rightarrow\infty$, we define Sobolev spaces 
associated to $H$.
\begin{defi}
For $1\leq p\leq +\infty$ and $s\in \R$, we define the space ${\mathcal W}^{s,p}(\R)$ 
via the norm
$
\|u\|_{{\mathcal W}^{s,p}(\R)}=\|H^{s/2}u\|_{L^p(\R)}.
$
In the case $p=2$ we write $\W^{s,2}(\R)=\H^{s}(\R)$ and if
$
u=\sum_{n=0}^{\infty}c_n h_{n}
$
we have
$
\|u\|^2_{\H^s}=\sum_{n=0}^{\infty}\lambda_{n}^{2s}|c_n|^2.
$
\end{defi}
For future references, we state the following key property of the spaces ${\mathcal W}^{s,p}$, which is actually a consequence of the fact that $H^{-s}$ is a pseudo differential operator in a suitable class (which ensures its $L^p$ boundedness)
\begin{prop}[\protect{\cite{DzGl09}}]
For any $
1<p<\infty, s\geq 0$, there exists $C>0$ such that 
\begin{equation}\label{weyl-hormander}
\frac 1 C \|u\|_{{\mathcal W}^{s,p}(\R)}\leq \|\langle D_x\rangle^s u\|_{L^p(\R)}+\|\langle x\rangle^s u\|_{L^p(\R)}\leq C  \|u\|_{{\mathcal W}^{s,p}(\R)}.
\end{equation}
\end{prop}
Let $\sigma>0$. Then $(\varphi_N)$ is a Cauchy sequence in $L^2(\Omega;{\mathcal
H}^{-\sigma}(\R))$ 
which defines
\begin{equation}\label{def.phi}
\phi(\om,x)=\sum_{n=0}^{\infty}{\frac{\sqrt 2}{\lambda_{n}}}g_{n}(\omega)h_n(x),
\end{equation}
as the limit of $(\varphi_N)$. Indeed, the map
$
\omega\mapsto\varphi(\omega,x)
$
defines a (Gaussian) measure on ${\mathcal H}^{-\sigma}(\R)$ which defines measure $\mu$.
Notice also that the measure $\mu$ can be decomposed into 
$$\mu= \mu^N \otimes \tilde{\mu}_N$$ where $\mu^N$ is the the distribution of the random variable on $E_N^\perp$
$$
\sum_{n=N+1}^{\infty} \frac{ \sqrt 2}{\lambda_n}g_{n}(\omega)h_n(x).
$$

$\bullet$\;{\bf The defocusing case ($\kappa_{0}=1$) and $k\geq 3$}.
In this case we can define the Gibbs measure $\rho$ by
\begin{equation}\label{drho}
\text{d}{\rho}(u)=\exp\big({-\frac{1}{k+1}\|u\|^{k+1}_{L^{k+1}(\R)}}\big)\text{d}\mu(u).
\end{equation}
We also define its finite dimensional approximations 
\begin{equation}\label{tilda}
\begin{aligned}
\text{d}\tilde{\rho}_{N}(u)&=\exp\big({-\frac{1}{k+1}\|S_Nu\|^{k+1}_{L^{k+1}(\R)}}\big)\text{d}\tilde{\mu} _N(u),\\
\text{d}{\rho}_{N}(u)&=\exp\big({-\frac{1}{k+1}\|S_Nu\|^{k+1}_{L^{k+1}(\R)}}\big)\text{d}\mu(u)= \text{d}\mu^N \otimes \text{d}\tilde{\rho}_N.
\end{aligned}
\end{equation}

\noindent $\bullet$\;{\bf The focusing case ($\kappa_{0}=-1$) and $k=3$}.
Let $\zeta :\R\rightarrow\R$, $\zeta\geq 0$ be a continuous function with compact
support (a cut-off).
Define the measures $\tilde{\rho}_{N}, \rho_N$ as
\begin{equation}\label{tilda_pak}
\begin{aligned}
\text{d}\tilde{\rho}_{N}(u)&=\zeta\big(\|\Pi_{N}u\|_{L^2(\R)}^{2}-\alpha_{N}\big)
e^{\frac{1}{4}\int_{\R}|S_Nu(x)|^4\text{d}x}\text{d}\tilde{\mu} _N(u),\\
\text{d}{\rho}_{N}(u)&=\text{d}\mu^N \otimes \text{d}\tilde{\rho}_N.
\end{aligned}
\end{equation}

We have the following statement defining the Gibbs measure associated to the equations
\eqref{nlsk}.
\begin{theo}\label{prop}
(i) Defocusing case ($\kappa_0 = 1$) and $k<+\infty$.  Let the measure $\rho$ be defined by~\eqref{drho}. 

(ii) Focusing case ($\kappa_{0}=-1$) and $k= 3$ : 
The sequence
\begin{equation}\label{cinm}
G_{N}(u)=\zeta\big(\|\Pi_{N}u\|_{L^2(\R)}^{2}-\alpha_{N}\big)
\e^{\frac{1}{4}\int_{\R}|S_Nu(x)|^4\text{d}x},
\end{equation}
converges in measure, as $N\rightarrow\infty$, with respect to the measure $\mu$.
Denote by $G(u)$ the limit of (\ref{cinm}) as $N\rightarrow\infty$.
Then for every $p\in [1,\infty[$,
$G(u)\in L^{p}(\text{d}\mu(u))$ and we define
$
\text{d}\rho(u)\equiv G(u)\text{d}\mu(u)
$.

In both cases,  the sequence $\text{d}\rho_{N}$ converges weakly to $\text{d}\rho$ and for any Borelian set $A \subset \mathcal{H}^{- \sigma}$, we have
\begin{equation}
\label{limite}
 \lim_{n\rightarrow + \infty} \rho_N(A) = \rho(A).
 \end{equation}
\end{theo}
\noindent 
The result in the defocusing case is quite a direct application of the argument of \cite{AyTz}.
The construction of the measure in the focusing case is much more involved and is inspired by
the work \cite{Tzvetkov3} of the third author on the Benjamin-Ono equation. 
The main difficulty in this construction lies in proving that the weight 
$G(u)$ belongs to $L^1( \text{d}\mu)$. A first candidate for the weight $G(u)$ would have
been \smash{$\exp(\|u\|^4_{L^4( \mathbb{R} )}/4)$}, but then the large
deviation estimates are too weak to ensure the integrability of this weight
with respect to the measure $\text{d} \mu$. A second guess would have been $
\zeta (\|u\|_{L^2}) \exp(\|u\|^4_{L^4( \mathbb{R} )}/4)$, as then the
same large deviation estimates and Gagliardo-Nirenberg inequalities would
ensure this integrability. Unfortunately, on the support of the measure $\mu$,
the $L^2$ norm is almost surely infinite and this choice of weight would lead
to a trivial (vanishing) invariant measure. The renormalized (square of the)
$L^2$ norm provides us with an acceptable substitute to this latter
choice. Let us also observe that if we vary $\zeta$ then we get the support of
$\mu$ (see Proposition~\ref{union} below). Notice also that this choice of weight is reminiscent of 
Bourgain's work~\cite{Bourgain2} 
where a similar renormalization is performed at the level of the equation 
itself rather than the level of the Gibbs measure.

It is now a natural question whether the measure $\rho$ constructed in Theorem~\ref{prop} 
is indeed invariant by a well-defined flow of \eqref{nlsk}. It turns out to be the case as shows 
the following statement.
\begin{theo}\label{thm4} 
Assume that $k=3$ in the focusing case and $3\leq k <+\infty$ in the defocusing case. Then the Cauchy 
problem~\eqref{nlsk} is, for  $\mu$-almost every initial data, globally well posed in a strong sense and the Gibbs 
measure $\rho$ constructed in Theorem~\ref{prop} is invariant under this flow, $\Phi(t)$. 
More precisely,
\begin{itemize}
\item There exists a set $\Sigma$ 
of full $\rho$ measure and $s<\frac12$ (for $k=3$, $s<\frac 13$ can be taken arbitrarily close to $\frac 1 3$ while for $k\geq 5$, $s$ can be taken arbitrarily close to $\frac 12$) so that for every $f\in \Sigma$ the
equation \eqref{nlsk} with 
initial condition $u(0)=f$ has a global solution such that
$
u(t,\cdot)-\e^{-itH}f\,\in\,\mathcal{C}\big( \R; \H^{s}(\R) \big).
$
The solution is unique in the following sense : for every $T>0$ there is a functional space $X_T$ continuously 
embedded in $\mathcal{C}\big([-T,T]; \H^{s}(\R) \big)$ such that the solution is unique in the class
$$
u(t,\cdot)-\e^{-itH}f\,\in X_T.
$$
Moreover, for all $\s>0$ and $t\in \R$
\begin{equation*}
\|u(t,\cdot)\|_{\H^{-\s}(\R)}\leq C\bigl( \Lambda(f,\s)+\ln^{\frac12} \big(1+|t|\big)\bigr),
\end{equation*}
and the constant $\Lambda(f,\s)$ satisfies the bound 
$
\mu\big(f :\Lambda(f,\s)>\lambda \big) \leq C\e^{-c\lambda^{2}}.
$
\item For any $\rho$ measurable set $A\subset \Sigma$, for any $t\in \R$, $\rho(A)=\rho(\Phi(t)(A))$.
\end{itemize}
\end{theo}
Let us remark that the uniqueness statement can also be formulated as the fact that the flow of (\ref{nlsk})
on smooth data (for instance $H^1$) can be extended in a unique fashion a.s. on the support of the 
measure $\rho$. Notice also that in this paper, we had to modify the definition of the finite dimensional approximations measures $\rho_N$ with respect to previous results on the subject (see e.g. \cite{Bourgain1}). Indeed, the lack of continuity of the rough projectors $\Pi_N$ on our resolution spaces forbid the usual approximation results (see e.g.~\cite[Theorem 1.2]{Tzvetkov2}). As a consequence, our new measures enjoy better approximation properties (see~\eqref{limite}), but the invariance properties we have to prove are stronger (see Corollary~\ref{cor.inv}). We believe nevertheless this new approach is more natural.
%%%%%%%%%%%%%%%%%%%%%%%%%%%%%%%%%%%%%%%%%%%%%%%%%%%
%%%%%%%%%%%%%%%%%%%%%%%%%%%%%%%%%%%%%%%%%%%%%%%%%%
\section{Proof of Theorem~\ref{prop}}\label{section2}~
In this section we prove Theorem~\ref{prop}. As we already mentioned, the main issue is 
the construction of the measure for \eqref{nlsk} with $k=3$ in the focusing case. 
\subsection{Preliminaries and construction of the density}
First we recall the following Gaussian bound (Khinchin inequality), 
which is one of the key points in the study of our random series. See e.g. \cite[Lemma 4.2.]{BT2} 
for a proof in a more general setting. Let us
notice that in our particular setting, the random variable being a Gaussian variable of variance 
$\sum _{n\geq 0}|c_{n}|^{2}$, this estimate is also an easy consequence of the growth of the 
$r$'th moments of centered Gaussians (uniform with respect to the variance).
\begin{lemm}\label{Gaussian}
Let $\big(g_{n}(\om)\big)_{n\geq 0}\in \mathcal{N}_{\C}(0,1)$ be independent, complex, 
$L^{2}$- normalized Gaussian random variables.
Then there exists $C>0$ such that for all $r\geq 2$ and $(c_{n})\in l^{2}(\N)$
\begin{equation*}
\|\sum_{n\geq 0}g_{n}(\om)\,c_{n}\|_{L^{r}(\Omega)}\leq C\sqrt{r}\big(\sum _{n\geq 0}|c_{n}|^{2}\big)^{\frac12}.
\end{equation*}
\end{lemm}
\noindent 
We will need the following particular case of the bounds on the eigenfunctions
$(h_{n})$, proved for example by K. Yajima and G. Zhang \cite{YajimaZhang1} (see also H. Koch and D. Tataru~\cite{KochTataru}) .
\begin{lemm}[Dispersive bound for $h_{n}$]\label{disp}
For every $p\geq 4$ there exists $C(p)$ such that for every $n\geq 0$,
$$
\|h_{n}\|_{L^p(\R)}\leq C(p)\lambda_{n}^{-\frac{1}{6}}\,.
$$
\end{lemm}
\noindent As a consequence, we may show the following statement.
\begin{lemm}\label{ld}
Fix $p\in [4,\infty)$ and $s\in [0,1/6)$. Then
\begin{multline}\label{parvo}
\exists C>0,\exists c>0, \forall \lambda\geq 1, \forall N\geq 1,
\\
\mu\big(\,u\in {\mathcal H}^{-\sigma}: \|S_Nu\|_{{\mathcal
W}^{s,p}(\R)}>\lambda\,\big)\leq Ce^{-c\lambda^2}\,.
\end{multline}
Moreover there exists $\beta(s)>0$ such that
\begin{multline}\label{vtoro}
\exists\, C>0,\exists\, c>0,\,\, \forall \lambda\geq 1,\, \forall N\geq N_0\geq 1,
\\
\mu\big(\,u\in {\mathcal H}^{-\sigma}\,:\, \|S_Nu-S_{N_0}u\|_{{\mathcal
W}^{s,p}(\R)}>\lambda\,\big)\leq 
Ce^{-cN_0^{\beta(s)}\lambda^2}\,.
\end{multline}
\end{lemm}
\begin{proof}
We have that
\begin{multline*}
\mu\big(\,u\in {\mathcal H}^{-\sigma}\,:\, \|S_Nu\|_{{\mathcal W}^{s,p}(\R)}>\lambda\,\big)\\
\begin{aligned}
&=
{\bf p}\big(\,\omega\,:\,
\|\sum_{n=0}^{\infty}{\chi\bigl(\frac{2n+1}{2N+1}\bigr)\frac{\sqrt 2}{\lambda_{n}}}g_{n}(\omega)h_n(x)\|_{{\mathcal
W}^{s,p}(\R)}>\lambda\,\big)
\\
&=
{\bf p}\big(\,\omega\,:\,
\|\sum_{n=0}^{\infty}\chi\bigl(\frac{2n+1}{2N+1}\bigr)
{\frac{\sqrt 2}{\lambda^{1-s}_{n}}}g_{n}(\omega)h_n(x)\|_{L^{p}(\R)}>\lambda\,\big).
\end{aligned}
\end{multline*}
Set 
$$
f(\omega,x)\equiv\sum_{n=0}^{\infty}
\chi\bigl(\frac{2n+1}{2N+1}\bigr){\frac{\sqrt 2}{\lambda^{1-s}_{n}}}g_{n}(\omega)h_n(x)\,.
$$
Then for $q\geq p$, using the Minkowski inequality, we get
$$
\|f(\omega,x)\|_{L^q_{\omega}L^p_x}\leq \|f(\omega,x)\|_{L^p_xL^q_\omega}\,.
$$
By Lemma \ref{Gaussian} we get
\begin{eqnarray*}
\|f(\omega,x)\|_{L^q_\omega}&\leq& C\sqrt{q}\big(\sum_{n=0}^{\infty}\chi^{2}_{0}
\bigl(\frac{2n+1}{2N+1}\bigr){\frac{ 2}{\lambda^{2(1-s)}_{n}}}|h_{n}(x)|^2\big)^{1/2}\\
&\leq &C\sqrt{q}\big(\sum_{n=0}^{\infty}{\frac{ 2}{\lambda^{2(1-s)}_{n}}}|h_{n}(x)|^2\big)^{1/2}\,.
\end{eqnarray*}
Since $s<1/6$, using Lemma~\ref{disp} and the triangle inequality, we get
$$
\|f(\omega,x)\|_{L^q_{\omega}L^p_x}\leq C\sqrt{q}\,.
$$
Using Bienaym\'e-Tchebichev inequality, we obtain
$$
{\bf p}\big(\,\omega\,:\, \|f(\omega,x)\|_{L^p_x}>\lambda\,\big)\leq 
(\lambda^{-1}\|f(\omega,x)\|_{L^q_{\omega}L^p_x})^q
\leq
(C\lambda^{-1}\sqrt{q})^{q}\,.
$$
Thus by choosing $q=\delta \lambda^2$, for $\delta$ small enough, we get the bound 
$$
{\bf p}\big(\,\omega\,:\, \|f(\omega,x)\|_{L^p_x}>\lambda\,\big)\leq Ce^{-c\lambda^2}\,.
$$
This in turn yields (\ref{parvo}). The proof of (\ref{vtoro}) is very similar.
Indeed, in this case, we analyze the function
$$
f_{N_0}(\omega,x)\equiv\sum_{n=0}^{\infty} \big( \chi(\frac{2n+1}{2N+1})-
\chi(\frac{2n+1}{2N_{0}+1}) \big) {\frac{ 2}{\lambda^{2(1-s)}_{n}}}g_{n}(\omega)h_n(x),
$$
and we use that there is a negative power of $N_0$ saving in the estimate. 
Namely, there is $\gamma(s)>0$ such that
$$
\|f_{N_0}(\omega,x)\|_{L^q_{\omega}L^p_x}\leq C\sqrt{q}N_0^{-\gamma(s)}\,,
$$
which implies (\ref{vtoro}). This completes the proof of Lemma~\ref{ld}.
\end{proof}

\noindent With the same arguments one can prove the following statement.
\begin{lemm}\label{lem.3.4}
Let $\s>0$, then 
\begin{multline}\label{normehs}
\exists C>0,\exists c>0, \forall \lambda\geq 1,
\mu\big(\,u\in {\mathcal H}^{-\sigma} : \|u\|_{{\mathcal
H}^{-\s}(\R)}>\lambda\,\big)\leq Ce^{-c\lambda^2}\,.
\end{multline}
\end{lemm}
%%%%%%%%%%%%%%%%%%%%%
\begin{lemm}[Gagliardo-Nirenberg inequality associated to $H$]\label{lem.GN}
For any $s\in(0,1/6)$, there exists $p<\infty$ and $\theta<2$ such that
$$
\|u\|_{L^4(\R)}^4\leq C \|u\|_{L^2(\R)}^{4-\theta}\|u\|^{\theta}_{\W^{s,p}(\R)}\,.
$$
\end{lemm}
\begin{proof}
First we prove that for any $s\in(0,1/6)$, there exists $p<\infty$ and $\theta<2$ such that
\begin{equation}\label{gagliardo}
\|u\|_{L^4(\R)}^4\leq C \|u\|_{L^2(\R)}^{4-\theta}\|u\|^{\theta}_{W^{s,p}(\R)}\,,
\end{equation}
where $W^{s,p}$ is the usual Sobolev space.

Fix $s\in (0, 1/6)$ and write 
\begin{equation}\label{i}
\|u\|_{L^4(\R)}^4\leq C \|u\|_{L^2(\R)}^{2}\|u\|^{2}_{L^\infty(\R)}\,.
\end{equation}
Using \cite[Proposition~A.3]{Tao}, we get that there exists $p\gg 1$ and $\kappa>0$ such that
\begin{equation}\label{ii}
\|u\|_{L^\infty(\R)}\leq C\|u\|_{L^p(\R)}^{\kappa}\|u\|^{1-\kappa}_{W^{s,p}(\R)},
\end{equation}
(indeed for large $p$ the derivative loss tends to zero, i.e. we may assume
that it is smaller than $s$).
Finally the H\"older inequality and (\ref{ii}) implies that for any $q>p$ there exists $\alpha>0$
such that
\begin{equation}\label{iii}
\|u\|_{L^p(\R)}\leq C\|u\|_{L^2(\R)}^\alpha \|u\|_{L^q(\R)}^{1-\alpha}
\leq C\|u\|_{L^2(\R)}^\alpha \|u\|^{1-\alpha}_{W^{s,p}(\R)}\,.
\end{equation}
A combination of (\ref{i}), (\ref{ii}) and (\ref{iii}) yields \eqref{gagliardo}.

Finally, to complete the proof of the lemma use that thanks to (\ref{weyl-hormander})
$$
\|u\|_{W^{s,p}(\R)}\leq C\|u\|_{{\mathcal W}^{s,p}(\R)}\,.
$$
This completes the proof of Lemma~\ref{lem.GN}.
\end{proof}
Denote by 
\begin{equation*}
F_{N}(u)=\|\Pi_{N}u\|_{L^2(\R)}^{2}-\alpha_{N}.
\end{equation*}
As in \cite{Tzvetkov3}, we need the following convergence properties of the sequence $\big(F_{N}(u)\big)_{N\geq 0}$.

\begin{lemm}\label{lem.cauchy}
The sequence $(F_{N}(u))$ is a Cauchy sequence in  $L^{2}(\H^{-\s}(\R),\text{d}\mu)$. As a consequence, if we denote by $F(u)$ its limit, the sequence $\big(F_{N}(u)\big)_{N\geq 0}$ converges to $F(u)$ in measure :
\begin{equation*}
\forall\,\eps>0,\;\; \lim_{N\to\infty}\mu\big(\,u\in {\mathcal H}^{-\sigma}\,:\,
\big|F_{N}(u)-F(u)\big|>\eps\,\big)=0.
\end{equation*}
\end{lemm}

\begin{proof}
Let $N>M\geq 0$, then 
\begin{multline*}
\|F_{N}(u)-F_{M}(u)\|^{2}_{L^{2}(\H^{-\s}(\R),\text{d}\mu)}=\\
\begin{aligned}
&=\int_{\Omega}\big| \big( \|\phi_{N}\|_{L^2(\R)}^{2}-\alpha_{N}\big)-\big(\|\phi_{M}\|_{L^2(\R)}^{2}-\alpha_{M} \big) \big|^{2}\text{d}{\bf p}(\om),
\end{aligned}
\end{multline*}
where $\phi_{N}$ is defined in \eqref{def.phin}. By definition of $\alpha_{N}$, we have 
\begin{equation}\label{pphi}
\|\phi_{N}\|_{L^2(\R)}^{2}-\alpha_{N}=\sum_{n=0}^{N}\frac{2}{\lambda_{n}^{2}}(|g_{n}(\om)|^{2}-1),
\end{equation}
and therefore
\begin{equation}\label{eq.fnfm}
\|F_{N}(u)-F_{M}(u)\|^{2}_{L^{2}(\H^{-\s}(\R),\text{d}\mu)}=
\int_{\Omega}\big|\sum_{n=M+1}^{N} 
\frac{2}{\lambda_{n}^{2}}(|g_{n}(\om)|^{2}-1) \big|^{2}\text{d}{\bf p}(\om).
\end{equation}
Now, as the random variables $\big(g_{n}(\om)\big)_{n\geq 0}$ are normalized and independent, 
for all $n_{1}\neq n_{2}$ we have 
\begin{equation*}
\int_{\Omega}\big(|g_{n_{1}}(\om)|^{2}-1\big)\big(|g_{n_{2}}(\om)|^{2}-1\big)\text{d}{\bf p}(\om)=0,
\end{equation*}
therefore from \eqref{eq.fnfm} we deduce 
\begin{equation*}
\|F_{N}(u)-F_{M}(u)\|^{2}_{L^{2}(\H^{-\s}(\R),\text{d}\mu)}=c\sum_{n=M+1}^{N} \frac{1}{\lambda_{n}^{4}}\leq \frac{C}{M+1},
\end{equation*}
as $\lambda_{n}^{2}=2n+1$. This proves the first assertion of the lemma.

By the Tchebychev inequality, $L^{2}$ convergence implies convergence in measure, hence the result. 
This completes the proof of Lemma~\ref{lem.cauchy}.
\end{proof}

\noindent The following result is a large deviation bound for the sequence $\big(F_{N}(u)\big)$.
\begin{lemm}\label{lem.BN}
There exist $C,c>0$ so that for all $N>M\geq 0$ and $\lambda>0$
\begin{equation}
\mu\big(\,u\in {\mathcal H}^{-\sigma}\,:\,
\big|F_{N}(u)-F_{M}(u)\big|>\lambda\, \big)\leq C\e^{-c(M+1)^{\frac12}\lambda}.
\end{equation}
\end{lemm}
\begin{proof}
The result can be viewed as a consequence of a smoothing property of a suitable heat flow, 
but we give here a direct proof. Define the set 
\begin{equation*}
B_{M,N}=\big\{u\in {\mathcal H}^{-\sigma}\,:\,
\big|F_{N}(u)-F_{M}(u)\big|>\lambda\big\}.
\end{equation*}
Then by \eqref{pphi} for $N>M$,
\begin{eqnarray}\label{BN}
\mu(B_{M,N})&=&{\bf p}\big(\,\om \,:\,
\big|\big(\,\|\phi_{N}\|_{L^2(\R)}^{2}-\alpha_{N}\big)-\big(\|\phi_{M}\|_{L^2(\R)}^{2}-\alpha_{M}\big)\big|>\lambda\,\big)\nonumber \\
&=&{\bf p}\big(\,\om \,:\,
\big| \sum_{n=M+1}^{N}\frac{2}{\lambda_{n}^{2}}(|g_{n}(\om)|^{2}-1) \big|>\lambda\,\big).
\end{eqnarray}
By the Tchebychev inequality, for all $ 0\leq t\leq \frac{\lambda^{2}_{M+1}}4$,
\begin{multline}\label{eq.diff}
{\bf p}\big(\om \,:\,
\sum_{n=M+1}^{N}\frac{2}{\lambda_{n}^{2}}(|g_{n}(\om)|^{2}-1) >\lambda\big)\leq \\
\begin{aligned}
& \leq \e^{-\lambda t}\,\mathbb{E}\big[\exp\big({t \sum_{n=M+1}^{N}\frac{2}{\lambda_{n}^{2}}(|g_{n}(\om)|^{2}-1)}\big)\big]\\
&=\e^{-\lambda t}\prod_{n=M+1}^{N}\int_{\Omega}\e^{\frac{2t}{\lambda_{n}^{2}}(|g_{n}(\om)|^{2}-1)}\text{d}{\bf p}(\om)\\
&=\e^{-\lambda t}\prod_{n=M+1}^{N}\e^{-\frac{2t}{\lambda_{n}^{2}}}\big(1-\frac{2t}{\lambda_{n}^{2}}\big)^{-1}.
\end{aligned}
\end{multline}
Now observe that for all $0\leq x\leq \frac12$, $(1-x)^{-1}\leq \e^{x+x^{2}}$, hence \eqref{eq.diff} gives
\begin{multline*}
{\bf p}\big(\om \,:\,
\sum_{n=M+1}^{N}\frac{2}{\lambda_{n}^{2}}(|g_{n}(\om)|^{2}-1) >\lambda\big)\leq \\
\begin{aligned}
& \leq \e^{-\lambda t}\exp\big(4t^{2} \sum_{n=M+1}^{\infty} \frac1{\lambda_{n}^{4}} \big)\leq
\exp\big(-\lambda t+\frac{C t^{2}}{M+1} \big),
\end{aligned}
\end{multline*}
as $\lambda_{n}^{2}=2n+1$.
Choose $t=c (M+1)^{\frac12}$, with $c>0$ small enough and deduce
\begin{equation*}
{\bf p}\big(\om \,:\,
\sum_{n=M+1}^{N}\frac{2}{\lambda_{n}^{2}}(|g_{n}(\om)|^{2}-1) >\lambda\big)\leq C\e^{-c(M+1)^{\frac12} \lambda}.
\end{equation*}
Using a slight modification of the previous argument, we can show that 
\begin{equation*}
{\bf p}\big(\om \,:\,
\sum_{n=M+1}^{N}\frac{2}{\lambda_{n}^{2}}(|g_{n}(\om)|^{2}-1) <-\lambda\big)\leq C\e^{-c(M+1)^{\frac12}\lambda},
\end{equation*}
and the result follows, by \eqref{BN}. This completes the proof of Lemma~\ref{lem.BN}.
\end{proof}
We are now able to define the density $G\,:\,\H^{-\s}(\R)\longrightarrow \R$ (with respect to the measure $\mu$) of the measure $\rho$. By Lemmas \ref{lem.cauchy} and \ref{ld}, we have the following convergences in the $\mu$ measure : $F_{N}(u)$ converges in to $F(u)$ and $\|S_Nu\|_{L^{4}(\R)}$ to $\|u\|_{L^{4}(\R)}$. Then, by composition and multiplication of continuous functions, we obtain 
\begin{equation}\label{def.G}
G_{N}(u)
\longrightarrow \zeta\big(F(u)\big)
e^{\frac{1}{4}\int_{\R}|u(x)|^4\text{d}x}\equiv G(u),
\end{equation}
in measure, with respect to the measure $\mu$. As a consequence, $G$ is measurable from $\big(\H^{-\s}(\R), \mathcal{B}\big)$ to $\R$.
\subsection{Integrability of $G_{N}(u)$}
\noindent We now have all the ingredients to prove the following proposition,
which is the key point in the proof of Theorem \ref{prop}.
\begin{prop}\label{prop.lp}
Let $1\leq p<\infty$. Then there exists $C>0$ such that for every $N\geq 1$,
$$
\big\|
\zeta\big(\|\Pi_{N}u\|_{L^2(\R)}^{2}-\alpha_{N}\big)
e^{\frac{1}{4}\int_{\R}|S_Nu(x)|^4\text{d}x}
\big\|_{L^p(\text{d}\mu(u))}\leq C\,.
$$
\end{prop}
\begin{proof}
Our aim is to show that the integral
$\int_{0}^{\infty}\lambda^{p-1}\mu(A_{\lambda,N})d\lambda$ 
is convergent uniformly with respect to $N$, where
\begin{equation*}
A_{\lambda,N}=
\big\{u\in {\mathcal H}^{-\sigma}\,:\,
\zeta\big(\|\Pi_{N}u\|_{L^2(\R)}^{2}-\alpha_{N}\big)
e^{\frac{1}{4}\int_{\R}|S_Nu(x)|^4\text{d}x}>\lambda
\big\}.
\end{equation*}
Proposition \ref{prop.lp} is a straightforward consequence of the following lemma.
\begin{lemm} For any $L>0$, there exists $C>0$ such that for every $N$ and
every $\lambda\geq 1$, 
$$\mu(A_{\lambda, N}) \leq C \lambda ^{-L}.$$
\end{lemm}

\noindent We set
$$
N_0\equiv (\log\lambda)^{l},
$$
where $l$ is fixed such that $l>\max(2, \frac{1}{\beta(0)}+1)$ with $\beta(0)$
defined by Lemma~\ref{ld}.

\noindent Let us first suppose that $N_0\geq N$. Using Lemma~\ref{lem.GN} and that $\|S_Nu\|_{L^{2}}\leq \|\Pi_{N}u\|_{L^{2}}$, we get for $u\in
A_{\lambda,N}$,
\begin{eqnarray*}
\int_{\R}|S_Nu(x)|^4\text{d}x &\leq &C
\|S_Nu\|_{L^2(\R)}^{4-\theta}\|S_Nu\|^{\theta}_{{\mathcal W}^{s,p}(\R)}\\
&\leq &C(\log\log\lambda)^{2-\theta/2}\|S_Nu\|^{\theta}_{{\mathcal W}^{s,p}(\R)}\,.
\end{eqnarray*}
Therefore there exists $\delta>0$ such that
$$
\mu(A_{\lambda,N})\leq C\mu \big(u\in {\mathcal H}^{-\sigma}\,:\, 
\|S_Nu\|_{{\mathcal W}^{s,p}(\R)}>(\log\lambda)^{1/2+\delta}\big),
$$
and using Lemma~\ref{ld}, we obtain that for every $L>0$ there exists $C_L$ such
that for every
$N$ and $\lambda$ such that $(\log\lambda)^{l}\geq N$ one has
\begin{equation}\label{An2}
\mu(A_{\lambda,N})\leq C_{L}\lambda^{-L}.
\end{equation}
We next consider the case $N>N_0$. Consider the set
$$
B_{\lambda,N}=\big\{u\in {\mathcal H}^{-\sigma}\,:\,
\big|(\|\Pi_{N}u\|_{L^2(\R)}^{2}-\alpha_{N})-(\|\Pi_{N_0}u\|_{L^2(\R)}^{2}-\alpha_{N_0})\big|>1\big\}.
$$
By Lemma \ref{lem.BN}, we get
$$
\mu(B_{\lambda,N})\leq 2\exp(-c(\log\lambda)^{l/2})\leq C_{L}\lambda^{-L}\,.
$$
Hence it remains to evaluate $\mu(A_{\lambda,N}\backslash B_{\lambda,N})$.
Let us observe that for $u\in A_{\lambda,N}\backslash B_{\lambda,N}$ one has
\begin{equation*}
\begin{aligned}
\|\Pi_{N_0}u\|_{L^2}^{2} & = (\|\Pi_{N}u\|_{L^2}^{2}-\alpha_{N})-
\big[(\|\Pi_{N}u\|_{L^2}^{2}-\alpha_{N})-(\|\Pi_{N_0}u\|_{L^2}^{2}-\alpha_{N_0})\big]
\\
&\ \qquad+\alpha_{N_0}
\\
& \leq 
C+C\log(N_0)\leq C\log\log\lambda\,.
\end{aligned}
\end{equation*}
Therefore $A_{\lambda,N}\backslash B_{\lambda,N}\subset C_{\lambda,N}$ where
$$
C_{\lambda,N}\equiv
\big\{
u\in {\mathcal H}^{-\sigma}\,:\, 
\|S_Nu\|_{L^4}
\geq
c[\log\lambda]^{1/4},\,\,
\|\Pi_{N_0}u\|_{L^2}^{2}\leq C\log\log\lambda
\big\}.
$$
We next observe that thanks to the triangle inequality
$C_{\lambda,N}\subset D_{\lambda,N}\cup E_{\lambda,N}$, where
$$
D_{\lambda,N}\equiv
\big\{
u\in {\mathcal H}^{-\sigma}\,:\, 
\|S_{N_0}u\|_{L^4}
\geq
\frac{c}{4}[\log\lambda]^{1/4},\,\,
\|\Pi_{N_0}u\|_{L^2}^{2}\leq C\log\log\lambda
\big\},
$$
and
$$
E_{\lambda,N}\equiv
\big\{
u\in {\mathcal H}^{-\sigma}\,:\, 
\|S_Nu-S_{N_0}u\|_{L^4}
\geq
\frac{c}{4}[\log\lambda]^{1/4}
\big\}.
$$
The measure of $D_{\lambda,N}$ can be estimated exactly as we did in the analysis of
the case $N_0\geq N$.
Finally, using Lemma~\ref{ld}, thanks to the choice of $N_0$, we get
$$
\mu(E_{\lambda,N})\leq Ce^{-cN_0^{\beta(0)}(\log\lambda)^{1/2}}
\leq C_{L}\lambda^{-L}\,.
$$
This ends the proof of the lemma, and Proposition \ref{prop.lp} follows.
\end{proof}
We are now able to complete the proof of Theorem \ref{prop}

\begin{proof}[Proof of Theorem \ref{prop} (ii)] 
According to~\eqref{def.G}, we can extract a sub-sequence $G_{N_{k}}(u)$ so that $G_{N_{k}}(u)\longrightarrow G(u)$, $\mu$ a.s. Then by Proposition \ref{prop.lp} and the Fatou lemma, for all $p\in [1,+\infty)$,
\begin{equation*}
\int_{\H^{-\s}(\R)}|G(u)|^{p}\text{d}\mu(u)\leq \liminf_{k\to \infty} \int_{\H^{-\s}(\R)}|G_{N_{k}}(u)|^{p}\text{d}\mu(u)\leq C,
\end{equation*}
thus $G(u)\in L^{p}(\text{d}\mu(u))$.

Now it remains to check that for any Borelian set, $A\subset\mathcal{H}^{- \sigma}$, we have 
\begin{equation}\label{conv.faible}
\lim_{N\to +\infty}\int_{\H^{-\s}(\R)}1_{u\in A}G_{N}(u)\text{d}\mu(u)= \int_{\H^{-\s}(\R)}1_{u\in A}G(u)\text{d}\mu(u), \end{equation}
which will be implied by 
\begin{equation}\label{conv.faible@}
\lim_{N\to +\infty}\int_{\H^{-\s}(\R)}|1_{u\in A}(G_{N}(u)-G(u))|\text{d}\mu(u)= 0.
\end{equation}
For $N\geq 0$ and $\eps>0$, we introduce the set 
\begin{equation*} 
B_{N,\eps}=\big\{u \in \H^{-\s}(\R)\::\:|G_{N}(u)-G(u)|\leq \eps \big\},
\end{equation*}
and denote by $\ov{B_{N,\eps}}$ its complementary.\\
Firstly, as $1_{u\in A}$ is bounded, there exists $C>0$ so that for all $N\geq 0$, $\eps>0$
\begin{equation*}
\big|\int_{B_{N,\eps}}1_{u\in A}\big(G_{N}(u)-G(u)\big)\text{d}\mu(u)\big|\leq C\eps.
\end{equation*}
Secondly, by Cauchy-Schwarz, Proposition \ref{prop.lp} and as $G(u)\in L^{p}(\text{d}\mu(u))$, 
we obtain
\begin{eqnarray*}
\big|\int_{\ov{B_{N,\eps}}}1_{u\in A}\big(G_{N}(u)-G(u)\big)\text{d}\mu(u)\big|&\leq& \|G_{N}-G(u)\|_{L^{2}(\text{d}\mu)}\mu(\,\ov{B_{N,\eps}}\,)^{\frac12}\\
&\leq&C \mu(\,\ov{B_{N,\eps}}\,)^{\frac12}.
\end{eqnarray*}
By \eqref{def.G}, we deduce that for all $\eps>0$,
\begin{equation*}
\mu(\,\ov{B_{N,\eps}}\,)\longrightarrow 0, \quad N\longrightarrow +\infty,
\end{equation*}
which yields \eqref{conv.faible@}. This ends the proof of Theorem \ref{prop}
(ii). 
\end{proof}
Notice that \eqref{conv.faible@} with $A= \mathcal{H}^{-\s}$ gives
\begin{equation}\label{masse}
\rho\big(\H^{-\s}(\R)\big)=\lim_{N\rightarrow + \infty} \rho_N\big( \mathcal{H}^{- \sigma}( \R)\big)= \lim_{N\to \infty}\tilde{\rho}_{N}(E_{N}).
\end{equation}
\begin{proof}[Proof of Theorem \ref{prop} (i)]
By the argument giving~\eqref{parvo}, $\|u\|_{L^{p+1}(\R)}$ is $\mu$ almost surely
finite. As a consequence, the measure $\rho$ in the defocusing case is
nontrivial. The proof of the weak convergence of $\text{d}\rho_{N}$ to
$\text{d}\rho$ can be deduced from the proof in the focusing case. This
completes the proof of Theorem~\ref{prop}.
\end{proof}
\noindent In the focusing case, the measure $\rho=\rho_{\zeta}$ we have constructed depends on $\zeta\in \mathcal{C}_{0}^{\infty}(\R)$. We now check that it is in general not trivial. Indeed we have the following result

\begin{prop}\label{union}
The supports of the measures satisfy
\begin{equation*}
\bigcup_{\zeta \in \mathcal{C}_{0}^{\infty}(\R)}{\rm supp}\, \rho_{\zeta}={\rm supp}\;\mu.
\end{equation*}
\end{prop}
\begin{proof}
By construction, it is clear that for all $\zeta\in \mathcal{C}_{0}^{\infty}(\R)$, the support of $\rho_{\zeta}$ is included in the support of $\mu$.\\
Let $R\gg 1$ and $\zeta \in \mathcal{C}^{\infty}_{0}(\R)$ so that $0\leq \zeta \leq 1$ with $\zeta=1$ on $|x|\leq R$, and consider the associated 
measure $\rho_{\zeta}$. Let $\eps>0$. We will show that if $R$ is large enough
\begin{equation}\label{2.20}
\mu\big(u\in \H^{-\s}\,:\,|F(u)|\leq R\,\big)\geq 1-\eps.
\end{equation}
This will yield the result, as the density of $\rho_{\zeta}$ does not vanish on the set 
\linebreak[4] $\{ u\in \H^{\s}\,:\,|F(u)|\leq R \}$.\\
Write
\begin{multline}\label{2.21}
\big\{ u\in \H^{-\s}\,:\,|F(u)|>R\,\} \subset \\
\big\{ u\in \H^{-\s}\,:\,|F_{N}(u)|>R-1\,\} \cup \big\{ u\in \H^{-\s}\,:\,|F(u)-F_{N}(u)|>1\,\}, 
\end{multline}
and
\begin{equation*}
\big\{ |F_{N}(u)|>R-1\,\} \subset \\
\big\{ |F_{N}(u)-F_{0}(u)|>\frac{R-1}2\,\} \cup \big\{ |F_{0}(u)|>\frac{R-1}2\,\}.
\end{equation*}
By Lemma \ref{lem.BN} and by the direct estimate 
$$\mu(u\in \H^{-\s}\,:\, |F_{0}(u)|>\frac{R-1}2)\leq C\e^{-c R},$$
we obtain that (uniformly in $N$)
\begin{equation}\label{2.22}
\mu(u\in \H^{-\s}\,:\, |F_{N}(u)|>R-1)\leq C\e^{-c R}\leq \eps/2,
\end{equation}
if $R$ is large enough. By Lemma \ref{lem.cauchy}, if $N$ is large enough, we also have 
\begin{equation}\label{2.23}
\mu(u\in \H^{-\s}\,:\,|F(u)-F_{N}(u)|>1)\leq \eps/2.
\end{equation}
Hence from \eqref{2.21}-\eqref{2.23} we deduce \eqref{2.20}.
This in turn completes the proof of Proposition~\ref{union}.
\end{proof}
Let us remark that in the construction of the measure $\rho$ in the focusing case one may replace the
assumption of compact support on $\zeta$ by a sufficiently rapid decay as for example 
$\zeta(x)\sim \exp(-|x|^K)$, $|x|\gg 1$ with $K$ large enough. 
\section{Functional calculus of $H$}
By the classical Mehler formula for $2t\neq k\pi$ ($k\in\Z$), $f\in L^1(\R)$,
\begin{equation}\label{bg1}
e^{-itH}(f)=\frac{1}{|\sin(2t)|^{1/2}}\int_{-\infty}^{\infty}e^{i\frac{(x^2/2+y^2/2)\cos(2t)-xy}{\sin(2t)}}f(y)dy.
\end{equation}
One may check (\ref{bg1}) by a direct computation.
The explicit representation of the kernel of $\exp(-itH)$ given by (\ref{bg1}) will allow us 
to develop the functional calculus of $H$ which will be of importance in several places of our
proof of Theorem~\ref{thm4}. 
The representation of $e^{-itH}$ given by (\ref{bg1}) is also the key point of the 
proof of the local in time (deterministic) Strichartz estimates of the next section. The goal of this 
section is to prove the following statement.
\begin{prop}\label{prop.cont}
Consider for $\phi \in \mathcal{S}(\mathbb{R})$ the operator $\phi( h^2 H)$. 
Then, for any $1\leq p \leq +\infty$ and any $|\alpha| <1$, there exists $C>0$ 
such that for any $0<h\leq 1$,
$ 
\|\langle x \rangle ^{-\alpha} \phi( h^2 H) \langle x \rangle ^{\alpha} \|_{\mathcal{L} ( L^p ( \mathbb{R}))}
\leq C. 
$
\end{prop}
One may prove Proposition~\ref{prop.cont} by using a suitable pseudo-differential calculus.
We present here a direct proof based on the Mehler formula. 
The result of Proposition~\ref{prop.cont} is a consequence of the following lemma.
\begin{lemm}\label{lem.cont}
Let  $K(x,y, h ) $ be the kernel of the operator $\phi( h^2 H)$. Then there exists $C>0$ such that for any $0<h\leq 1$, we have 
\begin{equation} \label{eq.kernel}
|K(x,y,h)| \leq \frac{C} {h ( 1 + \frac{ (|x|-|y|)^2} {h^2})}\,.
\end{equation}
\end{lemm}
Let us now show how Lemma~\ref{lem.cont} implies Proposition~\ref{prop.cont}. By duality, it suffices to consider the case $\alpha\geq 0$.\goodbreak
For $\alpha \geq 0$, we have 
\begin{multline*}
\int_{-\infty}^{\infty} |K(x,y,h)| \langle y \rangle ^ {\alpha} dy\leq  
C \int_{-\infty}^{\infty} \frac{\langle y \rangle ^ {\alpha} } {h ( 1 + \frac{ (|x|-|y|)^2} {h^2})} dy  \\
\begin{aligned}
&\leq  
2C \int_{0}^{\infty} \frac{\langle y \rangle ^ {\alpha} } {h ( 1 + \frac{ (|x|-y)^2} {h^2})} dy 
 \leq 
2C \int_{0}^{\infty} \frac{1+ |x|^\alpha + ||x|-y| ^{\alpha} } {h ( 1 + \frac{ (|x|-y)^2} {h^2})}dy \\
 &\leq  
2C\langle x \rangle ^{\alpha} + 2C\int_{0}^{\infty} \frac{ ||x|-y| ^{\alpha} } {h ( 1 + \frac{ (|x|-y)^2} {h^2})}dy 
  = 
2C\langle x \rangle ^{\alpha} + 2C\int_{-\infty}^{\infty}\frac{ |hz| ^{\alpha} } { ( 1 + z^2)}dz \leq 
C'\langle x \rangle ^{\alpha} \,.
\end{aligned}
\end{multline*}
On the other hand,
\begin{equation}
\int_{-\infty}^{\infty} |K(x,y,h)| \langle x \rangle ^ {-\alpha} dx \leq \int_{||x|-|y|| < |y| /2 }\cdots + \int_{||x|-|y|| >|y|/2}\cdots\,.
\end{equation}
The contribution of the first term is bounded by $C \langle y \rangle ^ {-\alpha}$, whereas, noticing that in the second integral we have 
$(|x|-|y|)^2 \geq c ( x^2 + y^2)$, the contribution of the second term is also easily bounded by 
$C\langle y \rangle ^{-\alpha}$. Finally, Proposition~\ref{prop.cont} follows by the Schur Lemma.
Thus in order to complete the proof of Proposition~\ref{prop.cont}, it remains to prove Lemma~\ref{lem.cont}.
\begin{proof}[Proof of Lemma~\ref{lem.cont}] 
We start from the representation
$$
\varphi(h^2 H)=(2\pi)^{-1}\int_{\tau\in\R}e^{i\tau h^2 H}\widehat{\varphi}(\tau)d\tau
$$
and thus according to Mehler's formula, we have 
$$ 
|K(x,y,h)|\leq C \big|\int_{\tau\in \mathbb{R}}\frac{1}{|\sin(2h^2 \tau)|^{1/2}} 
e^{i \psi(\tau, h, x, y)} \widehat{\phi} ( \tau) d \tau\big|,
$$
where 
$$ 
\psi (\tau, h, x, y) = -\frac {1} { \sin( 2h^2 \tau) } \big( \frac {x^2 + y^2} 2 \cos( 2 h^2 \tau ) -xy\big).
$$
As a consequence, decomposing
$$|K(x,y,h)| \leq \sum_{k \in \mathbb{Z}} 
\big|\int_{- \frac \pi 2 + k \pi < h^2 \tau < \frac \pi 2 + k \pi}
\frac {1}{ |\sin(2h^2 \tau)|^{1/2}} e^{i \psi(\tau, h, x, y)} \widehat{\phi} ( \tau) d \tau\big|
$$
and taking benefit that the function $\widehat { \phi}$ is in $\mathcal{S}$, we obtain
\begin{multline*}
|K(x,y,h)|\\ 
 \leq  \frac C h + 
\sum_{k \in \mathbb{Z}\setminus \{0\}} 
\big|\int_{- \frac \pi 2 + k \pi < h^2 \tau < \frac \pi 2 + k \pi}
\frac {1}{|\sin(2h^2 \tau)|^{1/2}} e^{i \psi(\tau, h, x, y)} \widehat{\phi} ( \tau) \big|
d \tau
 \leq  \frac C h\,.
\end{multline*}
As a consequence, it is enough to prove
\begin{equation}\label{eq.final}
|K(x,y,h)| \leq \frac {Ch} { (|x|-|y|)^2}.
\end{equation}
The key point of the analysis will be the following estimates on the phase function.
\begin{lemm}\label{lem.1} 
There exists $C>0$ such that for any $x,y\in \mathbb{R}$ and any $0<h\leq 1$ we have 
$$ \partial_\tau \psi(x,y,\tau, h) \geq h^2 \frac {x^2 + y^2}{2}\,.
$$
\end{lemm}
\begin{proof}
Indeed,
\begin{equation}\label{eq.derivee} 
\partial_\tau \psi(x,y,\tau, h) = 
\frac{ 2h^2} {  \sin ^2( 2h^2 \tau )} ( \frac { x^2 + y^2} 2 -xy \cos ( 2h^2 \tau ) )
\end{equation}
and minimizing with respect to $x$ the expression above gives 
$$ \partial_\tau \psi(x,y,\tau, h) \geq \frac{2 h^2} {  \sin ^2( 2h^2 \tau )} 
( \frac{y^2} 2 \sin^2 ( 2h^2 \tau ) )=
h^2y^2\,.
$$ 
Similarly 
$$ 
\partial_\tau \psi(x,y,\tau, h) \geq \frac{2 h^2} {  \sin ^2( 2h^2 \tau )} 
( \frac{x^2} 2 \sin^2 ( 2h^2 \tau ) )=h^2x^2
$$
and the result of Lemma~\ref{lem.1} follows.
\end{proof}
\begin{lemm}
\label{lem.2} 
There exists $C>0$ such that for any $x,y\in \mathbb{R}$, any $\tau$, and $0<h\leq 1$, we have 
$$ 
\partial_\tau \psi(x,y,\tau, h) \geq \frac{h^2(|x|-|y|)^2}{\sin^2 (2h^2 \tau )} \,.
$$
\end{lemm}
\begin{proof}
Indeed, this estimate is a straightforward consequence of~\eqref{eq.derivee} and
$$
\frac { x^2 + y^2} 2 -xy \cos (2h^2 \tau) \geq \frac { x^2 + y^2} 2 -|xy|= \frac{ (|x|-|y|)^2} 2  \,.
$$
This completes the proof of Lemma~\ref{lem.2}.
\end{proof}
Let us now complete the proof of Lemma~\ref{lem.cont}. 
To estimate $K$ we integrate by parts using the operator 
$T= \frac 1{ \partial_\tau \psi (x,y,\tau,h)} \partial _\tau$. Notice that according to Lemma~\ref{lem.2}, the singularity of $\frac 1 { |\sin(2h^2 \tau)|^{1/2}}$ is harmless and we obtain
$$
|K(x,y,h)| \leq C \big| \int_{\tau\in \mathbb{R}} e^{ i \psi(x,y,\tau, h)} \partial_\tau 
\bigl( \frac 1 { |\sin(2h^2 \tau)|^{1/2}} \frac 1 {\partial_\tau \psi } 
\widehat{\phi} ( \tau) \bigr) d\tau\big|. 
$$
In the expression above, we have three contributions according whether the derivative falls on either terms. 
If the derivative falls on the last term, we obtain a contribution which is, 
according to Lemma~\ref{lem.2}, bounded by 
\begin{multline*}
C_N 
\int_{\tau\in \mathbb{R}} 
\frac{\sin^2(2h^2 \tau)} { |\sin|^{1/2}(2h^2 \tau)h^2 ( |x|-|y|)^2} (1+|\tau|)^{-N} 
d\tau\\
\leq C_N 
\int_{\tau\in \mathbb{R}} \frac {h^3 |\tau|^{3/2}} { h^2 ( |x|-|y|) ^2}
(1+|\tau|)^{-N}
d \tau
\leq \frac{Ch} {(|x|-|y|)^2}\,.
\end{multline*}
If the derivative falls on the first term, we obtain a contribution which is bounded by 
\begin{multline}
C\int_{\tau\in \mathbb{R}} 
\frac{ h^2 |\widehat{\phi} (\tau)| d\tau } { |\sin|^{3/2}(2h^2 \tau)|\partial_ \tau \psi|} 
\leq Ch \int_{\tau\in \mathbb{R}} \frac{|\sin|^{1/2}(2h^2 \tau)} { ( |x|-|y|)^2}|\widehat{\phi} (\tau)| d\tau\\
\leq C \int_{\tau\in \mathbb{R}} \frac{h |\tau|^{1/2}} { ( |x|-|y|)^2}|\widehat{\phi} (\tau)| d\tau\leq \frac{Ch} {(|x|-|y|)^2}\,.
\end{multline}
Finally, the last case is when the derivative falls on the second term. 
In this case, using the relation
\begin{equation*}
\frac{ \partial_\tau^2 \psi} { (\partial_\tau \psi)^2} = 
-\frac{4 h^2\cos (2h^2 \tau)} {\sin (2h^2 \tau) \partial _\tau \psi} 
+ \frac {4h^4 xy} {\sin (2h^2 \tau) (\partial _\tau \psi)^2}
\end{equation*}
and Lemma~\ref{lem.2}, we obtain a contribution which is bounded by 
\begin{multline}
C\int_{\tau\in \mathbb{R}} 
(\frac{ h^2} {|\sin|^{3/2} (2h^2 \tau) |\partial _\tau \psi| } + 
\frac{ h^4 |xy| } { |\sin|^{3/2} (2h^2 \tau)(\partial _\tau \psi)^2}) |\widehat{\phi} (\tau)|
d \tau \\
\leq C 
\int_{\tau\in \mathbb{R}} 
\big|\frac{|\sin|^{1/2}(2h^2 \tau)} 
{ ( |x|-|y|)^2}\widehat{\phi} (\tau)\big| d\tau\leq \frac{Ch} {(|x|-|y|)^2} 
\end{multline} 
where to estimate
$\frac{ h^4 xy } { \sin (2h^2 \tau)(\partial _\tau \psi)^2}$ we used Lemma~\ref{lem.1} 
to estimate one of the $\partial _\tau \psi$ factors and Lemma~\ref{lem.2} to estimate the other one.
This concludes the proof of~\eqref{eq.final} and hence of Lemma~\ref{lem.cont}. 
\end{proof} 
%%%%%%%%%%%%%%%%%%%%%%%%%%%%%%%%%%%%%%%%%%%%%%%%%%%%
%%%%%%%%%%%%%%%%%%%%%%%%%%%%%%%%%%%%%%%%%%%%%%%%%%%% 
\section{Strichartz estimates}
We state the Strichartz inequality (local in time) satisfied by the free evolution $e^{-itH}$.
\begin{lemm}\label{mehler}
Let us fix $s\in \R$.
For every $p\geq 4,q\geq 2$ satisfying $\frac 2p+\frac 1q=\frac 12$, every $T>0$, there exists $C>0$ and such that\begin{equation}\label{strich1}
\|e^{-itH}\|_{{\mathcal H}^s(\R)\rightarrow L^p((0, 2\pi);{\mathcal W}^{s,q}(\R))}\leq C.
\end{equation}
\end{lemm}
There is also a set of inhomogeneous Strichartz estimates which will not be used here.
\begin{proof}
Coming back to the definition of the spaces ${\mathcal W}^{s,p}(\R)$, we first observe that
it suffices to consider the case $s=0$.
We have that $
\|e^{-itH}\|_{L^2\rightarrow L^2}=1.
$
Next, as a consequence of (\ref{bg1}),
$
\|e^{-itH}\|_{L^1\rightarrow L^\infty}\leq C/|t|^{1/2}
$ 
for $t$ close to zero, i.e. the singularity of $\|e^{-itH}\|_{L^1\rightarrow L^\infty}$ for $t\sim 0$
is the same as for $\exp(it\partial_x^2)$ and thus (see e.g. \cite{GV}) $e^{-itH}$ enjoys the same 
local in time Strichartz estimates as the free 
Schr\"odinger equation which is precisely the statement of \eqref{strich1}.
This completes the proof of Lemma~\ref{mehler}.
\end{proof}
%%%%%%%%%%%%%%%%%%%%%%%%%%%%%%%%%%%%%%%%%%%%%%%%%%%%%%%%%%%%%%%%%%%%%%%%%%%%%%%%%%%%%%%%%%%%%%%%%%%%%%
We need some stochastic improvements of the Strichartz estimates.
\noindent The following lemma shows that there is a gain of regularity in 
$L^{p}$ spaces for the free Schr\"odinger solution.
\begin{lemm}\label{lemm2}
Let $\eps <\frac 1 6$. For any $p,q\geq 4$, there exist $C,c>0$ such that 
\begin{equation*}
\\\begin{gathered}
\forall\, \lambda\geq 1,\, \forall\,  N\geq 1,\;
{\mu} (u\in \mathcal{H}^{- \sigma}\,:\, \|e^{-itH}u\|_{L^{p}_{(0,2\pi)}{\mathcal
W}^{\eps ,q}(\R)}>\lambda)\leq Ce^{-c\lambda^2}\,\\
\forall\, \lambda\geq 1,\, \forall\,  N\geq 1,\; \tilde {\mu} _N(u\in E_N\,:\, \|e^{-itH}u\|_{L^{p}_{(0,2\pi)}{\mathcal
W}^{\eps ,q}(\R)}>\lambda)\leq Ce^{-c\lambda^2}\,.
\end{gathered}
\end{equation*}
\end{lemm}

\begin{proof}
Let us prove the first estimate, the proof of the second being similar. By the definition of ${\mu}$, we have to show that 
\begin{equation}\label{p}
{\bf p}\big(\om \in \Omega \,:\,\| \e^{-itH}\phi\|_{L^{p}_{(0,2\pi)}{\mathcal
W}^{\eps, q}(\R)}>\lambda\big)\leq Ce^{-c\lambda^2}\,.
\end{equation}
Now by Lemmas \ref{Gaussian}, \ref{disp} and Minkowski's inequality, for $r\geq p,q$, we obtain
\begin{eqnarray}\label{series}
\|\e^{-itH}\phi(\omega, \cdot)\|_{L^{r}(\Omega)L^p_{(0,2\pi)}\W^{{\eps}, q}(\R)}&\leq&
C\|\<H\>^{\frac{\eps}2}\e^{-itH}\phi\|_{L^p_{(0,2\pi)}L^{q}(\R)L^{r}(\Omega)}\nonumber \\\nonumber 
&\leq&C \sqrt{r} \big(\sum_{n= 0}^{\infty}\lambda_{n}^{2(\eps-1)}\|h_{n}\|_{L^{q}}^{2}\big)^{\frac12}\\
&\leq&C \sqrt{r}\big(\sum_{n=0}^\infty\lambda_{n}^{2(\eps-1-\frac1{6})}\big)^{\frac12}.
\end{eqnarray}
Coming back to the definition of $\lambda_n$, we get that the sum \eqref{series} is finite.
The estimate \eqref{p} then follows from the Bienaym\'e-Tchebychev inequality :
\begin{multline*}
{\bf p}\big(\,\om \in \Omega \,:\,\| \e^{-itH}\phi\|_{L^{p}_{(0,2\pi)}{\mathcal
W}^{\eps, q}(\R)}> \lambda\, \big)\\
 \leq  \bigl(
\frac{\|\e^{-itH}\phi\|_{L^{r}(\Omega)L^p_{(0,2\pi)}\W^{\eps, q}(\R)}}{ \lambda}
\bigr) ^{r}
 \leq \big(\frac{ C \sqrt{r}} {\lambda} \big)^r,
\end{multline*}
and the choice $r= \epsilon \lambda ^2$ with $\epsilon >0$ small enough.
This completes the proof of Lemma~\ref{lemm2}. 
\end{proof}
%%%%%%%%%%%%%%%%%%%%%%%%%%%%%%
%%%%%%%%%%%%%%%%%%%%%%%%%%%%%
We shall in practice need the following consequence of Lemma~\ref{lemm2} and \eqref{normehs}.
\begin{lemm}\label{lemmrho}
Let $\s>0$, $0<\eps<\frac16$ and $p,q\geq 4$. Then there exist $C,c>0$ so that for 
every $\lambda\geq 1$, every $N\geq 1$,
\begin{equation}\label{dev1}
\begin{gathered}
\rho \big(\,u\in \H^{-\s}\,:\,\|u\|_{\H^{-\s}}>\lambda\,\big)\leq C\e^{-c\lambda^{2}},\\
\tilde\rho_{N}\big(\,u\in E_N\,:\,\|u\|_{\H^{-\s}}>\lambda\,\big)\leq C\e^{-c\lambda^{2}}
\end{gathered}
\end{equation}
and 
\begin{equation}\label{dev2}
\begin{gathered}
\rho \big(\,u\in \H^{-\s}\,:\,\|\e^{-itH}u\|_{L^{p}_{(0,2\pi)}\W^{\eps,q}(\R)}>\lambda\,\big)\leq 
C\e^{-c\lambda^{2}},\\
\tilde\rho_{N}\big(\,u\in E_N\,:\,\|\e^{-itH}u\|_{L^{p}_{(0,2\pi)}\W^{\eps,q}(\R)}>\lambda\,\big)\leq 
C\e^{-c\lambda^{2}}.
\end{gathered}
\end{equation}
\end{lemm}
\begin{proof}
In the defocusing case the proof is a straightforward consequence of the bounds for $\mu, \tilde{\mu} _N$ 
we have already established. 
Namely, in this case it is a straightforward consequence of the inequalities
$$\rho(A) \leq \mu(A), \qquad \tilde \rho_N(A)\leq \tilde{\mu}_N(A).$$
We thus only consider the focusing case which is slightly more delicate.
We prove \eqref{dev1}. By the definition and the Cauchy-Schwarz inequality, we have
\begin{equation*}
\begin{aligned}\rho_{N}\big(\,u\in \H^{-\s}\,:\,\|u\|_{\H^{-\s}}>\lambda\,\big)
&=\int_{\mathcal{H}^{- \sigma}} 1_{\|u\|_{\mathcal{H}^{-\sigma}} > \lambda} G_N(u) d\mu(u)\\
&\leq \|G_{N}(u)\|_{L^{2}(\text{d}\mu(u))} \,\mu\big(u: \|u\|_{\H^{-\s}}>\lambda\,\big)^{\frac12},
\end{aligned}
\end{equation*}
and we obtain 
$$\rho_N \big(\,u\in \H^{-\s}\,:\,\|u\|_{\H^{-\s}}>\lambda\,\big)\leq C\e^{-c\lambda^{2}},$$
and the first claim follows by using~\eqref{limite}. The proof of the three other claims are similar. This completes the proof of Lemma~\ref{lemmrho}.
\end{proof}
%%%%%%%%%%%%%%%%%%%%%%%%%%%%%%%%%%%%%%%%%%%%%%%%%%%%%%%%%%%%%%%%%%%%%%%%%%%%%%%%%%%%%%%%%%%%%%%%%%%%%%%%%%%%%%%%%%%%%%%%%%%%%%%%%%%%%%%%%%%%%%%%
\section{Local smoothing effects}
The next result is based on the well-known smoothing effect.
\begin{lemm}[Deterministic smoothing effect]\label{gl}
Let us fix two positive numbers $s$ and $\sigma$ such that $s<\sigma<1/2$. Then there exists $C>0$ so that 
\begin{equation}\label{kiwi}
\big\|\<x\>^{-\sigma}\,\sqrt{H}^{s}\,\e^{-itH}f\big\|_{L^{2}([0,2\pi]\times \R)} \leq C \|f\|_{L^{2}(\R)}.
\end{equation}
\end{lemm}
\begin{proof}
Inequality (\ref{kiwi}) is a slight variation of the ``usual'' local smoothing effect 
for the harmonic oscillator, namely for $\alpha>1/2$,
\begin{equation}\label{gladak}
\big\|\<x\>^{-\alpha}\,\sqrt{H}^{\frac1 2}\,\e^{-itH}f\big\|_{L^{2}([0,2\pi]\times \R)} \leq C \|f\|_{L^{2}(\R)}.
\end{equation}
We refer to \cite{YajimaZhang1,YajimaZhang2} for a proof of (\ref{gladak}).
Let us fix $\alpha>1/2$ such that $1<2\alpha<\sigma/s$. Take $\theta\in (0,1)$ such that $\sigma=\theta\alpha$. 
Then thanks to our choice of $\alpha$, we have that $s<\frac{\theta}{2}$.
Applying (\ref{gladak}) to $h_n$ gives that 
$$
\|\<x\>^{-\alpha}\,h_{n}(x)\|_{L^{2}(\R)}\leq C \lambda_{n}^{-\frac12}\,.
$$
Interpolation between the last inequality and the 
equality $\|h_{n}\|_{L^{2}(\R)}=1$ yields that 
\begin{equation*}
\|\<x\>^{-\sigma}\,h_{n}(x)\|_{L^{2}(\R)}\leq C \lambda_{n}^{-\frac{\theta}{2}}.
\end{equation*}
Since $s<\frac{\theta}{2}$, we obtain that there exists $\delta(s,\sigma)>0$ such that
\begin{equation}\label{kkk}
\|\<x\>^{-\sigma}\,h_{n}\|_{L^{2}(\R)}\leq C \lambda_{n}^{-\delta(s,\sigma)-s}.
\end{equation}
The last estimate in conjugation with \cite[Corollary 1.2]{Thomann6} implies (\ref{kiwi}) (notice that here we
do not need the $\delta(s,\sigma)$ saving in \eqref{kkk}).
This completes the proof of Lemma~\ref{gl}.
\end{proof}
We also have the following stochastic improvement of the smoothing effect.
%%%%%%%%%%%%%%%%%%%%%%%%%%%%%%%%%%
\begin{lemm}[Stochastic smoothing effect]\label{lem.smooth}
Let $s,\sigma$ be two positive numbers  such that $s<\sigma<1/2$
and $q\geq 2$. Then there exist $C,c>0$ so that for every $\lambda>0$, every 
 $N\geq 1$, 
\begin{equation}\label{eq.lem.smooth}
\begin{gathered}
\rho \big(\,u\in \H^{-\s} \;:\; 
\big\|\,\<x\>^{-\sigma}\,
\sqrt{H}^{s}\,\e^{-itH}u\,\big\|_{L^{q}_{(0,2\pi)}L^{2}(\R)} > \lambda \, ) \leq C\e^{-c\lambda^{2}},\\
\tilde \rho_{N}\big(\,u\in E_N\;:\; 
\big\|\,\<x\>^{-\sigma}\,
\sqrt{H}^{s}\,\e^{-itH}u\,\big\|_{L^{q}_{(0,2\pi)}L^{2}(\R)} > \lambda \, ) \leq C\e^{-c\lambda^{2}}.
\end{gathered}
\end{equation}
\end{lemm}
%%%%%%%%%%%%%%%%%%%%%%%%%%%%%%%%%%%%%%%%%%%%%%%%%%%%%%%%%%%%%%%%%%%%%%%%%%%%%%%%%%%%%%%%%%%%%%%%%%%%%%%%%%%%%%%%%%%%%%%%%%%%%%%%%%%%%%%%%%%%%%%%%%%%%%%%
\begin{proof}
Again we only prove the first claim. We compute 
\begin{equation*}
\<x\>^{-\sigma}\,\sqrt{H}^{s}\,
\e^{-itH}
\varphi(\omega,x) = 
\sum_{n\geq 0}\frac{\sqrt{2}}{\lambda^{1-s}_{n}}
\e^{-it\lambda_{n}^{2}}g_{n}(\om) \frac1{\<x\>^{\sigma}}h_{n}(x).
\end{equation*}
Then by Lemma~\ref{Gaussian}
\begin{equation*}
\big\|\,\<x\>^{-\sigma}\,\sqrt{H}^{s}\,\e^{-itH}\varphi(\omega,x) \,\big\|_{L^{r}(\Omega)}
\leq C\sqrt{r}\,\big( \sum_{n\geq 0}\frac{1}{\<\lambda_n\>^{2(1-s)}}\,
\big|\frac{h_{n}(x)}{\<x\>^{\sigma}}\big|^{2}\big)^{\frac12}.
\end{equation*}
An application of the Minkowski inequality and \eqref{kkk} give 
\begin{equation*}
\begin{aligned}
\big\|\,
\<x\>^{-\sigma}\,\sqrt{H}^{s}\,\e^{-itH}
\varphi(\omega,x)\, \big\|_{L^{r}(\Omega ;L^q_{T}L^{2}(\R))}&\leq C\sqrt{r}\,
\big( \sum_{n\geq 0}\frac{1}{\<\lambda_n\>^{2(1-s)}}\,
\big\|\frac{h_{n}}{\<x\>^{\sigma}}\big\|_{L^{2}(\R)}^{2}\big)^{\frac12}\\
&\leq C\sqrt{r}\,\big( \sum_{n\geq 0}\frac{1}{\<\lambda_n\>^{2+2\delta(s,\sigma)}}\big)^{\frac12}\\
&\leq C \sqrt{r}.
\end{aligned}
\end{equation*}
Using the Tchebychev inequality, as we did in the proof of Lemma~\ref{lemm2} yields
\begin{equation}\label{eq.lem.smooth.mu}
\mu\big(\,u\in \mathcal{H}^{-\sigma} \;:\; 
\big\|\,\<x\>^{-\sigma}\,
\sqrt{H}^{s}\,\e^{-itH}u\,\big\|_{L^{q}_{(0,2\pi)}L^{2}(\R)} > \lambda \, ) \leq C\e^{-c\lambda^{2}}.
\end{equation}
Finally we deduce \eqref{eq.lem.smooth} 
from \eqref{eq.lem.smooth.mu} as we did in the proof of Lemma~\ref{lemmrho}. 
This completes the proof of Lemma~\ref{lem.smooth}.
\end{proof}
%%%%%%%%%%%%%%%%%%%%%%%%%%%%%%%%%%%%%%%%%%%%%%%%%%%%%%%%%%%%%%%%%%%%%%%%%%%%%%%%%%%%%%%%%
%%%%%%%%%%%%%%%%%%%%%%%%%%%%%%%%%%%%%%%%%%%%%%%%%%%%%%%%%%%%%%%%%%%%%%%%%%%%%%%%%%%%%%%%%%%%%%%%%%%%%%
\section{Local in time results for the nonlinear problem}
In this section, we use the linear dispersive estimates established in the previous sections to develop a local Cauchy theory. As the solution we are looking for, will be the sum of the linear solution associated to our initial data and of a smoother term, our functional  spaces are naturally {\em the sum  of two spaces}: one which corresponds to the properties of the linear probabilistic solutions, and the other one corresponding to the properties of the deterministic smoother solutions. Fortunately, it turns out that these two spaces have a non trivial intersection which is sufficient to perform the analysis and hence avoid the technicalities involving sum spaces.  
\subsection{Initial data spaces}
For $\alpha\in\R$, we define the spaces ${\mathcal H}^{s}_{\langle x \rangle^{\alpha}}(\R)$ equipped with the
norm
$$
\|u\|_{{\mathcal H}^{s}_{\langle x \rangle^{\alpha}}}=
\|\langle x \rangle^{\alpha}\sqrt{H}^s u\|_{L^2}\,.
$$
Recall that $e^{-itH}$ defines the free evolution.
We define the spaces for the initial data $Y^s$
$$
Y^s=\big\{u\in {\mathcal H}^{-\varepsilon/10}\,:\, e^{-itH}(u)\in L^{2(k-1)+\eps}_{2\pi}
{\mathcal W}^{\frac{s+\eps}{k-1},r}
\cap L^{2}_{2\pi}{\mathcal H}^{s}_{\langle x \rangle^{-s-\eps/4}}\big\},
$$
where  $s$ is a positive number satisfying
\begin{equation}\label{restriction_s}
\frac{k-3}{2(k-2)}<s<\min\big(\frac{1}{2},\frac{k-1}{6}\big),
\end{equation}
$\varepsilon>0$, is a small number and $r$ is a large number all 
depending on $s$ to be fixed.
The values of $\eps$, $r$ in the definition of the space $Y^s$ 
(and also the space $X^s_T$ defined in the next 
section) will be fixed by the analysis of the next two sections.
Note that since $e^{-itH}$ is $2\pi$ periodic the interval 
$[0,2\pi]$ in the definition of $Y^s$ may be replaced by any interval of size $2\pi$.
We equip the spaces $Y^s$ with the natural norm
$$
\|u\|_{Y^s}=\|u\|_{{\mathcal H}^{-\varepsilon/10}}+
\|e^{-itH}u\|_{L^{2(k-1)+\eps}_{2\pi}{\mathcal W}^{\frac{s+\eps}{k-1},r}}
+\|e^{-itH}u\|_{ L^{2}_{2\pi}{\mathcal H}^{s}_{\langle x \rangle^{-s-\eps/4}}}\,.
$$

Thanks to Proposition~\ref{prop.cont}, we obtain that $\|S_N\|_{Y^s\rightarrow Y^s}$ is bounded,
uniformly in $N$, provided $\eps$ is small enough. The main property of the space $Y^s$ we use, is the following Gaussian property.
%%%%%%%%%%%%%%%%%%%%%%%%%%%%%%%%%%%%%%%%%%%%%%%%%%%
\begin{lemm}\label{cfg}
For every $s$ satisfying (\ref{restriction_s})
there exists $\eps_0>0$ and two positive constants $C$ and $c$ such that for
every $N\geq 1$, every $\lambda\geq 1$, every $\eps\in (0,\eps_0)$, every $r\geq 4$, every $N$ 
\begin{equation}
\begin{gathered}
\rho (u\in {\mathcal H}^{-\eps/10}\,:\, \|u\|_{Y^s}>\lambda)\leq Ce^{-c\lambda^2}\,\\
\tilde\rho_{N}(u\in E_N\,:\, \|u\|_{Y^s}>\lambda)\leq Ce^{-c\lambda^2}\,.
\end{gathered}
\end{equation}
(recall that the dependence on $\eps$ and $r$ of $Y^s$ is implicit). 
\end{lemm}
\begin{proof}
As before, we only prove the first claim. As a consequence of Lemma~\ref{lem.smooth}, 
we get that for every $s\in (0,1/2)$ and every $\eps>0$,
\begin{equation*}
\rho \big(\,u\in \H^{-\eps/10} \;:\; 
\big\|\e^{-itH}u\,\big\|_{L^{2}_{2\pi}
{\mathcal H}^{s}_{\langle x \rangle^{-s-\eps/4}}} > \lambda \, ) 
\leq C\e^{-c\lambda^{2}}.
\end{equation*}
Next, using Lemmas~\ref{lem.3.4} and~\ref{lemmrho}, we obtain that for every $s\in (0,(k-1)/6)$ and $\eps>0$ 
such that 
$s+\eps<(k-1)/6$, every $r\geq 4$,
\begin{equation*}
\rho\big(\,u\in \H^{-\eps/10} \;:\; \|u\|_{{\mathcal H}^{-\eps/10}}+
\big\| \bigl(\e^{-itH}u\bigr)\,\big\|_{L^{2(k-1)+\eps}_{2\pi}{\mathcal W}^{\frac{s+\eps}{k-1},r}} 
>\lambda \big) \leq C\e^{-c\lambda^{2}}.
\end{equation*}
This  completes the proof of Lemma~\ref{cfg}.
\end{proof}
%%%%%%%%%%%%%%%%%%%%%%%%%%%%%%%%%%%%%%%%%%%%%%%%%%%%%%%%%%%%%%%%%%%%%%%%%%%%%%%%%%%%%%%%%%%%%%%%%%%%%%
%%%%%%%%%%%%%%%%%%%%%%%%%%%%%%%%%%%%%%%%%%%%%%%%%%%%%%%%%%%%%%%%%%%%%%%%%%%%%%%%%%%%%%%%%%%%%%%%%%%%%
\subsection{Solution spaces and linear estimates}
We define the solution spaces of functions on $[-T,T]\times \R$, defined by
$$
X^s_{T}=\{ u\in  L^{\infty}_{T}{\mathcal H}^{-\varepsilon/10}
\cap
L^{2}_{T}{\mathcal H}^{s}_{\langle x \rangle^{-s-\eps/4}} \cap L^{2(k-1)+\eps}_{T}{\mathcal W}^{\frac{s+\eps}{k-1},r}
$$

where $s$ satisfies (\ref{restriction_s}), $\varepsilon$ is a small positive number
and $r\gg 1$ is a large number, to be chosen in function of $s$. We equip $X^s_{T}$ with the norm
$$
\|u\|_{X^s_{T}}=\|u\|_{L^{\infty}_{T}{\mathcal H}^{-\varepsilon/10}}
+\|u\|_{L^{2(k-1)+\eps}_{T}{\mathcal W}^{\frac{s+\eps}{k-1},r}}
+\|u\|_{L^{2}_{T}{\mathcal H}^{s}_{\langle x \rangle^{-s-\eps/4}}}\,.
$$
In the next lemma, we state the linear estimates.
\begin{lemm}\label{linear}
For every $\tau\in\R$ and $T\leq 2\pi$,
\begin{equation*}
\|e^{-i\tau H}u\|_{X^s_{T}}\leq \|u\|_{Y^s},\qquad 
\|e^{-i\tau H}u\|_{Y^s}= \|u\|_{Y^s}.
\end{equation*}
Moreover, if $s$ satisfies (\ref{restriction_s}), there exist $\eps_0>0$ and $r_0\geq 2$ 
such that for $\eps, \eta 
\in (0,\eps_0)$, $r>r_0$,
\begin{equation}\label{gi}
\|\int_{0}^t e^{-i(t-\tau)H}(F(\tau))d\tau\|_{X^s_{T}}\leq C\|F\|_{L^1_{T}{\mathcal H}^{s-\eta}}
\end{equation}
and for $t\in [-T,T]$
\begin{equation}\label{gi_pak}
\|\int_{0}^t e^{-i(t-\tau)H}(F(\tau))d\tau\|_{Y^s}\leq C\|F\|_{L^1_{T}{\mathcal H}^{s- \eta}}
\end{equation}
(recall that the dependence on $\eps$ and $r$ of $X^s_T$ and $Y^s$ is implicit).
\end{lemm}
\begin{proof}
The first estimate is a direct consequence of the definition.
Let us prove (\ref{gi}).
We first observe that if $s$ satisfies (\ref{restriction_s}) then thanks to the Sobolev
inequality and (\ref{weyl-hormander}) we have
\begin{equation}\label{parvo.bis}
\|u\|_{L^{2(k-1)+\eps}_{T}{\mathcal W}^{\frac{s+\eps}{k-1},r}}
\leq
C(\|u\|_{L^{\infty}_{T}{\mathcal H}^{s- \eta}}+\|u\|_{L^4_{T}{\mathcal W}^{s-\eta,\infty}}),
\end{equation}
provided the positive numbers $\eps_0$ is small enough and $r$ is large enough.
Indeed, thanks to the Sobolev embedding, we have that
$$
\|u\|_{L^{2(k-1)+\eps}_{T}{\mathcal W}^{\frac{s+\eps}{k-1},r}}
\leq C
\|u\|_{L^{2(k-1)+\eps}_{T}
{\mathcal W}^{\frac{s+\eps}{k-1}+\sigma,\frac{4k-4+2\eps}{2k-6+\eps}}}\,,
$$
provided
$$
\sigma>\frac{2k-6+\eps}{4k-4+2\eps}-\frac{1}{r}.
$$
Observe that the couple $(p,q)=(2(k-1)+\eps,\frac{4k-4+2\eps}{2k-6+\eps})$ satisfies $p\geq 4$ and
$\frac{2}{p}+\frac{1}{q}=\frac{1}{2}$. 
Therefore (\ref{parvo.bis}) holds, if we can assure that
$$
s- \eta>\frac{2k-6+\eps}{4k-4+2\eps}-\frac{1}{r}+\frac{s+\eps}{k-1}\,.
$$
But the last condition follows from (\ref{restriction_s}), provided $0< \eps, \eta < \eps_0$, if $\eps_0$ is small enough and $r$ large enough. This proves (\ref{parvo.bis}).
Using (\ref{parvo.bis}), the Strichartz estimates of Lemma~\ref{mehler} and
the Minkowski inequality, we obtain that
\begin{multline*}
\|\int_{0}^t e^{-i(t-\tau)H}(F(\tau))d\tau\|_{
L^{2(k-1)+\eps}_{T}{\mathcal W}^{\frac{s+\eps}{k-1},r}
}\\
\leq 
C\|e^{i\tau H}F(\tau)\|_{L^1_{T}{\mathcal H}^{s-\eta}}
=
C\|F\|_{L^1_{T}{\mathcal H}^{s-\eta}}\,.
\end{multline*}
We next observe that as a consequence of Lemma~\ref{gl}, for every $s$ satisfying (\ref{restriction_s}) 
there exists $\eps_0$ such that for $\eps\in (0,\eps_0)$,
\begin{equation}\label{vtoro.bis}
\|e^{-itH}\|_{L^2\rightarrow L^2_{T}
{\mathcal H}^{s}_{\langle x \rangle^{-s-\eps/4}}}
\leq C.
\end{equation}
Using (\ref{vtoro.bis}) and the Minkowski inequality, we obtain that
$$
\|\int_{0}^t e^{-i(t-\tau)H}(F(\tau))d\tau\|_{
L^2_{T}{\mathcal H}^{s}_{\langle x \rangle^{-s-\eps/4}}}
\leq 
C\|F\|_{L^1_{T}{L^2}}
\leq
C\|F\|_{L^1_{T}{\mathcal H}^{s-\eta}}\,.
$$
The proof of (\ref{gi}) is completed by the straightforward bound
$$
\|\int_{0}^t e^{-i(t-\tau)H}(F(\tau))d\tau\|_{
L^\infty_{T}{\mathcal H}^{-\eps/10}}
\leq C\|F\|_{L^1_{T}{\mathcal H}^{-\eps/10}}
\leq C\|F\|_{L^1_{T}{\mathcal H}^{s-\eta}}\,.
$$
Let us now prove (\ref{gi_pak}). Using (\ref{parvo.bis}), 
(\ref{vtoro.bis}) and the Minkowski inequality, we obtain that for $t\in [0,T]$,
\begin{equation*}
\|\int_{0}^t e^{-i(t-\tau)H}(F(\tau))d\tau\|_{Y^s}
\leq C\|e^{-i(t-\tau)H}(\chi(\tau,t)F(\tau))\|_{L^1_{T}{\mathcal H}^{s-\eta}}
\leq C\|F\|_{L^1_{T}{\mathcal H}^{s-\eta}}\,,
\end{equation*}
where $\chi(\tau,t)$ denotes the indicator function of $\tau\in[0,t]$. 
This completes the proof of Lemma~\ref{linear}.
\end{proof}
%%%%%%%%%%%%%%%%%%%%%%%%%%%%%%%%%%%%%%%%%%%%%%%%%%%%%%%%%%%%%%%%%%%%%%%%%%%%%%%%%%%%%%%%%%%%%
%%%%%%%%%%%%%%%%%%%%%%%%%%%%%%%%%%%%%%%%%%%%%%%%%%%%%%%%%%%%%%%%%%%%%%%%%%%%%%%%%%%%%%%%%%%%%%
\subsection{Multilinear estimates}
\begin{prop}\label{multi}
Assume that $s$ satisfies (\ref{restriction_s}) and let $\eta>0$.
There exist $\eps_0>0$ and $r_0\geq 2$  such that the following holds true.
For every $\eps\in (0,\eps_0)$, and every $r>r_0$ satisfying $3\eps r > 4(k-1)$, there exists $\kappa>0$ such that
for $T\leq 2\pi$ we have the estimates
\begin{equation}\label{multi_1}
\|u_1 \cdots u_{k}\|_{L^1_{T}{\mathcal H}^{s-\eta}}
\leq CT^{\kappa}\prod_{j=1}^{k}\|u_j\|_{X^s_{T}}
\end{equation}
and, uniformly in $N$,
\begin{multline}
\|S_N((S_N u_1) (S_N u_2)\cdots (S_Nu_{k}))\|_{L^1_{T}{\mathcal H}^{s-\eta}}
\\
\leq CT^{\kappa}\prod_{j=1}^{k}\|S_Nu_j\|_{X^s_{T}}\leq CT^{\kappa}\prod_{j=1}^{k}\|u_j\|_{X^s_{T}}\,.
\end{multline}
(we recall again that the dependence on $\eps$ and $r$ of $X^s_T$ is implicit).
\end{prop}
\begin{proof}
Recall that by \eqref{S_NN}, $S_N=\chi\big(\frac{H}{2N+1}\big)$. Therefore
the second inequality is a consequence of the first as 
thanks to Proposition~\ref{prop.cont} the map $S_N$ is bounded on $X^s_T$ uniformly in $N$. 

Let us now prove the first inequality.
Consider a classical Littlewood-Paley decomposition of unity with respect to $H$,
\begin{equation}\label{LP}
{\rm Id}=\sum_{N}\Delta_N,
\end{equation}
where the summation is taken over dyadic integers $N=2^k$, $\Delta_0=\psi_0(\sqrt{H})$ and for $N\geq 1$,
$\Delta_N=\psi(\sqrt{H}/N)$, where $\psi_0,\psi$ are suitable $C_0^\infty(\R)$ functions (the support of $\psi$ does
not meet zero).
Estimate (\ref{multi_1}) is a consequence of the following localized version of it.
\begin{lemm}\label{multi_loc}
For any $\delta >0$, there exists $C>0$ such that for $N_1\leq\cdots\leq N_k$,
$$
\|\Delta_{N_1}(u_1) \cdots \Delta_{N_k}(u_{k})\|_{L^1_{T}L^2}
\leq CT^{\kappa}N_{k}^{-s+ \delta}\prod_{j=1}^{k}\|u_j\|_{X^s_{T}}\,.
$$
\end{lemm}
Let us now explain how Lemma~\ref{multi_loc} implies Proposition~\ref{multi}. 
Using the definition of ${\mathcal H}^s$, after performing (\ref{LP}), we can write
\begin{multline*}
\|u_1\cdots u_{k}\|_{L^1_{T}{\mathcal H}^{s-\eta}}
\leq 
\\
C\sum_{M}\sum_{N_1\leq \cdots\leq N_k}
\sum_{\sigma\in S_k}
(1+M)^{s-\eta}\|\Delta_{M}\big(\Delta_{N_1}(u_{\sigma(1)}) \cdots \Delta_{N_k}(u_{\sigma(k)})\big)\|_{L^1_{T}L^2}\,.
\end{multline*}
We now observe that Proposition~\ref{multi} is a consequence of Lemma~\ref{multi_loc} (with the choice $\delta = \eta/2$) and the 
following statement, applied with $\alpha>0$ small enough.
\begin{lemm}\label{reste}
Let $\alpha>0$. For every $K>0$, there exists $C>0$ such that for $M\geq N_k^{1+\alpha}$, $N_1\leq \cdots \leq N_k$,
\begin{equation}\label{ipp}
\|\Delta_M\big(\Delta_{N_1}(u_1) \cdots \Delta_{N_k}(u_{k})\big)\|_{L^1_{T}L^2}
\leq CT(1+M)^{-K}\prod_{j=1}^{k}\|u_j\|_{X^s_{T}}\,.
\end{equation}
\end{lemm}
Let us observe that in \cite{BGT}, a similar stronger property (assuming only $M\geq D N_k$, $D\gg1$)  in the context of the analysis on a compact Riemannian 
manifold is proved, the projectors $\Delta_N$ being replaced by the corresponding objects associated to
the Laplace-Beltrami operator. In the context of our analysis below the argument is much simpler
compared to \cite{BGT}.
\begin{proof}[Proof of Lemma~\ref{reste}]
Since the space $X^s_{T}$ is embedded in $L^\infty_{T}{\mathcal H}^{-\varepsilon/10}$
(which is the only $L^2$ type component of our resolution space),
by duality and summing of geometric series 
the bound (\ref{ipp}) is a consequence of the eigenfunction bound
\begin{multline}\label{ipp_bis}
\forall\,K>0,\,
\exists\, C_{K}\,:\,\forall\, N_i\geq 1,\, \forall\, M\geq 1,\,\forall\, n_i,\,m\,:\,
N_i\leq \sqrt{2n_i+1}\leq 2N_i,
\\ 
M\leq \sqrt{2m+1}\leq 2M,\quad
\big|\int_{\R}h_{n_1}\cdots h_{n_k} h_{m}\big|\leq C_{K}(1+M)^{-K}
\end{multline}
(the argument is trivial in the time variable).
By writing $h_{m}=\frac{1}{(2m+1)^j} H^{j}h_{m}$, we make integrations by parts in the left hand-side of 
(\ref{ipp_bis}). Starting from the definition of $h_n$ 
\eqref{formula} we have the relations
\begin{equation}\label{derive_hermite}
h'_{n}(x)=\sqrt{\frac{n}2}\,h_{n-1}(x)-\sqrt{\frac{n+1}2}\,h_{n+1}(x)
\end{equation}
and
\begin{equation*}
x\,h_{n}(x)=h_{n}'(x)+\sqrt{2(n+1)}\,h_{n+1}(x),
\end{equation*}
which implies the bound (for $p\geq 2$),
\begin{equation}\label{crude}
\|x^{k_1}\partial_{x}^{k_2}h_{n}\|_{L^p(\R)}\leq C_{k_1,k_2}(1+|n|)^{\frac{k_1+k_2}{2}}\,.
\end{equation}
Using (\ref{crude}) (applied when extending the powers of $H$) we obtain that the left hand-side of 
(\ref{ipp_bis}) is bounded by $C_j(N_{k}/M)^{2j}$ which implies the needed bound thanks to 
our restriction on $M$. This completes the proof of Lemma~\ref{reste}.
\end{proof}
It remains to prove Lemma~\ref{multi_loc}.
\begin{proof}[Proof of Lemma~\ref{multi_loc}]
We take  the parameters $\eps$ and $r$ in the scope of applicability of 
Lemma~\ref{cfg} and Lemma~\ref{linear}.
By introducing artificially the weight $\langle x\rangle^{s+\eps/4}$, using the H\"older inequality
and Proposition~\ref{prop.cont}, we can write
\begin{multline}\label{jmb}
\|\Delta_{N_1}(u_1) \cdots \Delta_{N_k}(u_{k})\|_{L^1_{T}L^2}
\leq
\\
\leq C
\|\langle x\rangle^{-s-\eps/4}\Delta_{N_k}(u_{k})\|_{L^2_{T}L^2}
\prod_{j=1}^{k-1}\|
\langle x \rangle^{\frac{s+\eps/4}{k-1}}u_{j}\|_{L^{2(k-1)}_{T}L^\infty}\,.
\end{multline}
We now estimate the right hand-side of (\ref{jmb}).
First, we observe that there exists $\kappa>0$ such that for $j=2,\cdots k$,
using the Sobolev inequality, and the boundedness on $L^r$ ($1<r<\infty$) of zero-th order pseudo-differential operators, we can write
\begin{equation}\label{jmb2}
\|\langle x \rangle^{\frac{s+\eps/4}{k-1}}u_{j}\|_{L^{2(k-1)}_{T}L^\infty}
\\
\leq CT^{\kappa}
\|u_{j}\|_{L^{2(k-1)+\eps}_{T}{\mathcal W}^{\frac{s+\eps}{k-1},r}}
\leq
CT^{\kappa}\|u_j\|_{X^s_{T}}\,,
\end{equation}
provided $\frac{3\eps}{4(k-1)}r>1$. Next, using similar arguments as in the proof of Lemma~\ref{reste}, we obtain the following statement.
\begin{lemm}\label{cambera}
For any $\kappa>0$ and any $K$, there exists $C$ such that for any $M$ satisfying $M\leq N^{1-\kappa}$, we have 
$$\|\Delta_M \langle x\rangle^{-s-\varepsilon/4}\Delta_{N}(u)\|_{L^2_{T}L^2} \leq C (1+M+N)^{-K} \|u\|_{X^s_T}.
$$
\end{lemm}
As a consequence of Lemma~\ref{cambera}, we obtain for arbitrarily small $\kappa >0$,
  \begin{multline*}
\|\langle x\rangle^{-s-\varepsilon/4}\Delta_{N_k}(u_{k})\|_{L^2_{T}L^2}\\
\begin{aligned}
& \leq
\sum_{ M \geq N_k^{1-\kappa}}\|\Delta_{M}(\langle x\rangle^{-s-\varepsilon/4}\Delta_{N_k}(u_{k}))\|_{L^2_{T}L^2}+ CN_k^{-s}\|u_k\|_{X^s_T}
\\
& \leq
C\sum_{ M \geq N_k^{1-\kappa}}\|\Delta_{M}\frac{\sqrt{H}^{s}}{M^s} (\langle x\rangle^{-s-\varepsilon/4}
\Delta_{N_k}(u_{k}))\|_{L^2_{T}L^2}+ CN_k^{-s}\|u_k\|_{X^s_T}
\\
& \leq 
CN_{k}^{-s+ \kappa s}
\|\sqrt{H}^{s}(\langle x\rangle^{-s-\varepsilon/4}
\Delta_{N_k}(u_{k}))\|_{L^2_{T}L^2}+ CN_k^{-s}\|u_k\|_{X^s_T}\,.
\end{aligned}
\end{multline*}
Using Proposition~\ref{prop.cont}, we can write
\begin{multline*}
\|\sqrt{H}^{s}(\langle x\rangle^{-s-\varepsilon/4}
\Delta_{N_k}(u_{k}))\|_{L^2_{T}L^2}
\leq
C\|u_k\|_{_{L^{2}_{T}{\mathcal H}^{s}_{\langle x \rangle^{-s-\varepsilon/4}}}}
\\
+\,\,
\|[\sqrt{H}^{s},\langle x\rangle^{-s-\varepsilon/4}]
\Delta_{N_k}(u_{k})
\|_{L^2_{T}L^2}\,.
\end{multline*}
In order to estimate the commutator contribution, we shall use 
the Weyl-H\"ormander pseudo-differential calculus associated to the metric
\begin{equation}\label{metric}
\frac{dx^2}{1+x^2}+\frac{d\xi^2}{1+\xi^2}.
\end{equation}
The symbol classes $S(\mu,m)$ associated to (\ref{metric}) are the spaces of smooth functions on $\R^2$
satisfying the bounds
$$
|\partial_{x}^\alpha\partial_\xi^\beta a(x,\xi)|
\leq C_{\alpha,\beta}\langle x\rangle^{\mu-\alpha}\langle\xi\rangle^{m-\beta}\,.
$$
Then we have (see cf. \cite{Hormander}, Section 18.5, \cite{Robert} or \cite{Bouclet})
that if $a_1\in S(\mu_1,m_1)$ and $a_2\in S(\mu_2,m_2)$ with corresponding operators ${\rm Op}(a_1)$
and ${\rm Op}(a_2)$ then the commutator $[{\rm Op}(a_1),{\rm Op}(a_2)]$ is a pseudo-differential operator
with symbol in $S(\mu_1+\mu_2-1,m_1+m_2-1)$. Using this fact, we obtain that
$$
\|[\sqrt{H}^{s},\langle x\rangle^{-s-\varepsilon/4}]
\Delta_{N_k}(u_{k})\|_{L^2_{T}{\mathcal H}^{\eps/4}}
\leq C\|\Delta_{N_k}(u_{k})\|_{L^2_{T}L^2}
$$
and by duality
$$
\|[\sqrt{H}^{s},\langle x\rangle^{-s-\varepsilon/4}]
\Delta_{N_k}(u_{k})
\|_{L^2_{T}L^2}
\leq C\|\Delta_{N_k}(u_{k})\|_{L^2_T{\mathcal H}^{-\eps/4}}\leq CT^{\frac{1}{2}}\|u_k\|_{X^s_{T}}\,.
$$ 
Therefore, we obtain the bound
\begin{equation}\label{jmb3}
\|\langle x\rangle^{-s-\varepsilon/4}\Delta_{N_k}(u_{k})\|_{L^2_{T}L^2}
\leq CN_k^{-s+\kappa s}\|u_k\|_{X^s_{T}}\,.
\end{equation}
We now collect (\ref{jmb}), (\ref{jmb2}) and (\ref{jmb3}) 
in order to complete the proof of Lemma~\ref{multi_loc}.
\end{proof}
This completes the proof of Proposition~\ref{multi}.
\end{proof}
%%%%%%%%%%%%%%%%%%%%%%%%%%%%%%%%%%%%%%%%%%%%%%%%%%%%%%%%%%%%%%%%%%%%%%%%%%%%%%%%%%%%%%%%%%%%%%%
%%%%%%%%%%%%%%%%%%%%%%%%%%%%%%%%%%%%%%%%%%%%%%%%%%%%%%%%%%%%%%%%%%%%%%%%%%%%%%%%%%%%%%%%%%%%%%%
\subsection{Further properties of $Y^s$ with respect to the measure $\rho$}
From now each time we invoke the space $Y^s$, we mean that $s$ satisfies \eqref{restriction_s} and $\eps$ and $r$ are in the scope of applicability of Lemma~\ref{cfg}, 
Lemma~\ref{linear} and Proposition~\ref{multi}.
Let us next define some auxiliary spaces. Let $\widetilde{Y}^s$ be defined by 
$$
\widetilde{Y}^s=\big\{u\in {\mathcal H}^{-\varepsilon/20}\,:\, e^{-itH}(u)\in L^{2(k-1)+\eps}_{2\pi}
{\mathcal W}^{\frac{s+\eps}{k-1},r}
\cap L^{2}_{2\pi}{\mathcal H}^{s}_{\langle x \rangle^{-s-\eps/5}}\big\},
$$ 
equipped with the natural norm.
The remaining part of this section is devoted to three lemmas needed in the proof of Theorem~\ref{thm4}.
Using the density in $L^p$, $1\leq p<\infty$ of the Schwartz class ${\mathcal S}(\R)$,
as a consequence of Proposition~\ref{prop.cont} (and the fact that the result is straightforward if $f\in \mathcal{S}$), we have the following statement.
\begin{lemm}\label{Stone}
For every $f\in Y^s$, $\|(1-S_N)(f)\|_{Y^s}=o(1)_{N\rightarrow \infty}$.
A similar statement holds for $\tilde{Y}^s$.
\end{lemm}
One can easily see that the analysis of Lemma~\ref{lemm2} and Lemma~\ref{lem.smooth} 
implies that $\varphi_N$, defined by \eqref{def.phin} 
is a Cauchy sequence in $L^2(\Omega;Y^s)$ and thus we may see the measures $\mu$ and $\rho$ 
as finite Borel measures on $Y^s$. We deduce, thanks to Lemma~\ref{cfg}
\begin{lemm}\label{weak}
There exist $C,c>0$ such that
$$
\rho(u\,:\, \|u\|_{Y^s}>\lambda)+\rho(u\,:\, \|u\|_{\tilde{Y}^s}>\lambda)
\leq Ce^{-c\lambda^2}\,.
$$
\end{lemm}
We also have the following statement.
\begin{lemm}\label{kompakt}
Assume that $s<s'<s+\frac \eps {20}$. Then we have that $\tilde{Y}^s\subset Y^s$ and the embedding is compact.
In particular, thanks to Lemma~\ref{weak}, for every $\delta>0$ there exists a compact $K$ of $Y^s$ such that
$\rho(Y^s)-\rho(K)<\delta$.
\end{lemm}
Notice that  as soon  we gain some positive power in $H$, we gain compactness because powers of $H$ controls both powers of $|D_x|$ and of $x$. As a consequence,  the assumption $s'>s$ ensures that we have compactness in terms of $x$ derivatives and weights in $\langle x \rangle$ for the second norm, whereas it ensures compactness in terms of derivatives in the third norm, while the assumption $s'< s + \frac \eps {20}\Rightarrow s' + \frac \eps 5 < s + \frac \eps 4 $ ensures compactness  in terms of weights in $\langle x \rangle$ in this last norm.  Finally, since the second and the third term in the definition of $Y^s$ are defined in terms of the free evolution, we may exchange some saving derivatives  in $H$
for some  compactness in time. We omit the details.
%%%%%%%%%%%%%%%%%%%%%%%%%%%%%%%%%%%%%%%%%%%%%%%%%%%%%%%%%%%%%%%%%%%%%%%%%%%%%%%%%%%%%%%%%%%%%%%
%%%%%%%%%%%%%%%%%%%%%%%%%%%%%%%%%%%%%%%%%%%%%%%%%%%%%%%%%%%%%%%%%%%%%%%%%%%%%%%%%%%%%%%%%%%%%%
\subsection{Local well-posedness results}
Using the results of the previous subsections, we can now get local well-posedness 
results (uniform with respect to the parameter $N$) for 
\begin{equation}\label{eq_N}
(i\partial_t-H)u=\kappa_{0}S_N\big(|S_Nu|^{k-1}S_Nu\big),\quad u(0,x)=u_{0}(x)\in E_{N}.
\end{equation}
Here is a precise statement.
\begin{prop}\label{wp_N}
There exist $C>0$, $c\in (0,1)$, $\gamma>0$ such that
for every $A\geq 1$ if we set $T=cA^{-\gamma}$ then for every $N\geq 1$, every $u_0\in E_N$ satisfying
$\|u_0\|_{Y^s}\leq A$ there exists a unique solution of (\ref{eq_N})
on the interval $[-T,T]$ such that
$\|u\|_{X^s_{T}}\leq A+A^{-1}.$ 
In addition for $t\in [0,T]$,
\begin{equation}\label{soto}
\|u(t)\|_{Y^s}\leq A+A^{-1}.
\end{equation}
Moreover, if $u$ and $v$ are two solutions with data $u_0$ and $v_0$ 
respectively, satisfying $\|u_0\|_{Y^s}\leq A$, $\|v_0\|_{Y^s}\leq A$ then
$\|u-v\|_{X^s_{T}}\leq C\|u_0-v_0\|_{Y^s}$ and for $t\in [0,T]$,
$$
\|u(t)-v(t)\|_{Y^s}\leq C\|u_0-v_0\|_{Y^s}\,.
$$
\end{prop}
\begin{proof}
We rewrite (\ref{eq_N}) as the integral equation
$$
u(t)=e^{-itH}(u_0)+\kappa_0 \int_{0}^{t}e^{-i(t-\tau)H}\big(
S_N\big(|S_Nu(\tau)|^{k-1}S_Nu(\tau)\big)\big)d\tau\equiv
\Phi_{u_0}(u).
$$
Using Lemma~\ref{linear} and Proposition~\ref{multi}, we infer the bounds
\begin{equation}\label{mica1}
\|\Phi_{u_0}(u)\|_{X^s_{T}}\leq \|u_0\|_{Y^s}+CT^{\kappa}\|u\|_{X^s_{T}}^{k}
\end{equation}
and (after some algebraic manipulations on $|u|^{k-1}u-|v|^{k-1}v$)
\begin{equation}\label{mica2}
\|\Phi_{u_0}(u)-\Phi_{u_0}(v)\|_{X^s_{T}}\leq CT^{\kappa}\|u-v\|_{X^s_{T}}
(\|u\|_{X^s_{T}}^{k-1}+\|u\|_{X^s_{T}}^{k-1}).
\end{equation}
Therefore if we choose $T$ as $T=cA^{-K}$ with $c\ll 1$ and $K> (k+10)/\kappa$, the estimates \eqref{mica1}
and \eqref{mica2} yield that
the map $\Phi_{u_0}$ is a contraction on the ball of radius $2A$ and centered at the origin of $X_T^s$. 
The fixed point of this contraction is a 
solution of (\ref{eq_N}). The uniqueness and the estimate on the difference of two solutions is a consequence 
of Proposition~\ref{multi}. 
Finally coming back to (\ref{mica1}), we infer that the solution satisfies
\begin{equation}\label{soto_bis}
\|u\|_{X^s_T}=\|\Phi_{u_0}(u)\|_{X^s_T}\leq A+ Cc^{\kappa}(1+A)^{-\kappa K}A^{k}\leq A+A^{-1},
\end{equation}
for $c$ small enough and
by possibly taking $K$ slightly larger (replacing $K$ by $K+1/\kappa$ for instance).
Let us now prove \eqref{soto}. Using Lemma~\ref{linear} ( $\eta$ in the scope of its applicability), Proposition~\ref{multi} and (\ref{soto_bis}), we obtain 
of Lemma~\ref{linear},
\begin{eqnarray*}
\|u(t)\|_{Y^s}& \leq & \|u_0\|_{Y^s}+CT^{\kappa}\||u|^{k-1} u\|_{L^1_T{\mathcal H}^{s-\eta}}
\\
& \leq &
A+CT^{\kappa}\|u\|_{X^s_{T}}^{k}
\\
& \leq &
A+A^{-1}.
\end{eqnarray*}
This completes the proof of Proposition~\ref{wp_N}.
\end{proof}
Let us remark that the existence statement in Proposition~\ref{wp_N} is not of importance (indeed see the next
section for a global existence statement). The important point is the uniformness 
with respect to $N$ of the bounds obtained.
Similarly, we can also obtain a well-posedness result for the original problem
\begin{equation}\label{eq}
(i\partial_t-H)u=\kappa_{0}|u|^{k-1}u,\quad u(0,x)=u_0(x)\in Y^s.
\end{equation}
\begin{prop}\label{wp}
Then there exist $C>0$, $c\in (0,1)$, $\gamma>0$ such that
for every $A\geq 1$ if we set $T=cA^{-\gamma}$ for every $u_0\in Y^s$ satisfying
$\|u_0\|_{Y^s}\leq A$ there exists a unique solution of (\ref{eq}) on $[-T,T]$ such that
$
\|u\|_{X^s_{T}}\leq A+A^{-1}.
$
In addition for $t\in [0,T]$,
\begin{equation*}
\|u(t)\|_{Y^s}\leq A+A^{-1}.
\end{equation*}
Moreover, if $u$ and $v$ are two solutions with data $u_0$ and $v_0$ respectively, satisfying
$\|u_0\|_{Y^s}\leq A$, $\|v_0\|_{Y^s}\leq A$ then
$
\|u-v\|_{X^s_{T}}\leq C\|u_0-v_0\|_{Y^s}.
$
\end{prop}
The proof of Proposition~\ref{wp} is essentially the same as that of Proposition~\ref{wp_N} and hence will be omitted.
%%%%%%%%%%%%%%%%%%%%%%%%%%%%%%%%%%%%%%%%%%%%%%%%%%%%%%%%%%%%%%%%%%%%%%%%%%%%%%%%%%%%%%%%%%%%%%%%%%
%%%%%%%%%%%%%%%%%%%%%%%%%%%%%%%%%%%%%%%%%%%%%%%%%%%%%%%%%%%%%%%%%%%%%%%%%%%%%%%%%%%%%%%%%%%%%%%%%%
\section{Global well-posedness }\label{Section.Global}
In this section, we prove the global existence results for a full measure set for \eqref{nlsk}.
Moreover this set will be reproduced by the flow which is a key element in the measure invariance 
argument of the next section.
\subsection{Hamiltonian structure of the approximate problem}
\noindent Here we consider again the problem 
\begin{equation}\label{odekbis}
(i\partial_t-H)u=\kappa_{0}S_N\big(|S_Nu|^{k-1}S_Nu\big),\quad u(0,x)=\Pi_{N}(u(0,x))\in E_{N},
\end{equation}
with $\kappa_{0}=\pm 1$ if $k=3$ and $\kappa_{0}=1$ if $k\geq 5$.

\noindent For $u\in E_{N}$, write 
$$
u=\sum_{n=0}^{N}c_n h_{n}=\sum_{n=0}^{N}(a_n+ib_n)h_{n},\quad a_n,b_n\in \R.
$$
Then we have the following result.
\begin{lemm}\label{Hamilton}
Set
\begin{multline*}
J(a_0,,\cdots,a_N,b_0,\cdots, b_{N})=
\frac{1}{2}\sum_{n=0}^{N}\lambda_{n}^{2}(a_n^2+b_n^2)
\\
+\frac{\kappa_{0}}{k+1}
\|S_N\big(\sum_{n=0}^{N}(a_n+ib_n) 
h_{n}\big)\|_{L^{k+1}(\R)}^{k+1}\,.
\end{multline*}
The equation \eqref{odekbis} is a Hamiltonian ODE of the form
\begin{equation}\label{hamm}
\dot{a}_n=\frac{\partial J}{\partial b_n},\quad
\dot{b}_n=-\frac{\partial J}{\partial a_n},\quad 0\leq n\leq N.
\end{equation}
In particular $J$ is conserved by the flow. Moreover the mass
\begin{equation}\label{mass}
\|u\|^{2}_{L^{2}(\R)}=\sum_{n=0}^{N}(a_n^2+b_n^2)
\end{equation}
is conserved under the flow of \eqref{odekbis}. 
As a consequence, \eqref{odekbis} has a well-defined global flow $\tilde{\Phi}_{N}$.
\end{lemm}
\begin{proof}
The proof of (\ref{hamm}) is straightforward.
Let us next show the $L^2$ conservation.
Multiply the equation \eqref{odekbis} with $\ov{u}$ and integrate over $\R$.
First, by an integration by parts, we have 
\begin{equation}\label{calc1}
-\int_{\R}\ov{u}\,Hu=\int_{\R}|H^{1/2}u|^{2}\in \R.
\end{equation}
Then by \eqref{commut}, we deduce that
\begin{eqnarray}\label{calc2}
\int_{\R}S_N\big(|S_Nu|^{k-1}S_Nu\big)\ov{u}&=&\int_{\R}S_N\big(|S_Nu|^{k-1}S_Nu\big)S_N\ov{u}\\
&=&\int_{\R}\big(|S_Nu|^{k-1}S_Nu\big)S_N\ov{u}\in \R.\nonumber 
\end{eqnarray}
Hence, from \eqref{calc1} and \eqref{calc2} we infer that 
\begin{equation*}
\frac{\text{d}}{\text{d}t}\|u\|^{2}_{L^{2}(\R)}=0.
\end{equation*}
This completes the proof of Lemma~\ref{Hamilton}.
\end{proof}
\noindent Denote by $\Phi_{N}(t)\,:\,E_{N}\longrightarrow E_{N}$ the flow of the ode \eqref{odekbis}.
We now state an invariance result which holds both in the defocusing and in the focusing cases.
%%%%%%%%%%%%%%%%%%%%%%%%%%%%%%%%%%%%
\begin{prop}\label{rhon}
The measure $\tilde{\rho}_{N}$ defined by (\ref{tilda}) (or (\ref{tilda_pak})) 
is invariant under the flow $\tilde{\Phi}_{N}$ of \eqref{odekbis}.
\end{prop}
%%%%%%%%%%%%%%%%%%%%%%%%%%%%%%%%%%%%%%
\begin{proof}
The proof is based on the Liouville theorem which we recall below.
\begin{lemm}\label{gentchev}
Consider the ODE $\dot{x}(t)=F(t,x(t))$, $x(t)\in\R^n$ with a local flow $\Phi(t)$. Suppose also
that $F$ is divergence free, i.e.
$
\sum_{j=1}^{n}\partial_j F=0
$
($\partial_j$ being the derivative with respect to the $(j+1)$'th variable of $F$).
Then the Lebesgue measure of $\R^n$ is invariant under the flow $\Phi(t)$.
\end{lemm}
Observe that the ODE's in the scope of applicability of the Liouville theorem are not necessarily autonomous.
Let us now return to the proof of Proposition~\ref{rhon}.
By Lemma~\ref{gentchev}, 
the measure $\text{d}a\text{d}b=\prod_{n=0}^{N}a_{n}b_{n}$ is invariant under $\Phi_{N}$. 
Then, as the Hamiltonian $J$ is conserved, the measure 
$
d_N^{-1}\e^{-J}\prod_{n=0}^{N}\text{d}a_{n}\text{d}b_{n}
$
is also invariant by the flow of \eqref{odekbis}. 
This completes the proof in the defocusing case.
A similar argument applies in the focusing case by invoking the $L^2$ conservation.
This completes the proof of Proposition~\ref{rhon}.
\end{proof}
Let us now decompose the space $\mathcal{H}^{- \sigma}( \R)= E_N^\perp \oplus E_N$, and denote by $\Phi_N(t)= (e^{itH}, \tilde{\Phi}_N(t))$ the flow of the equation  
\begin{equation}\label{odek*}
(i\partial_t-H)u=\kappa_{0}S_N\big(|S_Nu|^{k-1}S_Nu\big),\quad u(0,x)=(u_0^N, u_{0,N}) \in E_N^\perp \oplus E_{N},
\end{equation}
\begin{coro}\label{cor.inv}
The measure $\rho_N$ is invariant under the flow $\Phi_N(t)$. 
\end{coro}
Indeed, it is clear for product sets $A= A^N \times A_N, A^N \subset E_N^\perp, A_N \subset E_N$ and these sets generate the Borelian $\sigma$-algebra.
%%%%%%%%%%%%%%%%%%%%%%%%%%%%%%%%%%%%%%%%%%%%%%%%%%%%%%%%%%%%%%%%%%%%%%%%%%%%%%%%%%%%%%%%%%%%%%%%%%%
%%%%%%%%%%%%%%%%%%%%%%%%%%%%%%%%%%%%%%%%%%%%%%%%%%%%%%%%%%%%%%%%%%%%%%%%%%%%%%%%%%%%%%%%%%%%%%%%%%%
\subsection{Global existence }
Here we show that the problem \eqref{nlsk} is globally well-posed on a set of full $\rho$ measure.
Our first result gives bounds (independent of $N$) on the solution of the approximate equation \eqref{odek*}.
\begin{prop}\label{215}
There exists a constant $C>0$ such that for all $i,N\in \N^{*}$, there exists a $\tilde\rho_{N}$ 
measurable set $\tilde{\Sigma}_{N}^{i}\subset E_{N}$ so that for all $i,N\in \N^{*}$
\begin{equation*}
\tilde{\rho}_{N}(E_{N}\backslash \tilde{\Sigma}_{N}^{i})\leq 2^{-i}.
\end{equation*}
For all $f\in \tilde{\Sigma}_{N}^{i}$ and $t\in \R$
\begin{equation}\label{bonn1}
\|\tilde \Phi_{N}(t)f\|_{Y^s}\leq C\big( i+\log(1+|t|) \big)^{\frac12}.
\end{equation}
Moreover, there exists $c>0$ such that for every $t_0$, every $i\geq 1$ and $N\geq 1$,
\begin{equation}\label{bonn2}
\tilde \Phi_N(t_0) ( \tilde{\Sigma}^{i}_N) \subset \tilde{\Sigma}^{i+ [c\log(|t_0|+1)] +3}_N.
\end{equation}
\end{prop}
The property \eqref{bonn2} allows to simplify the construction 
of a set invariant by the limit flow, compared to a similar situation 
in \cite{Tzvetkov2,BT1}.
\begin{proof}
We set, for $i,j$ integers $\geq 1$, 
$$
\tilde{B}_{N}^{i,j}(D)\equiv
\big\{u\in E_{N}\,:\,\|u\|_{Y^s}\leq D(i+j)^{\frac{1}{2}}\big\},
$$
where the number $D\gg 1$ (independent of $i,j,N$) will be fixed later. 
Thanks to Proposition~\ref{wp_N}, there exist $c>0$, $\gamma>0$ only depending on $s$ 
such that if we set $\tau \equiv c D^{-\gamma}(i+j)^{-\gamma/2}$ then for every $t\in[-\tau,\tau]$,
\begin{multline}\label{flotN}
\tilde{\Phi}_{N}(t)\big(\tilde{B}_{N}^{i,j}(D)\big)\subset 
\big\{u\in E_{N}\,:\,
\|u\|_{Y^s}
\leq D(i+j)^{\frac{1}{2}}+ D^{-1} (i+j)^{-\frac 1 2}
\\
\leq D (i+j+1)^{\frac 12}\big\}\, ,
\end{multline}
provided $D\gg 1$, independently of $i,j$. Following \cite{Bourgain1}, we set
$$
\tilde{\Sigma}_{N}^{i,j}(D)\equiv
\bigcap_{k=-[2^{j}/\tau]}^{[2^{j}/\tau]}\tilde{\Phi}_{N}(-k\tau)(\tilde{B}_{N}^{i,j}(D))\, ,
$$
where $[2^{j}/\tau]$ stays for the integer part of $2^{j}/\tau$. 
Notice that thanks to \eqref{flotN}, we obtain that the solution of \eqref{odek*} with data $f\in\tilde{\Sigma}_{N}^{i,j}(D)$ satisfies
\begin{equation}\label{eqflot}
\big\|\tilde{\Phi}_{N}(t)(f)\big\|_{Y^s}
\leq D(i+j+1)^{\frac{1}{2}},\quad |t|\leq 2^{j}\,.
\end{equation}
Indeed, for $|t|\leq 2^{j}$, we can find an integer $k\in [-[2^{j}/\tau],[2^{j}/\tau]]$ and 
$\tau_1\in [-\tau,\tau]$ so that $t=k\tau+\tau_1$ and thus 
$u(t)=\tilde \Phi_{N}(\tau_1)\big(\tilde \Phi_{N}(k\tau)(f)\big)$.
Since $f\in\tilde{\Sigma}_{N}^{i,j}(D)$ implies that $\tilde \Phi_{N}(k\tau)(f)\in 
\tilde B_{N}^{i,j}(D)$, we can apply \eqref{flotN} and arrive at \eqref{eqflot}.

By Proposition \ref{rhon}, the measure $\tilde{\rho}_{N}$ is invariant by the flow $\tilde{\Phi}_{N}$. Hence
\begin{eqnarray*}
\tilde{\rho}_{N}\big(E_{N}\backslash\tilde{\Sigma}_{N}^{i,j}(D)\big)
&\leq &
(2[2^{j}/\tau]+1)\tilde{\rho}_{N}\big(E_N\backslash \tilde B_{N}^{i,j}(D)\big)\\
&\leq &
C2^{j}D^{\gamma}(i+j)^{\gamma/2}\tilde{\rho}_{N}\big(E_N\backslash \tilde B_{N}^{i,j}(D)\big)\,.
\end{eqnarray*}
Now, by the large deviation bounds of Lemma~\ref{cfg} 
\begin{equation}\label{zvez}
\tilde{\rho}_{N}\big(E_{N}\backslash\tilde{\Sigma}_{N}^{i,j}(D)\big)\leq
C2^{j}D^{\gamma}(i+j)^{\gamma/2}e^{-cD^2(i+j)}\leq 2^{-(i+j)},
\end{equation}
provided $D\gg 1$, independently of $i,j,N$.

Next, we set
$$
\tilde{\Sigma}_{N}^{i}=\bigcap_{j= 1}^{\infty}\tilde{\Sigma}_{N}^{i,j}(D)\,.
$$
Thanks to (\ref{zvez}),
$
\tilde{\rho}_{N}(E_{N}\backslash \tilde{\Sigma}_{N}^{i})\leq 2^{-i}\,.
$
In addition, using \eqref{eqflot}, we get that there exists $C$ such that for every $i$, every
$N$, every $f\in \tilde{\Sigma}_{N}^{i}$, every $t\in \R$,
\begin{equation*} 
\big\|\tilde \Phi_{N}(t)(f)\big\|_{Y^s}
\leq C(i+2+\log(1+|t|))^{\frac{1}{2}}\,.
\end{equation*}
Indeed for $t\in \R$ there exists $j\in\N$ such that $2^{j-1}\leq 1+|t|\leq 2^j$ and we apply 
\eqref{eqflot} with this $j$.
This proves \eqref{bonn1}. 

Let us now turn to the proof of \eqref{bonn2}.
Consider $ f \in \tilde{\Sigma}^i_N= \cap _{j\geq 1 } \tilde{\Sigma}^{i,j}_N$. Denote by $j_0$ the integer part of $2+ \frac {\log(1+|t_0|)} {\log(2)}$.
According to~\eqref{eqflot}, as soon as $ j\geq j_0$, we have $|t_0|\leq 2^{j-1}$, 
and for any $|t|\leq 2^{j-1}$
\begin{equation}
\label{eq.propagation}
\tilde{\Phi}_N(t) \tilde{\Phi}_N(t_0)f \in \tilde B^{i+2,j-1}_N(D)\subset \tilde B^{i+ j_0+2, j-1}_N(D),
\end{equation}
which implies that 
$$ \tilde{\Phi}_N(t_0)f \in\tilde{\Sigma}^{i+j_0+2, j-1}_N(D),\quad \;\forall \, j> j_0.
$$
On the other hand, the trivial relation (for $j_0 -k > 0$, i.e. $k=1,2,\cdots j_0-1$)
$$ \tilde{B}^{i+2,j_0}_N(D)= \tilde{B}^{i+2+k, j_0-k}_N(D)\subset \tilde{B}^{i+ j_0+2, j_0-k}_N(D),
$$
and~\eqref{eq.propagation} (applied with $j=j_0+1$) implies that for $j< j_0$ and $|t|\leq 2^{j}\leq 2^{j_0-1}$,
$$ \tilde{\Phi}_N(t)\tilde{\Phi}_N(t_0)f \in \tilde{B}^{i+j_0+2, j}_N(D),\quad \;\forall \,1\leq j< j_0,
$$
and consequently 
$$ \tilde{\Phi}_N(t_0)f \in\tilde{\Sigma}^{i+ j_0+2, j}_N(D),\quad \;\forall \, j\geq 1.
$$ 
This proves \eqref{bonn2} and therefore the proof of Proposition~\ref{215} is completed.
\end{proof}
%%%%%%%%%%%%%%%%%%%%%%%%%%%%%%%%%%%%%%%%%%%%%%%%%%%%%%%%%%%%%%%%%%%%%%%%%%%%%%%%%%%%%%%%%%%%%%%%
%%%%%%%%%%%%%%%%%%%
%%%%%%%%%%%%%%%%%%%%%%%%%%%%%%%%%%%%%%%%%%%%%%%%%%%%%%%%%%%%%%%%%%%%%%%%%%%%%%%%%%%%%%%%%%%%%%%
%%%%%%%%%%%%%%%%%%%%%%%
\noindent 
For integers $i\geq 1$ and $N \geq 1$, we define the
cylindrical sets
$$
\Sigma_{N}^{i}\equiv\big\{u\in \mathcal{H}^{-\sigma}\,:\, \Pi_{N}(u)\in \tilde\Sigma_{N}^{i}\big\}.
$$
Next, for $i\geq 1$, we set
\begin{equation*}
\Sigma^{i}=\big\{ u\in \mathcal{H}^{-\sigma}: \exists N_{k}, \lim_{k\rightarrow + \infty} N_k = \infty, 
\exists\, u_{N_{k}}\in  {\Sigma}^{i}_{N_{k}}, \lim_{k\rightarrow + \infty} \|S_{N_k}u_{N_{k}}- u\|_{Y^s}=0\big\}.
\end{equation*}
Observe that $\Sigma^{i}$ is a closed subset of $Y^s$. Indeed, assume that there exists $u_{N_k} \in {\Sigma}^{i}_{N_{k}}, \lim_{k\rightarrow + \infty} \|S_{N_k}u_{N_{k}}- u\|_{Y^s}=0$. Then for any $P\in \mathbb{N}$, as soon as $N_k \gg P$, we have 
$$ \|S_P(u_{N_{k}}- u)\|_{Y^s}= \|S_P(S_{N_k}u_{N_{k}}- u)\|_{Y^s}\leq C  \|S_{N_k}u_{N_{k}}- u\|_{Y^s}\rightarrow 0$$
As a consequence, using~\eqref{bonn1} (with $t=0$), we deduce 
\begin{multline}
 \|S_P(u)\|_{Y^s} \leq \limsup_{k\rightarrow + \infty} \|S_P( u_{N_k})\|_{Y^s} =  \limsup_{k\rightarrow + \infty} \|S_P( u_{N_k})\|_{Y^s}\\
 \leq C \sup_{Q} \|S_Q\|_{\mathcal{L}( L^r( \R))}i^{1/2}
 \end{multline}
and passing to the limit $P\rightarrow + \infty$, we deduce 
$$ u\in Y^s,   \|u\|_{Y^s} \leq C'i^{1/2}.$$ The closeness property is clear. Notice also that we have the following inclusions
\begin{equation}\label{inclusion}
\limsup_{N\to +\infty}\Sigma_{N}^{i}=\bigcap_{N=1}^{\infty}\bigcup_{N_{1}=N}^{\infty}\Sigma_{N_{1}}^{i}\subset \Sigma^{i}.
\end{equation}
Indeed, if $u\in \limsup_{N\to +\infty}\Sigma_{N}^{i}$, there exists $N_k\rightarrow+ \infty$ such that 
$$ \Pi_{N_k}(u) \in \tilde \Sigma^i_{N_k},
$$ and the same proof as above shows that 
$$ u\in Y^s,   \|u\|_{Y^s} \leq C'i^{1/2}.$$
Now, we clearly have 
$$ 
\|S_n u - u \|_{Y^s}= o(1) _{n\rightarrow + \infty},
$$
and since $S_n (\Pi_n (u)) = S_n (u)$,  the sequence $u_{N_k}\equiv \Pi_{N_k}(u)$ is the one ensuring that $u\in \Sigma^i$. 
This proves \eqref{inclusion}. As a consequence of \eqref{inclusion}, we get
\begin{equation}\label{kr1}
\rho\big(\Sigma^{i}\big)\geq \rho\big(\limsup_{N\to +\infty}\Sigma_{N}^{i}\big).
\end{equation}
Using Fatou's lemma, we get
\begin{equation}\label{kr2}
\rho(\limsup_{N\rightarrow\infty}\Sigma_{N}^{i})
\geq 
\limsup_{N\rightarrow \infty}\rho(\Sigma_{N}^{i})\,.
\end{equation}
In the defocusing case, consider $G_{N}(u)=\exp(-\frac1{k+1}\|S_Nu\|^{k+1}_{L^{k+1}(\R)})$ and $G(u)=\exp(-\frac1{k+1}\|u\|^{k+1}_{L^{k+1}(\R)})$. 
In the focusing case, let $G_{N}$ be defined by \eqref{cinm} and $G$ by Theorem \ref{prop}. We have that
$$
\rho(\Sigma_{N}^{i})=\int_{\Sigma_{N}^{i}}G(u)\text{d}\mu(u),
$$
and
$$
{\rho}_{N}(\Sigma_{N}^{i})=\int_{\Sigma_{N}^{i}}G_{N}(u)\text{d}{\mu} _N(u)=\int_{\tilde \Sigma_{N}^{i}}G_{N}(u)\text{d}\tilde \mu(u)=\tilde {\rho}_{N}(\tilde \Sigma_{N}^{i}).
$$
Therefore, thanks to \eqref{conv.faible@}, we get
$$
\lim_{N\rightarrow \infty}\big(\rho({\Sigma}_{N}^{i})-\rho_{N}(\Sigma_{N}^{i})\big)=0\,.
$$
Therefore, using Proposition \ref{215} and \eqref{masse}, we obtain
\begin{equation}
\begin{aligned}
\limsup_{N\rightarrow \infty}\rho(\Sigma_{N}^{i})
& = 
\limsup_{N\rightarrow \infty}{\rho}_{N}(\Sigma_{N}^{i})
= 
\limsup_{N\rightarrow \infty}\tilde{\rho}_{N}({\tilde{\Sigma}}_{N}^{i})
\\
& = 
\limsup_{N\rightarrow \infty}\big(\rho_{N}(Y^s)-2^{-i}\big)
=
\rho(Y^s)-2^{-i}.
\end{aligned}
\end{equation}
Collecting the last estimate and \eqref{kr1}, \eqref{kr2}, we obtain that
\begin{equation}\label{kr5}
\rho\big(\Sigma^{i}\big) \geq \rho(Y^s)-2^{-i}.
\end{equation}
%%%%%%%%%%%%%%%%%%%%%%%%%%%%%%%%%%%%%%%%%%%%%%%%%%%%%%%%%%%%%%%%%%%%%%%%%%%%%%%%%%%%%%%%%%%%%%%%%%%%%%%%%%%
%%%%%%%%%%%%%%%%%%%%%%%%%%%%%%%%%%%%%%%%%%%%%%%%%%%%%%%%%%%%%%%%%%%%%%%%%%%%%%%%%%%%%%%%%%%%%%%%%%%%%%%%%%%
%%%%%%%%%%%%%%%%%%%%%%%%%%%%%%%%%%%%%%%%%%%%%%%%%%%%%%%%%%%%%%%%%%%%%%%%%%%%%%%%%%%%%%%%%%%%%%%%%%%%%%%%%%%
\noindent Now, we set
\begin{equation}\label{def.sigma}
\Sigma\equiv\bigcup_{i=1}^{\infty}\Sigma^{i}.
\end{equation}
Then, by \eqref{kr5}, the set $\Sigma$ is of full $\rho$ measure.
It turns out that one has global existence for any initial condition $f\in \Sigma$.
We now state the global existence results for the problem \eqref{eq}.
%%%%%%%%%%%%%%%%%%%%%%%%%%%%%%%%%%%%%%%%%%%%%%%%%%%%%%%%%%%%%%%%%%%%%%%%%%%%%%%%%%%%%%%%
\begin{prop}\label{prop.sigma}
For every integer $i\geq 1$ the local solution $u$ of \eqref{eq} 
with initial condition $f\in \Sigma^{i}$ is globally defined and we shall denote it by $u= \Phi(t) f$. 
Moreover, there exists $C>0$ such that for every $f\in \Sigma^{i}$ and every $t \in \R$, 
\begin{equation*}
\|u(t)\|_{Y^s}\leq C\big( i+\log(1+|t|) \big)^{\frac12}.
\end{equation*}
Furthermore, if $(f_{p})_{p\geq 0}\in \Sigma^{i}_{N_{p}}$, $N_{p}\to +\infty$ are so that 
$$
\lim_{p\rightarrow + \infty}\|S_{N_p}f_{p}- f \|_{Y^s}=0, 
$$ 
then 
\begin{equation}\label{eq.limite}
\lim_{p \rightarrow+ \infty}\|u(t)-S_{N_p}(\Phi_{N_{p}}(t)(f_{p}))\|_{Y^s}=0.
\end{equation}
Finally, for every $t\in\R$, $\Phi(t)(\Sigma)=\Sigma$.
\end{prop}
\begin{proof}
The key point is now the following lemma.
\begin{lemm}\label{lem.limite}
There exist $\Lambda_0>0$, $C>0$, $K>0$ such that the following holds true.
Consider a sequence $u_{0,N_p}\in E_{N_p}$ and $u_0\in Y^s$. 
Assume that there exists $\Lambda>\Lambda_0$ such that 
$$ 
\|u_{0,N_p}\|_{Y^s}\leq \Lambda,\quad \|u_0\|_{Y^s} \leq \Lambda,\quad \lim_{p\rightarrow+ \infty}\|S_{N_p}u_{0, N_p} -u_0\|_{Y^s}=0.
$$
Then if we set $\tau= C \Lambda^{-K}$ then 
$\Phi_{N_p}(t)( u_{0, N_p})$ and $\Phi(t)(u_0)$ exist for $t \in [0, \tau]$ and satisfy 
$$
\|\Phi_{N_p}(t)( u_{0, N_p})\|_{L^\infty_{\tau}Y^s\cap X^s_{\tau}}\leq 
\Lambda +1, \qquad \|\Phi(t)( u_{0})\|_{L^\infty_{\tau} Y^s\cap X^s_{\tau}}\leq \Lambda +1.
$$
Furthermore
$$
\lim_{p\rightarrow+ \infty} \|S_{N_p} \Phi_{N_p}(t)( u_{0, N_p})- \Phi(t) (u_0)\|_{L^\infty((0,\tau);Y^s)} =0.
$$
\end{lemm}
\begin{proof}
The first part of this lemma is a direct consequence of our local 
well posedness results of Propositions~\ref{wp_N}, \ref{wp}. 
For the second part, let us write 
$$
\Phi(t) (u_0)\equiv u= e^{-itH} (u_0) + v, 
\qquad \Phi_{N_p}(t) (u_{0,N_p})\equiv u_p = e^{-itH} (u_{0,N_p}) + v_p, 
$$
and $w_p = v- S_{N_p}v_p$. We have 
$$ 
u-u_p = e^{-itH} (u_0 -S_{N_p}u_{0,N_p})+ w_p
$$
and by assumption, 
$$
\| e^{-itH} (u_0 - S_{N_p}u_{0,N_p})\|_{Y^s}= 
\|u_0 - S_{N_p}u_{0,N_p}\|_{Y^s} = o(1)_{p\rightarrow + \infty}.
$$
Therefore it remains to show that
$
\|w_p\|_{L^\infty_{\tau}Y^s}\leq C \|w_p\|_{L^\infty_{\tau}Y^s}=o(1)_{p\rightarrow + \infty},
$
for $\tau$ chosen as in the statement of the lemma.
Observe that $w_p$ solves the problem
\begin{multline}\label{eq.diff.bis}
(i\partial_t - H) w_p
= \kappa_0|u|^{k-1} u -\kappa_0 S^2_{N_p}( |S_{N_p} u_p|^{k-1} S_{N_p}u_p)
\\
=\kappa_0(\text{Id}- S^2_{N_p} ) ( |u|^{k-1} u) + 
\kappa_0 S^2_{N_p} ( |u|^{k-1} u - |S_{N_p} u_p|^{k-1} S_{N_p}u_p)
\end{multline}
with initial condition $w_p \mid_{t=0} =0$.
Using Proposition~\ref{multi} and Proposition~\ref{wp}, we obtain that for $\eta>0$
$$ 
\| |u|^{k-1} u\|_{L^{1} ((0,\tau); \mathcal{H}^{s-\eta})} 
\leq C \tau^\kappa\|u\|_{X_\tau^s}^{k}
\leq C\tau^\kappa(\Lambda+1)^k
$$
and consequently, 
\begin{equation}\label{eq.estim1} 
\|(\text{Id}- S^2_{N_p})( |u|^{k-1} u)\|_{L^{1} ((0,\tau); \mathcal{H}^{s-\eta})} 
\rightarrow 0 \text{ as } p\rightarrow + \infty\,.
\end{equation}
We estimate the second term in the r.h.s. of~\eqref{eq.diff.bis} by using a direct manipulation
on the expression $|z_1|^{k-1}z_1-|z_2|^{k-1}z_2$ and invoking
Proposition~\ref{multi}. This yields 
\begin{multline}\label{eq.estim2}
\| |u|^{k-1} u - |S_{N_p} u_p|^{k-1} S_{N_p}u_p\|_{L^{1} ((0,\tau); \mathcal{H}^{s-\eta})} 
\\
\leq C\tau^{\kappa}\| u- S_{N_p} u_p\|_{X^s_{\tau}}
(\| u\|^{k-1}_{X^s_{\tau}}+ \| S_{N_p}u_p\|^{k-1}_{X^s_{\tau}})
\\
\leq C\tau^{\kappa}(\Lambda+1)^{k-1}( \| e^{-itH} (u_0 -S_{N_p} (u_{0,N_p})) \|_{Y^s} + \| w_p\|_{X^s_{\tau}})
\\
\leq o(1) _{p\rightarrow + \infty} + C\tau^{\kappa}(\Lambda+1)^{k-1} \| w_p\|_{X^s_{\tau}}\,.
\end{multline}
We deduce from~\eqref{eq.estim1}, \eqref{eq.estim2}, \eqref{eq.diff.bis} and Lemma~\ref{linear}
that if $\eta\ll 1$, 
$$
\|w_k\|_{X^s_{\tau}} 
\leq C\tau^\kappa(\Lambda+1)^{k-1} 
\| w_p\|_{X^s_{\tau}} + o(1)_{p\rightarrow + \infty}\,.
$$
By taking $C\tau^\kappa( \Lambda+1)^{k-1} <1/2$, we infer that 
$
\|w_p\|_{X^s_{\tau}}=o(1)_{p\rightarrow + \infty}.
$
Next using (\ref{gi_pak}) of Lemma~\ref{linear}, we obtain that
$
\|w_p\|_{L^\infty_{\tau}Y^s}=o(1)_{p\rightarrow + \infty}.
$
This completes the proof of Lemma~\ref{lem.limite}.
\end{proof}
%%%%%%%%%%%%%%%%%%%%%%%%%%%%%%%%%%%%%%%%%%%%%%%%%%%%%%%%%%%%%%%%%%%%%%%%%%%%%%%%%%%%%%%%%
Let us now finish the proof of Proposition~\ref{prop.sigma}. 
By assumption, we know that there exist sequences 
$N_p \in \mathbb{N}, u_{N_p}\in \widetilde{\Sigma}^i_{N_p}$ 
(i.e. $ \Pi_{N_p}(u_{N_p}) \in \Sigma^i_{N_k}$) such that 
$$\lim_{p\rightarrow + \infty}\|S_{N_p}u_{N_p}- u_0\|_{Y^s}=0.$$
 Consequently,
by Proposition~\ref{215}, we know that 
\begin{equation}\label{eq.estapriori}
\big\|\tilde{\Phi}_{N_p}(t)(\Pi_{N_p}(u_{N_p}))\big\|_{Y^s}\leq C(i+\log(1+|t|))^{\frac{1}{2}}.
\end{equation}
The strategy of proof consists in proving that as long as the solution to~\eqref{eq} exists, 
we can pass to the limit in~\eqref{eq.estapriori} 
and there exists a constant $C'$ independent of $i$ such that 
\begin{equation}\label{eq.estaprioribis}
\big\|\Phi(t)(u)\big\|_{Y^s}
\leq C' C(i+\log(1+|t|))^{\frac{1}{2}}
\end{equation}
which (taking into account that the norm in $Y^s$ controls the local existence time), 
implies that the solution is global and satisfies~\eqref{eq.estaprioribis} for all times.

Equivalently, let us fix $T>0$ and $\Lambda>\Lambda_0$ (the number $\Lambda_0$ being fixed in 
Lemma~\ref{lem.limite}). We assume 
\begin{equation}\label{eq.estaprioriter}
\big\|\Phi_{N_p}(t)(\Pi_{N_p}(u_{N_p}))\big\|_{Y^s}
\leq \Lambda, \text{ for } |t| \leq T
\end{equation}
and we want to show 
\begin{equation}\label{eq.estaprioriquar}
\big\|\Phi(t)(u_{0})\big\|_{Y^s}
\leq C'\Lambda, \text{ for } |t| \leq T.
\end{equation}
As a first step, let us fix $t=0$. For $Q\in \mathbb{N}$, if 
$N_p \geq Q$, $\Pi_{N_p} \circ S_Q = S_Q$ and consequently, using 
Proposition~\ref{prop.cont}
and the definition of $\Sigma^i$, we obtain 
$$
\|S_Q (u_0)\|_{Y^s}= \lim_{p\to + \infty} \|S_Q \Pi_{N_p}(u_{N_p})\|_{Y^s} \leq C'\Lambda$$
and passing to the limit $Q \rightarrow + \infty$, we deduce 
$$\|u_0\|_{Y^s}= \lim_{Q\rightarrow + \infty}\|S_Q (u_{0})\|_{Y^s}\leq C'\Lambda$$
This implies that the sequences $\Pi_{N_p} u_{N_p}\equiv u_{0,p}$ and $u_0$ satisfy the assumptions 
of Lemma~\ref{lem.limite} (with $\Lambda$ replaced by $C'\Lambda$). As a consequence, we know that 
$$ \lim_{p\rightarrow+ \infty} \| \Phi_{N_p} (t)( \Pi_{N_p} (u_{N_p}))-\Phi(t) (u_0)\|_{L^\infty((0,\tau); Y^s)} =0$$ 
for $\tau\equiv c \Lambda^{-K}$. 
This convergence allows to pass to the limit 
in~\eqref{eq.estaprioriter} for $t= \tau$, using again Proposition~\ref{prop.cont}.
Indeed, fix $Q$, then for $N_p\gg 2Q$, 
$$ 
 \|S_Q \Phi(\tau) (u_0)\|_{Y^s}= \lim_{p\rightarrow\infty} 
\|S_Q  \Phi_{N_p} ( \tau) \Pi_{N_p}( u_{N_p})\|_{Y^s}
$$ 
and using first~\eqref{eq.estaprioriter} and passing to the limit $Q\rightarrow + \infty$, we deduce 
\begin{equation}
\|\Phi(\tau)(u_0)\|_{Y^s}= \lim_{Q\rightarrow + \infty}\|S_Q \Phi(\tau) (u_0)\|_{Y^s}
\leq \sup_{Q} \|S_Q\|_{\mathcal{L} ( Y^s)} \Lambda.
\end{equation}
Now, we can apply the results in Lemma~\ref{lem.limite}, with the same $\Lambda$ as in 
the previous step, which implies that ~\eqref{eq.estaprioriquar} holds for $t\in[0, 2\tau]$, 
and so on and so forth. 

Notice here that at each step the a priori bound does not get worse, because we only use the results in 
Lemma~\ref{lem.limite} to obtain the convergence of $\|\Phi_{N_p} ( t) (\Pi_{N_p}( u_{N_p}))- \Phi(t) u_0\|_{Y^s} $ to 
$ 0$, and then obtain the estimates on the norm $ \|\Phi(t) (u_0)\|_{Y^s}$ by passing to the 
limit in~\eqref{eq.estaprioriter} (applying first $S_Q$, passing to the limit $p\rightarrow + \infty$, then to the limit $Q\rightarrow + \infty$). A completely analogous argument holds for the negative times $t$.

In order to prove the last statement in Proposition~\ref{prop.sigma} we observe that, according to~\eqref{bonn2} 
there exists $c>0$ such that for any $t \in \mathbb{R}$,
$$ 
\Phi(t) ( \Sigma^i) \subset \Sigma^{i + [c\log (|t| + 1)] +3}
$$ 
which is a straightforward consequence of~\eqref{eq.limite} and (\ref{bonn2}).
As consequence we have $\Phi(t)(\Sigma)\subset\Sigma$ and thanks to the reversibility of the flow $\Phi(t)$, we
infer that $\Phi(\Sigma)=\Sigma$.
This completes the proof of Proposition~\ref{prop.sigma}.
\end{proof}
%%%%%%%%%%%%%%%%%%%%%%%%%%%%%%
%%%%%%%%%%%%%%%%%% 
\section{Measure invariance}
In this section, we prove the last part of Theorem~\ref{thm4}.
Recall that we see $\rho$ as a finite Borel measure on $Y^s$.
Let $\Sigma$ be the set of full $\rho$ measure constructed
in the previous section. This is the set involved in the statement of Theorem~\ref{thm4}.
Recall that thanks to Proposition~\ref{prop.sigma}, $\Phi(t)(\Sigma)=\Sigma$ and 
thanks to the reversibility of the flow $\Phi(t)$, it suffices to prove that for every $\rho$ measurable
set $A\subset \Sigma$ and every $t\in\R$, $\rho(A)\leq \rho(\Phi(t)(A))$. We 
perform several reductions which will allow us to reduce the matters to compact sets $A$ and small times $t$.
First by the regularity properties of $\rho$, we may assume that $A$ is a closed set of $Y^s$.
Then thanks to Lemma~\ref{kompakt}, it suffices to prove $\rho(K)\leq \rho(\Phi(t)(K))$ for $K$ a compact set 
of $Y^s$. Let us fix a compact $K$ of $Y^s$ and a time $t>0$ (the case $t<0$ is analogous).
Thanks to Proposition~\ref{prop.sigma}, there exists $R>1$ such that
$
\{\Phi(\tau)(K),0\leq \tau\leq t\}\subset B_{R},
$
where here and for future references $B_R$ denotes the open ball of $Y^s$ centered at the origin and of
radius $R$.
We have the following statement comparing $\Phi(t)$ and $\Phi_{N}(t)$ for small (but uniform) times and 
compacts contained in $B_R$.
\begin{lemm}\label{lem.limite.bis}
There exist $c>0$ and $\gamma>0$ such that the following holds true.
For every $R>1$, every compact $K$ of $B_R$ and every $\eps>0$ there exists $N_0\geq 1$ 
such that for every $N\geq N_0$, every $u_0\in K$, every $\tau\in [0, cR^{-\gamma}]$,
$
\|\Phi(\tau)(u_0)-\Phi_{N}(\tau)(u_0)\|_{Y^s}<\eps.
$
\end{lemm}
\begin{proof}To prove this lemma, take two new cut off $S_{N,i}= \chi_i(\frac{ H} { 2N+1})$ with $\chi_1 \chi= \chi$, $\chi_2 \chi_1= \chi_1$ so that $ S_{N,1} S_N= S_N$, $S_{N,2}S_{N,1} = S_{N,1}$. Notice first that
\begin{multline}
\|\Phi(\tau)(u_0)-\Phi_{N}(\tau)(u_0)\|_{Y^s}\leq \|(1-S_{N,1})\bigl(\Phi(\tau)(u_0)-\Phi_{N}(\tau)(u_0)\bigr)\|_{Y^s}\\
+ \|S_{N,1}\bigl(\Phi(\tau)(u_0)-\Phi_{N}(\tau)(u_0)\bigr)\|_{Y^s}.
\end{multline}
To bound the first term, we notice that 
$$ \|(1-S_{N,1})\bigl(\Phi_N(\tau)(u_0)\bigr)\|_{Y^s}=  \|(1-S_{N,1})\bigl(e^{itH}u_0\bigr)\|_{Y^s}=  \|(1-S_{N,1})(u_0)\|_{Y^s},$$
and
$$ \lim_{N\rightarrow = \infty} \|(1-S_{N,1})\bigl(\Phi(\tau)(u_0)\bigr)\|_{Y^s} = 0$$
uniformly with respect to $u_0$ in a compact set of $Y^s$.
To bound the second term, we notice that 
$$
\|S_{N,1}\bigl(\Phi(\tau)(u_0)-\Phi_{N}(\tau)(u_0)\bigr)\|_{Y^s}= \|S_{N,1}\bigl(\Phi(\tau)(u_0)-{S}_{N,2} \Phi_{N}(\tau)(u_0)\bigr)\|_{Y^s}.
$$
 Now to estimate this term we proceed as in the proof of Lemma~\ref{lem.limite}, the only additional point being 
the observation that ${S}_{N,2}(u)$ converges to $u$ in $Y^s$, uniformly with respect to $u$ in a compact 
of $Y^s$.
\end{proof}
We next observe that we only need to prove 
$\rho(K)\leq \rho(\Phi(\tau)(K))$ for $\tau\in [0, cR^{-\gamma}]$, where $R$ and $\gamma$ are fixed by
Lemma~\ref{lem.limite.bis}. Then we can iterate the inequality on the same time intervals since we know that
$\Phi(\tau)(K)$ remains included in $B_R$ as far as $\tau\in [0, t]$.
Using \eqref{limite}, Lemma~\ref{lem.limite.bis} and 
the well-posedness result of Proposition~\ref{wp_N} (notice that, though only stated for the flow $\tilde \Phi_N(t)$ on $E_N$, the result holds clearly for the flow $ \Phi_N = (e^{itH}, \tilde \Phi_N(t))$ on $E_N^\perp \times E_N$), we can write
\begin{multline*}
\rho(\Phi(\tau)(K)+{B_{2\eps}})= \lim_{N\rightarrow\infty}
\rho_{N}(\Phi(\tau)(K)+{B_{2\eps}})
\\
\geq
\limsup_{N\rightarrow\infty}
\rho_{N}(\Phi_N(\tau)(K)+B_{\eps})
\geq 
\limsup_{N\rightarrow\infty}
\rho_{N}(\Phi_N(\tau)(K+B_{\alpha\eps})),
\end{multline*}
where $\alpha$ is a fixed constant depending on $R$ but independent of $\eps$.
Next, using the invariance of the measure ${\rho}_N$ under the flow $\Phi_{N}(t)$ and using once again
\eqref{limite}, we can write
$$
\rho(\Phi(\tau)(K)+{B_{2\eps}})\geq
\limsup_{N\rightarrow\infty}
\rho_{N}(K+B_{\alpha\eps})
= \rho(K+B_{\alpha\eps})\geq \rho(K)\,.
$$
Using that $\Phi(t) (K)$ is closed and letting $\eps$ to zero, the dominated convergence theorem implies that $\rho(\Phi(\tau)(K))\geq \rho(K)$.
This proves the measure invariance.
The proof of Theorem~\ref{thm4} is therefore completed.
%%%%%%%%%%%%%%%%%%%%%%%%%%%%%%%%%%%%%%%%%%%%%%%%%%%%%%%%%%%%%%%%%%%%%%%%%%%%
%%%%%%%%%%%%%%%%%%%%%%%%%%%%%%%%%%%%%%%%%%%%%%%%%%%%%%%%%%%%%%%%%%%%%%%%%%%%
%%%%%%%%%%%%%%%%%%%%%%%%%%%%%%%%%%%%%%%%%%%%%%%%%%%%%%%%%%%%%%%%%%%%%%%%%%%%
\section{Proof of Theorem~\ref{thm5}}
Suppose that $v(s,y)$ is a solution of the problem
\begin{equation}\label{C1}
i\partial_sv+\partial_{y}^2v=|v|^{k-1}v,\quad s\in \R,\quad y\in\R.
\end{equation}
We define $u(t,x)$ for $|t|<\frac{\pi}{4}$, $x\in\R$ by
\begin{equation}\label{C2}
u(t,x)=\frac{1}{\cos^{\frac{1}{2}}(2t)}v\big(\frac{\tan(2t)} 2,\frac{x}{\cos(2t)}\big)
e^{-\frac{ix^2{\rm tg}(2t)}{2}}\,.
\end{equation}
We then can check that $u$ solves the problem
\begin{equation}\label{C3}
i\partial_t u-Hu=\cos^{\frac{k-5}{2}}(2t)|u|^{k-1}u,\quad |t|<\frac{\pi}{4},\, x\in\R.
\end{equation}
One also has that the map \eqref{C2} sends solutions of the linear Schr\"odinger 
equation without harmonic potential to solutions of the linear Schr\"odinger 
equation with harmonic potential. 
We refer to \cite{Carles2} for a use of \eqref{C2} in the context of scattering for 
$L^2$ critical problems,
i.e. quintic nonlinearities in $1d$. 
The problem (\ref{C3}) has also the following Duhamel formulation
\begin{equation}\label{C3-Duhamel}
u(t)=e^{-i(t-t_0)H}(u(t_0))-i
\int_{t_0}^t e^{-i(t-\tau)H}\big(\cos^{\frac{k-5}{2}}(2\tau)|u(\tau)|^{k-1}u(\tau)\big)d\tau
\end{equation}
with $t_0,t\in (-\frac{\pi}{4},\frac{\pi}{4})$.
The local analysis of (\ref{C3}) will be applied to (\ref{C3-Duhamel}) which fits well in the framework
of Propositions~\ref{wp_N},~\ref{wp}.
By the transformation (\ref{C2}) we may link the solutions of (\ref{C1}) on $\R\times\R$ to the
solutions of (\ref{C3}) on $(-\pi/4,\pi/4)\times\R$.
The results of Theorem~\ref{thm5} will therefore be a consequence of the following local in time 
(but large data) result concerning (\ref{C3}), together with the observation that thanks to (\ref{C2}) 
the $Y^s$ and ${\mathcal H}^s$ 
convergence in the context of \eqref{C3} implies the $Y^s$ and ${\mathcal H}^s$ convergence for the original 
problem \eqref{C1}.
\begin{prop}\label{sabota}
The equation (\ref{C3}) has $\mu$ almost surely a unique solution in $C([-\frac{\pi}{4},\frac{\pi}{4}];Y^s)$.
Moreover we can write the solution as 
$$
u(t)=e^{-itH}(f^{\pm})+w^{\pm}(t),
$$
with $f^{\pm}\in Y^s$ and where $w^{\pm}$ are such that
$$
\lim_{t\rightarrow\pm \pi/4}\|w^{\pm}(t)\|_{{\mathcal H}^s}=0\,.
$$
\end{prop}
\begin{proof}[Proof of Proposition~\ref{sabota}] 
The proof of this proposition is very similar in spirit to the proof of Theorem~\ref{thm4}.
The local analysis is essentially the same.
There is however a nontrivial modification in the globalization arguments 
because of the lack of energy conservation of (\ref{C3}).
We consider the ODE
\begin{equation}\label{C3_N}
i\partial_t u-Hu=\cos^{\frac{k-5}{2}}(2t)S_N(|S_Nu|^{k-1}S_Nu),\quad u(0)\in E_N.
\end{equation}
One may multiply \eqref{C3_N} by $\bar{u}$ and integrate over $\R$ to obtain that the $L^2$
norm is conserved by the flow and combining this fact with the local existence theory of
ODE's, we obtain that the ODE (\ref{C3_N}) with phase space $E_N$ has
a unique global in time solution.
For two real numbers $t_1$, $t_2$
let us denote by $\tilde\Phi_{N}(t_1,t_2)$ the flow of (\ref{C3_N}) from $t_1$ to $t_2$.
We have the following monotonicity property for the solutions of (\ref{C3_N}).
\begin{lemm}\label{lyapounov}
Set
$$
\mathcal{E}_{N}(t,u(t))=\frac{1}{2}\|\sqrt{H}\,u(t)\|_{L^2(\R)}^2+\frac{\cos^{\frac{k-5}{2}}(2t)}{k+1}
\|S_Nu(t)\|_{L^{k+1}(\R)}^{k+1}\,.
$$
Then the solution of \eqref{C3_N} satisfies
$$
\mathcal{E}_{N}(t,u(t))\leq E_{N}(0,u(0)),\quad |t|\leq \frac{\pi}{4}\,.
$$
\end{lemm}
\begin{proof}
A direct computation shows that along the flow of (\ref{C3_N}) one has
$$
\frac{d}{dt}\big(\mathcal{E}_{N}(t,u(t))\big)=
-\frac{(k-5)\sin(2t)\cos^{\frac{k-5}{2}}(2t)}{k+1}
\|S_Nu(t)\|_{L^{k+1}(\R)}^{k+1}\,.
$$
Therefore the function $\mathcal{E}_{N}(t,u(t))$ increases on the interval $[-\pi/4,0]$ and decreases on the interval
$[0,\pi/4]$, and attain its maximum at  $0$. This completes the proof of Lemma~\ref{lyapounov}.
\end{proof}
 We shall prove that \eqref{C3} is well-posed on $[-\pi/4,\pi/4]$ $\rho$-almost surely which in turn will imply the claimed well-posedness $\mu$ a.s.
The result of Lemma~\ref{lyapounov} implies the following key measure monotonicity property.
\begin{lemm}\label{lyapounov_conseq}
For every Borel set $A$ of $E_N$ and every $|t|\leq \frac{\pi}{4}$,
$$
\tilde{\mu} _N(\tilde\Phi_{N}(t,0)(A))\geq \tilde\rho_{N}(A).
$$
\end{lemm}
\begin{proof}
By definition
$$
\tilde{\mu} _N(\Phi_{N}(t,0)(A))=
d_{N}\int_{\tilde\Phi_{N}(t,0)(A)}e^{-\frac{1}{2}\|\sqrt{H}\,u\|_{L^2(\R)}^2}du,
$$
where $du$ is the Lebesgue measure on $E_N$ induced by $\C^{2(N+1)}$ by the map \eqref{transport}.
Let us perform the variable change $u\mapsto\tilde \Phi_{N}(t,0)(u)$. 
We can apply the result of Lemma~\ref{gentchev} to obtain that the Jacobian of this variable change
is one (the divergence free assumption can be readily checked by expressing $\tilde\Phi_{N}(t,0)(u)$ in terms
of its decomposition with respect to $h_0,\cdots h_{N}$). Thus we get
\begin{multline*}
\tilde{\mu} _N(\tilde\Phi_{N}(t,0)(A)) = 
d_{N}\int_{A}e^{-\frac{1}{2}\|\sqrt{H}\,\tilde \Phi_{N}(t,0)(u)\|_{L^2(\R)}^2}du
\\
\geq 
d_{N}\int_{A}e^{-\frac{1}{2}\|\sqrt{H}\,\tilde \Phi_{N}(t,0)(u)\|_{L^2(\R)}^2-
\frac{\cos^{\frac{k-5}{2}}(2t)}{k+1}
\|S_N\,\tilde \Phi_{N}(t,0)(u)\|_{L^{k+1}(\R)}^{k+1}}du\,.
\end{multline*}
Using Lemma~\ref{lyapounov} we hence obtain
$$
\tilde{\mu} _N(\tilde\Phi_{N}(t,0)(A))
\geq
d_{N}\int_{A}e^{-\frac{1}{2}\|\sqrt{H}\,u\|_{L^2(\R)}^2-
\frac{1}{k+1}
\|S_N\,u\|_{L^{k+1}(\R)}^{k+1}}du
=\tilde{\rho}_{N}(A).
$$
This completes the proof of Lemma~\ref{lyapounov_conseq}.
\end{proof}
For $I$ in interval, we can define the spaces $X^s_{I}$ similarly to the spaces $X^s_{T}$ by 
replacing $[-T,T]$ by $I$. We have the following well-posedness result concerning (\ref{C3_N}).
\begin{prop}\label{wp_N_bis}
There exist $C>0$, $c\in (0,1)$, $\gamma>0$, $\kappa>0$ such that
for every $A\geq 1$ if we set $T=cA^{-\gamma}$ then for every $N\geq 1$, every 
$t_0\in [-\frac{\pi}{4},\frac{\pi}{4}]$, every 
$u_0\in E_N$ satisfying
$\|u_0\|_{Y^s}\leq A$ there exists a unique solution of (\ref{C3_N})
with data $u(t_0)=u_0$
on the interval $I=[t_0-T,t_0+T]$ such that
$\|u\|_{X^s_{I}}\leq A+A^{-1}.$ 
In addition for $t\in I$,
\begin{equation}\label{sot}
\|u(t)\|_{Y^s}\leq A+A^{-1}.
\end{equation}
Moreover, if $u$ and $v$ are two solutions with data $u_0$ and $v_0$ 
respectively, satisfying $\|u_0\|_{Y^s}\leq A$, $\|v_0\|_{Y^s}\leq A$ then
$\|u-v\|_{X^s_I}\leq C\|u_0-v_0\|_{Y^s}$ and for $t\in I$,
$$
\|u(t)-v(t)\|_{Y^s}\leq C\|u_0-v_0\|_{Y^s}\,.
$$
Finally, if $J\subset I$ is an interval, then for $\eta>0$,
\begin{equation}\label{scat}
\|\int_{J} e^{-i(t-\tau)H}\big(\cos^{\frac{k-5}{2}}(2\tau)|u(\tau)|^{k-1}u(\tau)\big)d\tau\|_{{\mathcal H}^{s-\eta}}
\leq C|J|^{\kappa}A.
\end{equation}
\end{prop}
\begin{proof}
The proof of this statement is completely analogous to that of Proposition~\ref{wp_N},
one needs to observe that in Lemma~\ref{linear} and Proposition~\ref{multi} one may replace $[-T,T]$ by an
arbitrary interval, $T$ by the size of this interval and one may add the factor $\cos^{\frac{k-5}{2}}(2\tau)$
with the same conclusion. The only additional point is the estimate \eqref{scat}.
To prove estimates \eqref{scat}, we use that 
$$
\|\int_{J}
e^{-i(t-\tau)H}\big(\cos^{\frac{k-5}{2}}(2\tau)F(\tau)\big)d\tau\|_{{\mathcal H}^{s-\eta}}
\leq C\|F\|_{L^1_{J}{\mathcal H}^{s-\eta}}
$$
and apply the estimates of Proposition~\ref{multi}.
\end{proof}
The rest of the proof of Proposition~\ref{sabota} is very similar to the existence part of
Theorem~\ref{thm4}. We start by the counterpart of Proposition~\ref{215}.
\begin{prop}\label{215_bis}
There exists a constant $C>0$ such that for all $i,N\in \N^{*}$, 
there exists a $\rho_{N}$ measurable set $\tilde\Sigma_{N}^{i}\subset E_{N}$ so that 
for all $i,N\in \N^{*}$
\begin{equation*}
\tilde\rho_{N}( E_{N}\backslash \tilde \Sigma_{N}^{i})\leq 2^{-i}.
\end{equation*}
For all $f\in \tilde \Sigma_{N}^{i}$ and $t\in [-\frac{\pi}{4},\frac{\pi}{4}]$
\begin{equation*}
\|\tilde\Phi_{N}(t,0)f\|_{Y^s}\leq Ci^{\frac12}.
\end{equation*}
\end{prop}
\begin{proof}
We set, for $i$ an integer $\geq 1$, 
$
\tilde B_{N}^{i}(D)\equiv
\big\{u\in E_{N}\,:\,\|u\|_{Y^s}\leq D i^{\frac{1}{2}}\big\},
$
where the number $D\gg 1$ (independent of $i,N$) will be fixed later. 
Thanks to Proposition~\ref{wp_N_bis}, there exist $c>0$, $\gamma>0$ only depending on $s$ 
such that if we set $\tau \equiv c D^{-\gamma}i^{-\gamma/2}$ then for every $t_1,t_2$,
such that $|t_1-t_2|\leq \tau$,
\begin{equation}\label{flotN_bis}
\tilde\Phi_{N}(t_1,t_2)\big(\tilde B_{N}^{i}(D)\big)\subset 
\big\{u\in E_{N}\,:\,\|u\|_{Y^s}\leq D (i+1)^{\frac 12}\big\}\, ,
\end{equation}
provided $D\gg 1$, independently of $i$. Set
$$
\tilde \Sigma_{N}^{i}(D)\equiv
\bigcap_{k=-[\pi/4\tau]}^{[\pi/4\tau]}
\tilde\Phi_{N}(k\tau,0)^{-1}(\tilde B_{N}^{i}(D))\, .
$$ 
Notice that thanks to \eqref{flotN_bis}, 
we obtain that the solution of \eqref{C3} with data $f\in\tilde \Sigma_{N}^{i}(D)$ satisfies
\begin{equation}\label{eqflot_bis}
\big\|\tilde\Phi_{N}(t,0)(f)\big\|_{Y^s}
\leq D(i+1)^{\frac{1}{2}},\quad |t|\leq \frac{\pi}{4}\,.
\end{equation}
Indeed, for $|t|\leq \frac{\pi}{4}$, we can find an integer $k\in [-[\pi/4\tau],[\pi/4\tau]]$ and 
$\tau_1\in [-\tau,\tau]$ so that $t=k\tau+\tau_1$ and thus 
$$
\tilde \Phi_{N}(t,0)(f)=\tilde\Phi_{N}(t,k\tau)\tilde\Phi_{N}(k\tau,0)(f).
$$
Since $f\in\tilde\Sigma_{N}^{i}(D)$ implies that $\tilde\Phi_{N}(k\tau,0)(f)\in 
\tilde B_{N}^{i}(D)$, we can apply \eqref{flotN_bis} and arrive at \eqref{eqflot_bis}.
It remains to evaluate the $\tilde\rho_{N}$ complementary measure of the set $\tilde\Sigma_{N}^{i}(D)$.
Using Lemma~\ref{lyapounov_conseq}, we can write
\begin{eqnarray*}
\tilde\rho_{N}\big( E_{N}\backslash\tilde\Sigma_{N}^{i}(D)\big)
&\leq &
(2[\pi/4\tau]+1)\rho_{N}\big(\tilde\Phi_{N}(k\tau,0)^{-1}(E_N\backslash \tilde B_{N}^{i}(D))\big)\\
&\leq &
CD^{\gamma}i^{\gamma/2}\tilde{\mu} _N\big(E_N\backslash \tilde B_{N}^{i}(D)\big)\,.
\end{eqnarray*}
By the large deviation bounds of Lemma~\ref{cfg}, we get
\begin{equation*}
\tilde\rho_{N}\big( E_{N}\backslash\tilde \Sigma_{N}^{i}(D)\big)\leq
CD^{\gamma}i^{\gamma/2}e^{-cD^2i}\leq 2^{-i},
\end{equation*}
provided $D\gg 1$, independently of $i,N$.
This completes the proof of Proposition~\ref{215_bis}.
\end{proof}
%%%%%%%%%%%%%%%%%%%%%%%%%%%%%%%%%%%%%%%%%%%%%%%%%%%%%%%%%%%%%%%%%%%%%%%%%%%%%%%%%%%%%%%%%%
%%%%%%%%%%%%%%%%%%%%%%%%%%%%%%%%%%%%%%%%%%%%%%%%%%%%%%%%%%%%%%%%%%%%%%%%%%%%%%%%%%%%%%%%%%%
Since we are only concerned with a well-posedness statement, we need to prove less 
compared with Theorem~\ref{thm4}
(we do not need to prove that the statistical ensemble is a set reproduced by the flow).
For integers $i\geq 1$ and $N \geq 1$, we define the cylindrical sets
$$
\Sigma_{N}^{i}\equiv\big\{u\in Y^s\,:\, \Pi_{N}(u)\in \tilde \Sigma_{N}^{i}\big\}.
$$
For $i\geq 1$, we set
\begin{equation*}
\Sigma^{i}=\big\{ u\in Y^s\,:\, \exists\, N_{k}\in \N,\, N_k \to +\infty,\\
\exists\, u_{N_{k}}\in {\Sigma}^{i}_{N_{k}}, S_{N_k}(u_{N_{k}})\to u \;\text{ in }Y^s \big\}.
\end{equation*}
As in the proof of Theorem~\ref{thm4}, we obtain the bound
\begin{equation}\label{kr5_bis}
\rho\big(\Sigma^{i}\big) \geq \rho(Y^s)-2^{-i}.
\end{equation}
%%%%%%%%%%%%%%%%%%%%%%%%%%%%%%%%%%%%%%%%%%%%%%%%%%%%%%%%%%%%%%%%%%%%%%%%%%%%%%%%%%%%%%%%%%%%%%%%%%%%%%%%%%%
%%%%%%%%%%%%%%%%%%%%%%%%%%%%%%%%%%%%%%%%%%%%%%%%%%%%%%%%%%%%%%%%%%%%%%%%%%%%%%%%%%%%%%%%%%%%%%%%%%%%%%%%%%%
%%%%%%%%%%%%%%%%%%%%%%%%%%%%%%%%%%%%%%%%%%%%%%%%%%%%%%%%%%%%%%%%%%%%%%%%%%%%%%%%%%%%%%%%%%%%%%%%%%%%%%%%%%%
Next, we set
\begin{equation}\label{def.sigma_bis}
\Sigma\equiv\bigcup_{i=1}^{\infty}\Sigma^{i}.
\end{equation}
and by \eqref{kr5_bis}, the set $\Sigma$ is of full $\rho$ measure.
We now state a proposition yielding the existence part of Proposition~\ref{sabota}.
%%%%%%%%%%%%%%%%%%%%%%%%%%%%%%%%%%%%%%%%%%%%%%%%%%%%%%%%%%%%%%%%%%%%%%%%%%%%%%%%%%%%%%%%
\begin{prop}\label{prop.sigma_bis}
For every integer $i\geq 1$, every $f\in \Sigma^{i}$, the problem 
\eqref{C3} 
with initial condition $f$ has a unique solution in
$C([-\frac{\pi}{4},\frac{\pi}{4}];Y^s)$.
\end{prop}
The proof of Proposition~\ref{prop.sigma_bis} is very similar (simpler) to that of
Proposition~\ref{prop.sigma}, by invoking the counterpart of the approximation
statement of Lemma~\ref{lem.limite}.
This implies the existence part of Proposition~\ref{sabota}. Namely we proved the well-posedness 
for data in $\Sigma$ (defined by \eqref{def.sigma_bis}) and since $\Sigma$ is of full $\rho$ measure
it is of full $\mu$ measure too.

To prove the last statement of Proposition~\ref{sabota}, we write the obtained solution as
\begin{multline*}
u(t)=e^{-itH}
\big(u(0)-2i
\int_{0}^{\pi/4} 
e^{i\tau H}\big(\cos^{\frac{k-5}{2}}(2\tau)|u(\tau)|^{k-1}u(\tau)\big)d\tau\big)
\\
+
2i
\int_{t}^{\pi/4} e^{-i(t-\tau)H}\big(\cos^{\frac{k-5}{2}}(2\tau)|u(\tau)|^{k-1}u(\tau)\big)d\tau
\end{multline*}
and we apply estimate (\ref{scat}). A similar argument applies near $-\pi/4$.
This completes the proof of Proposition~\ref{sabota}.
\end{proof}
%%%%%%%%%%%%%%%%%%%%%%%%%%%%%%%%%%%%%%%%%%%%%%%%%%%%%%%%%%%%%%%%%%%%%%%%%%%%%%%%%%%%%%%%
%%%%%%%%%%%%%%%%%%%%%%%%%%%%%%%%%%%%%%%%%%%%%%%%%%%%%%%%%%%%%%%%%%%%%%%%%%%%%%%%%%%%%%%%
%%%%%%%%%%%%%%%%%%%%%%%%%%%%%%%%%%%%%%%%%%%%%%%%%%%%%%%%%%%%%%%%%%%%%%%%%%%%%%%%%%%%%%%%
\appendix
\section{Typical properties on the support of the measure}\label{appendix}
In this section, we give some additional properties of the stochastic series 
\begin{equation*}
\phi(\om,x)=\sum_{n=0}^{\infty}{\frac{\sqrt 2}{\lambda_{n}}}g_{n}(\omega)h_n(x),
\end{equation*}
which have their own interest.
\subsection{Mean and pointwise properties}
%%%%%%%%%%%%%%%%%%%%%%%%%%%%%%%%%%
%%%%%%%%%%%%%%%%%% 
\begin{prop}[$L^{p}$ regularisation]\label{prop.series}
Let $2\leq p<+\infty$ and denote by 
\begin{equation*}
\theta(p)=\left\{\begin{array}{ll} 
\frac12-\frac1p \quad &\text{if} \quad 2\leq p\leq 4, \\[6pt] 
\frac13(\frac12+\frac1p) \quad &\text{if} \quad
4\leq p<\infty.
\end{array} \right.
\end{equation*}
Then for all $s<\theta(p)$, there exist $C,c>0$ so that
\begin{equation*}\label{est.gdev}
{\bf{p}}\big(\om\in \Omega \;:\;\|\phi(\om,\cdot)\|_{\W^{s,p}(\R)}>\lambda\big)\leq C \e^{-c\lambda^{2}}.
\end{equation*}
In particular $\|\phi(\om,\cdot)\|_{\W^{s,p}(\R)}<+\infty$, ${\bf p}$ a.s.
\end{prop}
\begin{proof}
The proof is essentially the same as the proof of Lemma~\ref{ld}, using the precise $L^{p}$ 
bounds on the Hermite functions $h_{n}$ (see \cite{YajimaZhang1} or \cite[Theorem 2.1]{Thomann5}). 
\end{proof}
\begin{coro}[Decay]
Let $\alpha <\frac16$. Then there exist $C,c>0$ so that for all $x\in \R$ 
\begin{equation*}
{\bf{p}}\big(\om\in \Omega \;:\;|\phi(\om,x)|>\frac{\lambda}{\<x\>^{\alpha}}\,\big)\leq C \e^{-c\lambda^{2}}.
\end{equation*}
In particular, for almost all $\om \in \Omega$,
$$ \phi(\om,x)\longrightarrow 0 \quad\text{when}\quad x\longrightarrow \pm \infty.$$ 
\end{coro}
%%%%%%%%%%%%%%%%%%%%%%%
\begin{proof}
Let $\alpha <\frac16$. Then choose $s>0$ so that $s+\alpha<\frac16$ and $p\geq 4$ so that $s>\frac1p$. Then by Sobolev, there exists $C>0$ so that for all $\om \in \Omega$
\begin{equation*}
\|\<x\>^{\alpha}\phi(\om,\cdot)\|_{L^{\infty}(\R)}\leq C \|\<x\>^{\alpha}\phi(\om,\cdot)\|_{\W^{s,p}(\R)}.
\end{equation*}
Now by \cite[Lemma 2.4]{YajimaZhang2}, 
$$\|\<x\>^{\alpha}\phi(\om,\cdot)\|_{\W^{s,p}(\R)}\leq C\|\phi(\om,\cdot)\|_{\W^{s+\alpha,p}(\R)},$$
thus
\begin{equation*}
\big\{ \om\in \Omega \;:\;\<x\>^{\alpha}|\phi(\om,x)|>\lambda \big\} \subset \big\{ \om\in \Omega \;:\;\|\phi(\om,\cdot)\|_{\W^{s+\alpha,p}(\R)}>\frac{\lambda}C \big\} .
\end{equation*}
and we can conclude with the Proposition \ref{prop.series}, as $s+\alpha<\theta(p)$.
\end{proof}

\begin{prop}[H\"olderian regularity]
Let $\alpha <\frac16$. There exist $C,c>0$ so that for all $x,y\in \R$ 
\begin{equation*}
{\bf{p}}\big(\om\in \Omega \;:\;|\phi(\om,x)-\phi(\om,y)|>\lambda |x-y|^{\alpha}\big)\leq C \e^{-c\lambda^{2}}.
\end{equation*}
In particular, for almost all $\om \in \Omega$, the function $x\longmapsto \phi(\om,x)$ is $\alpha$-H\"olderian on $\R$.
\end{prop}
\begin{proof}
By Lemma \ref{disp}, for all $x,y\in \R$ we have 
\begin{equation*}
|h_{n}(x)-h_{n}(y)|\leq C\lambda_{n}^{-\frac16}.
\end{equation*}
By Lemma \ref{disp} again, we also have the bound (see \eqref{derive_hermite}) 
\begin{equation*}
|h_{n}(x)-h_{n}(y)|\leq \|h_{n}\|_{\W^{1,\infty}(\R)}|x-y|
\leq C\lambda_{n}^{\frac56}|x-y|,
\end{equation*}
and we can deduce by interpolation that for all $0\leq \alpha \leq 1$, 
\begin{equation*}
|h_{n}(x)-h_{n}(y)|\leq C\lambda_{n}^{\alpha-\frac16}|x-y|^{\alpha}.
\end{equation*}
Now, by Lemma \ref{Gaussian}, for all $r\geq 2 $
\begin{eqnarray*}
\|\phi(\om,x)-\phi(\om,y)\|_{L^{r}(\Omega)}&\leq &C\sqrt{r} \big(\sum_{n=0}^{\infty} \frac1{\lambda^{2}_{n}}|h_{n}(x)-h_{n}(y)|^{2} \big)^{\frac12}\\
&\leq &C\sqrt{r} |x-y|^{\alpha}\big(\sum_{n=0}^{\infty} \frac1{\lambda^{2(1-\alpha+\frac16)}_{n}} \big)^{\frac12}\\
&\leq &C\sqrt{r} |x-y|^{\alpha},
\end{eqnarray*}
for all $0\leq \alpha<\frac16$. We conclude with the Tchebychev inequality.
\end{proof}
\noindent The Proposition \ref{prop.series} shows that the random variables 
$\big(g_{n}(\om)\big)_{n\geq 0}$ yield no gain of derivatives in $\H^{s}$ spaces, 
however we can prove a local gain of regularity.
\begin{prop}[Local smoothing]
Let $\nu>0$ and define $\Psi(x)=\<x\>^{-\frac12-\nu}$. Then for all $s<\frac12$, there exist $C,c>0$ so that
\begin{equation*}
{\bf{p}}\big(\om\in \Omega \;:\;\|\Psi\, \phi(\om,\cdot)\|_{\H^{s}(\R)}>\lambda\big)\leq C \e^{-c\lambda^{2}}.
\end{equation*}
In particular $\|\Psi\, \phi(\om,\cdot)\|_{\H^{s}(\R)}<+\infty$, ${\bf p}$ a.s.
\end{prop}
\begin{proof}
By \cite[Corollary~1.2]{Thomann6} the following bound holds
$$\|\Psi\,h_{n}\|_{L^{2}(\R)}\leq \lambda_{n}^{-\frac12}.$$
Then we can perform the same argument as in the proof of Lemma \ref{ld}.
\end{proof}
%%%%%%%%%%%%%%%%%%%%%%%%%%%%%%%%%%%%%%%%%%%%%%%%%%%%%%
\subsection{Spatial decorrelation} Define the function $E$ for $(x,y,\a)\in \R\times \R\times [0,1[$ by 
\begin{equation}\label{funct_E}
E(x,y,\a)=\sum_{n\geq 0}\a^{n}\,h_{n}(x)\,h_{n}(y).
\end{equation}
Then we have an explicit formula for $E$. 
\begin{lemm}\label{wiki}
For all $(x,y,\a)\in \R\times \R\times [0,1[$
\begin{equation}\label{for}
E(x,y,\a)=\frac1{\sqrt{\pi(1-\a^{2})}}\exp\big(-\frac{1-\a}{1+\a}\,\frac{(x+y)^{2}}{4}-\frac{1+\a}{1-\a}\,\frac{(x-y)^{2}}{4}\big).
\end{equation}
\end{lemm}
\begin{rema} Notice that by taking  $\alpha= e^{2it}$, one can see that Lemma~\ref{wiki} is equivalent to Mehler formula~\eqref{bg1}, which in turn implies that the function defined by~\eqref{C2} satisfies~\eqref{C3}. Actually, one could probably extend Lemma~\ref{wiki} to more general potential (with quadratic growth) by precisely writing down a parametrix for $e^{it( - \partial_x ^2 + V(x))}$, or for the heat kernel $e^{-t( - \partial_x ^2 + V(x))}$.
\end{rema}
\begin{proof}
First we recall that the Fourier transform of the Gaussian reads
\begin{equation}\label{int.fourier0}
\e^{-\sigma^{2}x^{2}}=\frac1{2\sigma \sqrt{\pi}}\int_{\R}\e^{ix\,\xi -\frac{\xi^{2}}{4\sigma^{2}}}\text{d}\xi,
\end{equation}
thus, for all $n\geq 1$, 
\begin{equation}\label{int.fourier}
\frac{\text{d}^{n}}{\text{d}x^{n}}\big(\,\e^{-x^{2}}\,\big)=\frac1{2\sqrt{\pi}}\int_{\R}(i\,\xi)^{n}\,\e^{ix\,\xi -\xi^{2}/4}\text{d}\xi.
\end{equation}
With \eqref{formula} and \eqref{int.fourier}, we deduce that 
\begin{equation*}
\begin{aligned}
E(x,y,\a)&=
\frac1{4\pi^{3/2}} \e^{(x^{2}+y^{2})/2} \sum_{n\geq 0}\frac{\a^{n}}{2^{n}\,n\,!} 
\int_{\R}(i\,\xi)^{n}\,\e^{ix\xi -\xi^{2}/4}\text{d}\xi
\int_{\R}(i\,\eta)^{n}\,\e^{iy\eta -\eta^{2}/4}\text{d}\eta
\\
&=\frac1{4\pi^{3/2}} \e^{(x^{2}+y^{2})/2}\int_{\R^{2}} \sum_{n\geq 0}\frac1{n\,!} \big(-\frac{\a\,\xi\,\eta}{2}\big)^{n} \e^{i(x\xi+y\eta) -\xi^{2}/4 -\eta^{2}/4}\text{d}\xi\,\text{d}\eta
\\
&=\frac1{4\pi^{3/2}} \e^{(x^{2}+y^{2})/2}\int_{\R^{2}} \e^{-\a\,\xi\,\eta/2+ix\,(\xi+\eta) -\xi^{2}/4 -\eta^{2}/4}\text{d}\xi\,\text{d}\eta.
\end{aligned}
\end{equation*}
To compute the last integral, we make the change of variables $(\xi',\eta')=\frac1{\sqrt{2}}(\xi+\eta,\xi-\eta)$ and use \eqref{int.fourier0}. This completes the proof.
\end{proof}

\begin{prop}[Spatial decorrelation]\label{blabla}
There exists $C>0$ so that for all $x,y\in \R$, 
\begin{equation}\label{eq.decorr}
\big|\E\big[ \, \phi(x,\om)\,\ov{\phi(y,\om)} \,\big]\big|\leq C \e^{-\frac{(x-y)^{2}}4}.
\end{equation}
\end{prop}
\begin{proof}
Consider the function $F$ defined by 
\begin{equation*}
F(x,y,\a)=2\sum_{n\geq 0}\frac{\a^{2n+1}}{\lambda^{2}_{n}}\,h_{n}(x)\,h_{n}(y)=2\sum_{n\geq 0}\frac{\a^{2n+1}}{2n+1}\,h_{n}(x)\,h_{n}(y),
\end{equation*}
for $(x,y,\a)\in \R\times \R\times [0,1]$. Thanks to the bound \eqref{disp} we have 
\begin{equation*}
\big|\frac{\a^{2n+1}}{\lambda^{2}_{n}}\,h_{n}(x)\,h_{n}(y)\big|\leq C \frac{1}{\<n\>}\|h_{n}\|^{2}_{L^{\infty}(\R)}\leq C \frac{1}{\<n\>^{1+\frac16}},
\end{equation*}
hence $F\in \mathcal{C}\big(\R\times \R\times [0,1];\R\big)$. Therefore, 
\begin{equation}\label{eq.convv}
F(x,y,\a)\longrightarrow \sum_{n\geq 0}\frac{2}{\lambda^{2}_{n}}\,h_{n}(x)\,h_{n}(y)=\E\big[ \, \phi(x,\om)\,\ov{\phi(y,\om)} \,\big],
\end{equation}
when $\a \longrightarrow 1$.\\
Now observe that $F$ is smooth in $\a \in [0,1[$. Thus (as $F(x,y,0)=0$)
\begin{equation}\label{eq.intt}
F(x,y,\a)=\int_{0}^{\a}\partial_{\a}F(x,y,\beta)\,\text{d}\beta.
\end{equation}
By \eqref{for} we have 
\begin{eqnarray*}
\partial_{\a}F(x,y,\beta)&=&2 \sum_{n\geq 0}\beta^{2n}\,h_{n}(x)\,h_{n}(y)\nonumber\\
&=& \frac2{\sqrt{\pi(1-\beta^{4})}}\exp\big(-\frac{1-\beta^{2}}{1+\beta^{2}}\,\frac{(x+y)^{2}}{4}-\frac{1+\beta^{2}}{1-\beta^{2}}\,\frac{(x-y)^{2}}{4}\big).
\end{eqnarray*}
Hence there exists $C>0$ so that for all $x,y \in \R$ and $\beta \in [0,1[$ 
\begin{equation*}
\big|\partial_{\a}F(x,y,\beta)\big|\leq \frac{C}{\sqrt{1-\beta}} \e^{-\frac{(x-y)^{2}}4},
\end{equation*}
and this, together with \eqref{eq.convv} and \eqref{eq.intt} yields the estimate \eqref{eq.decorr}.
\end{proof}

\subsection{Bilinear estimates} In this section we give a proof of \eqref{eq.carre}. Observe that \eqref{eq.carre}, applied with $t=0$ implies that $\phi^2(\om,x)$ is a.s. in $\H^\theta$ for every $\theta<1/2$
which is a remarkable smoothing property satisfied by the random series  $\phi(\om,x)$. The key point in the proof of  \eqref{eq.carre} is the following bilinear estimate 
for Hermite functions.
\begin{lemm}\label{patrick}
There exists $C>0$ so that for all  $0\leq \theta\leq 1$ and $n,m\in \N$
\begin{equation}\label{eq.bilin}
\|h_{n}\,h_{m}\|_{\H^{\theta}(\R)}\leq C\max{(n,m)}^{-\frac14+\frac{\theta}2}\big(\log {\big(\min{(n,m)}+1\big)}\big)^{\frac12}.
\end{equation}
\end{lemm}
\begin{proof}
We give an argument we learned from Patrick G\'erard.
It suffices to prove \eqref{eq.bilin} for $\theta=0$ and $\theta=1$ (the general case then follows by interpolation).
The case $\theta=1$ can be directly reduced to the case $\theta=0$ thanks to \eqref{derive_hermite}.
Let us now give the proof of \eqref{eq.bilin} in the case $\theta=0$.
Consider again  the function $E$ defined by \eqref{funct_E} which can also be expressed by \eqref{for}.
Let $0\leq \alpha,\beta<1$ and $x\in \R$. By \eqref{for} we have 
\begin{equation*}
 E(x,x,\alpha)=\frac1{\sqrt{\pi}}(1-\alpha^{2})^{-\frac12}\e^{-\frac{1-\alpha}{1+\alpha}x^{2}}.
\end{equation*}
Therefore, if we set
$$
I(\alpha,\beta)\equiv\int_{\R}E(x,x,\alpha)E(x,x,\beta)\text{d}x, 
$$
then we get
\begin{multline}\label{prod.herm}
I(\alpha,\beta)=\frac1{\pi}(1-\alpha^{2})^{-\frac12}(1-\beta^{2})^{-\frac12}\int_{\R}\e^{-\big(\frac{1-\alpha}{1+\alpha}+\frac{1-\beta}{1+\beta}\big)x^{2}}\text{d}x
\\
=\frac1{\sqrt{2\pi}}(1-\alpha)^{-\frac12}(1-\beta)^{-\frac12}(1-\alpha\beta)^{-\frac12}.
\end{multline}
On the other hand, coming back to the definition 
\begin{equation*}
\dis I(\alpha,\beta)=\sum_{n,m\geq 0}\alpha^{n}\beta^{m}\int_{\R}h_{n}^{2}(x)\,h_{m}^{2}(x)\text{d}x.
\end{equation*}
Hence to get a useful expression for the $L^2$ norm of the product of two Hermite functions, 
it suffices to expand \eqref{prod.herm} in entire series in $\alpha$ and $\beta$.
Write 
\begin{equation*}
(1-x)^{-\frac12}=\sum_{p\geq 0}c_{p}x^{p}, \quad\dis c_{0}=1,\quad \quad c_{p}=\frac{(2p-1)\,!}{2^{2p-1}\,p\,!\,(p-1)\,!},\quad p\geq 1.
\end{equation*}
Therefore, by the Stirling formula, there exists $C>0$ so that $\dis |c_{p}|\leq \frac C{\sqrt{p+1}}$ for all $p\geq 0$. Now  by \eqref{prod.herm} and the previous estimate
\begin{multline*}
\int_{\R}h_{n}^{2}(x)\,h_{m}^{2}(x)\text{d}x
=\frac1{\sqrt{2\pi}}\sum_{\substack{p,q,r\geq 0\\p+r=n,\;q+r=m}} c_{p}\,c_{q}\,c_{r}
\\
\leq C\sum_{0\leq r\leq \min( n,m)}(n-r+1)^{-\frac12}\,(m-r+1)^{-\frac12}\,(r+1)^{-\frac12}.
\end{multline*} 
Without restricting the generality we may suppose that $m\geq n$.
If $m\leq 2n$ then we obtain the needed bound by considering 
separately the cases when the sum runs over $r<m/2$ and $r\geq m/2$.
If  $m> 2n$, then we can write $(m-r+1)^{-\frac12}\leq c(1+m)^{-\frac12}$ and the needed bound follows directly.
Therefore we get \eqref{eq.bilin} in the case $\theta=0$.
This completes the proof of Lemma~\ref{patrick}.
\end{proof}
\noindent 
Denote by $u(\om,t,x)$ the free Schr\"odinger solution with initial condition $\phi(\om,x)$, i.e.
\begin{equation*}
u(\om,t,x)=\e^{-itH}\phi(x,\om)=\sum_{n\geq 0}\frac{\sqrt{2}}{\lambda_{n}}\e^{-it\lambda^{2}_{n}}\,g_{n}(\om)\,h_{n}(x).
\end{equation*}
Write the decomposition
$ u=u_0+\sum_{N}u_{N}$, where the summation is taken over the dyadic integers and for $N$ a dyadic integer 
\begin{equation*}
u_{N}(\om,t,x)=\sum_{N\leq n< 2N}\a_{n}(t)h_{n}(x)g_{n}(\om),\quad \a_{n}(t)=\sqrt{\frac{2}{2n+1}}\e^{-i(2n+1)t}\, .
\end{equation*}
Let us fix $t\in\R$ and $0\leq \theta<\frac12$. It suffices to show that the expression
$$
J(t,x,\omega)\equiv |\sum_{M}\sum_{N}H^{\theta/2}\big(u_{N}\,u_{M})|
$$
belongs to $L^2(\R\times\Omega)$ (here the summation is again taken over the dyadic values of $M,N$).
Using the Cauchy-Schwarz inequality, a symmetry argument and summing geometric series, for all $\eps>0$ we can write
\begin{equation}\label{eq.dya}
J(t,x,\omega)
\leq C\big(\sum_{N\leq M}M^{\eps}|H^{\theta/2}\big(u_{N}\,u_{M}\big)|^{2} \big)^{\frac12}\, .
\end{equation}
Coming back to the definition we can write
\begin{equation*}
H^{\theta/2}\big(u_{N}\,u_{M}\big)=\sum_{\substack{N\leq n\leq 2N\\[2pt]M\leq m\leq 2M}}\a_{n}\,\a_{m}\,g_{n}\,g_{m}\,H^{\theta/2}\big(h_{n}\,h_{m}\big).
\end{equation*}
We now estimate $\|H^{\theta/2}\big(u_{N}\,u_{M}\big)\|_{L^{2}(\Omega)}$. We make the expansion
\begin{multline*}
|H^{\theta/2}
\big(u_{N}\,u_{M}\big)|^{2}=\\
\sum_{\substack{N\leq n_{1},n_{2}\leq 2N\\[2pt]M\leq m_{1},m_{2}\leq 2M}}\a_{n_{1}}\,
\ov{\a}_{n_{2}}\,\a_{m_{1}}\,\ov{\a_{m_{2}}}\,g_{n_{1}}\,\ov{g_{n_{2}}}\,g_{m_{1}}\,\ov{g_{m_{2}}}\,H^{\theta/2}\big(h_{n_{1}}\,h_{m_{1}}\big)\,
\ov{H^{\theta/2}\big(h_{n_{2}}\,h_{m_{2}}\big)}.
\end{multline*}
The random variables $g_{n}$ are centered and independent, and consequently, we have
%\begin{equation*}
$\E\big[  \, g_{n_{1}}\,\ov{g_{n_{2}}}\,g_{m_{1}}\,\ov{g_{m_{2}}}  \,  \big]=0,$
%\end{equation*}
unless the indexes are pairwise equal (i.e. ($n_{1}=n_{2}$ and $m_{1}=m_{2}$), or ($n_{1}=m_{2}$ and $n_{2}=m_{1}$).
%or ($n_{1}=m_{1}$ and $n_{2}=m_{2}$).
 This implies that 
\begin{equation}\label{eq.esp}
\int_{\Omega} |H^{\theta/2}\big(u_{N}\,u_{M}\big)|^{2} \leq 
C\sum_{\substack{N\leq n\leq 2N\\[2pt]M\leq m\leq 2M}}|\a_{n}|^{2}|\a_{m}|^{2}|H^{\theta/2}\big(h_{n}\,h_{m}\big)|^{2}.
\end{equation}
We integrate \eqref{eq.esp} in $x$ and by  \eqref{eq.bilin} we deduce that for all $\eps>0$ 
\begin{eqnarray*}
\int_{\Omega\times\R}
|H^{\theta/2}(u_{N}\,u_{M})|^{2}  
&\leq& C\sum_{\substack{N\leq n\leq 2N\\[2pt]M\leq m\leq 2M}}|\a_{n}|^{2}|\a_{m}|^{2}\int_{\R}|
H^{\theta/2}\big(h_{n}\,h_{m}\big)|^{2}\text{d}x\nonumber\\
&\leq& C\sum_{\substack{N\leq n\leq 2N\\[2pt]M\leq m\leq 2M}}(\max{(M,N)})^{-\frac12+\theta+\eps}|\a_{n}|^{2}|\a_{m}|^{2}.
\end{eqnarray*}
Therefore using that $|\a_{n}|\leq \<n\>^{-\frac12}$, we get
\begin{eqnarray*}
\int_{\Omega\times\R}(J(t,x,\om))^2
& \leq &
 C \sum_{N\leq M}
\sum_{\substack{N\leq n\leq 2N\\[2pt]M\leq m\leq 2M}}
M^{-\frac12+\theta+2\eps}|\a_{n}|^{2}|\a_{m}|^{2}
\\
& \leq  &
C \sum_{N\leq M}
\sum_{\substack{N\leq n\leq 2N\\[2pt]M\leq m\leq 2M}}
M^{-\frac12+\theta+2\eps}(MN)^{-1}<\infty,
\end{eqnarray*}
provided $\eps$ is small enough, namely $\eps$ such that $-\frac12+\theta+2\eps<0$.
This completes the proof of \eqref{eq.carre}.


\begin{thebibliography}{99}
\bibitem{AyTz}
A.~Ayache, N.~Tzvetkov.
\newblock $L^{p}$ properties of Gaussian random series. 
\newblock {\em Trans. Amer. Math. Soc.} 360 (2008), no. 8, 4425--4439. 
%
\bibitem{Bouclet}
J.-M.~Bouclet.
\newblock Distributions spectrales pour des op\'erateurs perturb\'es,
\newblock PhD Thesis, Nantes University 2000.
%
\bibitem{Bourgain1}
J.~Bourgain.
\newblock Periodic nonlinear Schr\"odinger equation and invariant measures
\newblock{\em Comm. Math. Phys.}, 166 (1994) 1--26.

\bibitem{Bourgain2}
J.~Bourgain.
\newblock Invariant measures for the 2D-defocusing nonlinear Schr\"odinger equation
\newblock{\em Comm. Math. Phys.}, 176 (1996) 421--445.

\bibitem{BGT} 
N. Burq, P, G\'erard, N. Tzvetkov.
\newblock Multilinear eigenfunction estimates and global existence for the three dimensional 
nonlinear Schr\"odinger equations. 
\newblock {\em Ann. Sci. Ecole Norm. Sup.} (4) 38 (2005), no. 2, 255--301.
%
\bibitem{BT1}
N.~Burq, N.~Tzvetkov.
\newblock Invariant measure for the three dimensional nonlinear wave equation.
\newblock{\em Int. Math. Res. Not. IMRN} 2007, no. 22, Art. ID rnm108, 26 pp. 

\bibitem{BT2}
N.~Burq, N.~Tzvetkov.
\newblock Random data Cauchy theory for supercritical wave equations \nolinebreak[4 ] I: local
existence theory.
\newblock{\em Invent. Math.} 173, No. 3, (2008), 449--475.
%
\bibitem{BT3}
N.~Burq, N.~Tzvetkov.
\newblock Random data Cauchy theory for supercritical wave equations \nolinebreak[4 ] II: A global existence result.
\newblock{\em Invent. Math.} 173, No. 3 (2008), 477--496. 
%
\bibitem{CCT} M.~Christ, J.~Colliander, T.~Tao, 
\newblock Ill-posedness for nonlinear Schr\"odinger and wave equations.
\newblock {\em Annales IHP}, to appear.
%
\bibitem{Carles}
R.~Carles.
\newblock Global existence results for nonlinear Schr\"odinger equations with quadratic potentials. 
\newblock{\em Discrete Contin. Dyn. Syst.} 13 (2005), no. 2, 385--398. 
%
\bibitem{Carles2}
R.~Carles.
\newblock Rotating points for the conformal NLS scattering operator. 
\newblock {\em Dynamics of PDE } 6 (2009), 35-51. 
%
\bibitem{CO1}
J. Colliander, T. Oh.
\newblock Almost sure local well-posedness of the periodic cubic nonlinear Schr\"odinger equation below L2,
\newblock{ \em preprint}

\bibitem{CO2}
 J. Colliander, T. Oh.
 \newblock  Almost sure global solutions of the periodic cubic nonlinear Schr\"odinger equation below L2, 
 \newblock{ \em preprint}
 \bibitem{DzGl09}J. Dziuba\'nski and P. G{\l}owacki.
 \newblock{Sobolev spaces related to Schr\"odinger operators with polynomial potentials,}
\newblock{\em Math. Z.}  262 (2009), no. 4, 881--894. 

\bibitem{Reika}
R.~Fukuizumi.
\newblock Stability and instability of standing waves for the nonlinear 
Schr\"odinger equation with harmonic potential.
\newblock{\em Discrete Contin. Dyn. Syst.} 7 (2001), no. 3, 525--544. 
%
%
\bibitem{GV} J. Ginibre, G. Velo.
\newblock The global Cauchy problem for the nonlinear Schr\"odinger equation. 
\newblock {\em Annales IHP, analyse non lin\'eaire, } 2 (1985), 309-327.
%
\bibitem{Hormander}
L.~H\"ormander.
\newblock The analysis of linear partial differential operators III.
\newblock Springer Verlag 1985.
%
\bibitem{KTV} Killip, T.~Tao, M.~Visan, 
\newblock The cubic nonlinear Schr\"odinger equation in two dimensions with radial data.
\newblock {\em J. Europ. Math. Soc.} to appear.
%
\bibitem{KochTataru}
H.~Koch, D.~Tataru.
\newblock $L^{p}$ eigenfunction bounds for the Hermite operator
\newblock {\em Duke Math. J.} 128 (2005), no. 2, 369--392. 

\bibitem{LeRoSp}
J.~Lebowitz, R. Rose, E. Speer.
\newblock Statistical dynamics of the nonlinear Schr\"odinger equation.
\newblock {\em J. Stat. Physics}, V 50 (1988) 657-687.
%
\bibitem{Kenji} K.~Nakanishi.
\newblock Energy scattering for nonlinear Klein-Gordon and Schr\"odinger equations in spatial dimensions
$1$ and $2$.
\newblock {\em J. Funct. Anal.} 169 (1999) 201--225. 
%
\bibitem{Oh1}
T.~Oh.
\newblock Invariance of the Gibbs measure for the Schr\"odinger-Benjamin-Ono system.
\newblock {\em to appear in SIAM
J. Math. Anal.}

\bibitem{Oh2}
T.~Oh.
\newblock Invariant Gibbs measures and a.s. global well-posedness for coupled KdV systems.
\newblock To appear in {\em Diff. Int. Eq. }

\bibitem{Robert}
D. Robert.
\newblock Autour de l'approximation semi-classique.
\newblock Progress in mathematics, Birkha\"user 1987.

\bibitem{Tao}
T.~Tao.
\newblock Nonlinear Dispersive Equations: Local and Global Analysis.
\newblock {\em CBMS 106. Providence. RI :} American Mathematical Society, 2006.
%
\bibitem{Taylor}
M. E.~Taylor.
\newblock Tools for PDE. Pseudodifferential operators, paradifferential operators, and layer potentials.
\newblock {\em Mathematical Surveys and Monographs 81.} American Mathematical Society, providence, RI, 2000.
%
\bibitem{Thomann5}
L.~Thomann.
\newblock Random data Cauchy problem for supercritical Schr\"odinger equations.
\newblock  {\em Ann. I. H. Poincar\'e - AN}, 26 (2009), no. 6, 2385--2402.
%
\bibitem{Thomann6}
L.~Thomann.
\newblock A remark on the Schr\"odinger smoothing effect.
\newblock To appear in {\em Asymptotic Analysis.}
%
\bibitem{Tzvetkov3}
N.~Tzvetkov.
\newblock Construction of a Gibbs measure associated to the periodic Benjamin-Ono equation.
\newblock {\em Probab. Theory Related Fields}, 146 (2010), 481--514.
%
\bibitem{Tzvetkov2}
N.~Tzvetkov.
\newblock Invariant measures for the defocusing NLS.
\newblock {\em Ann. Inst. Fourier}, 58 (2008) 2543--2604.
%
\bibitem{Tzvetkov1}
N.~Tzvetkov.
\newblock Invariant measures for the Nonlinear Schr\"odinger equation on the disc.
\newblock {\em Dynamics of PDE } 3 (2006), 111--160. 
%
\bibitem{YajimaZhang1}
K.~Yajima, and G.~Zhang.
\newblock Smoothing property for Schr\"odinger equations with potential superquadratic at infinity.
\newblock {\em Comm. Math. Phys.} 221 (2001), no. 3, 573--590.
%
\bibitem{YajimaZhang2}
K.~Yajima, and G.~Zhang.
\newblock Local smoothing property and Strichartz inequality for Schr\"odinger equations 
with potentials superquadratic at infinity.
\newblock {\em J. Differential Equations} (2004), no. 1, 81--110.
%
\bibitem{Zhidkov}
P. Zhidkov.
\newblock KdV and nonlinear Schr\"odinger equations : Qualitative theory.
\newblock {\em} Lecture Notes in Mathematics 1756, Springer 2001.
\end{thebibliography}
\end{document}